\documentclass[final]{siamltex}

\usepackage{amsmath}
\usepackage{amsfonts}
\usepackage{amssymb}
\usepackage{subeqnarray}
\usepackage{subfigure}
\usepackage{epsfig}
\usepackage{multirow}

\def\ol{\overline}

\def\ie{\textit{i.e.}}

\title{Stationary discrete shock profiles for scalar conservation laws with a discontinuous Galerkin method}

\author{Florent Renac\thanks{ONERA The French Aerospace Lab, 92320 Ch\^atillon Cedex, France ({\tt florent.renac@onera.fr}).}}

\begin{document}

\maketitle

\begin{abstract}
We present an analysis of stationary discrete shock profiles for a discontinuous Galerkin method approximating scalar nonlinear hyperbolic conservation laws with a convex flux. Using the Godunov method for the numerical flux, we characterize the steady state solutions for arbitrary approximation orders and show that they are oscillatory only in one mesh cell and are parametrized by the shock strength and its relative position in the cell. In the particular case of the inviscid Burgers equation, we derive analytical solutions of the numerical scheme and predict their oscillations up to fourth-order of accuracy. Moreover, a linear stability analysis shows that these profiles may become unstable at points where the Godunov flux is not differentiable. Theoretical and numerical investigations show that these results can be extended to other numerical fluxes. In particular, shock profiles are found to vanish exponentially fast from the shock position for some class of monotone numerical fluxes and the oscillatory and unstable characters of their solutions present strong similarities with that of the Godunov method.

\end{abstract}



\pagestyle{myheadings}
\thispagestyle{plain}
\markboth{F. RENAC}{DG SHOCK PROFILES OF SCALAR CONSERVATION LAWS}

%
%
\section{Introduction}

Discontinuous Galerkin (DG) methods are high-order finite element discretizations which were introduced in the early 1970s for the numerical simulation of the first-order hyperbolic neutron transport equation \cite{lesaint-raviart74,reed-hill73}. In recent years, these methods have become very popular for the solution of nonlinear convection dominated flow problems \cite{cockburn-shu89,cockburn-shu01,adigma-book10}. These methods allow high-order of accuracy and locality, which make them well suited to parallel computing, $hp$-refinement, $hp$-multigrid, unstructured meshes, application of boundary conditions, etc. 

However, the DG method suffers from spurious oscillations in the vicinity of discontinuities that develop in solutions of hyperbolic systems of conservation laws. These oscillations are due to the Gibbs phenomenon \cite{gottlieb_shu97} and may cause the solution to become locally nonphysical leading to robustness issues for the computation. Quadrature rules are usually used to compute integrals in the discretization of the equations. Properties of the DG method therefore depend on the local structure of the numerical solution at faces and within elements of the mesh. The control of such oscillations at a reasonable cost while keeping accuracy, robustness and stability is essential for the efficiency of the DG method and remains a challenge. Strategies have been proposed such as limiters \cite{cockburn-shu89,cockburn-shu01}, non-oscillatory reconstructions \cite{abgrall_94,qiu-shu_05}, $hp$-adaptation \cite{hartmann-houston02b}, shock capturing techniques \cite{persson-peraire06,guermond-etal_11}, etc. The latter methods aim at adding artificial viscosity to spread the structure of the discontinuity so that it can be resolved at the discrete level. For a DG method with polynomials of degree $p>0$ and cells of size $h$, the resolution scale is $h/p$ thus meaning that the method should resolve the discontinuity inside the element \cite{persson-peraire06}. The behavior of the numerical solution inside the discretization elements is therefore important for understanding the convergence of the DG computations.

In this work, we focus on the DG method for scalar nonlinear conservation laws which present stationary discontinuous solutions. More precisely, we are interested in the behavior of discrete profiles near shocks. Jennings \cite{jennings74} studied the approximation of scalar equations by monotone conservative finite difference schemes and proved the existence and stability of traveling discrete shocks. Discrete profiles for scalar equations were also studied in \cite{jiang-yu98,liu_et-al_99,osher-ralston82}. The analysis was then extended to systems in \cite{majda_raltson79,michelson84}. Bultelle et al. \cite{bultelle_et-al98} analyzed the linear stability of steady shock profiles obtained for general systems of conservation laws discretized with the Godunov method and constructed unstable profiles in the case of the Euler equations for gas dynamics. Recently, Lerat \cite{lerat_13} found exact discrete shock solutions for residual-based compact schemes \cite{lerat_corre01} up to seventh-order of accuracy discretizing the inviscid Burgers equation. Solutions were derived explicitly and were parametrized by the relative position of the shock in the discretization cell. Such analysis may help to tune parameters of the numerical method. In the context of DG methods, Cockburn and Guzm\'an \cite{cockburn_guzman08} considered a formally second order approximation of a scalar linear hyperbolic equation with discontinuous initial condition. They gave estimates of the size of extent of oscillations upstream and downstream of the discontinuity based on suitable weights introduced in \cite{johnson_etal84}. Recently, their work was extended to arbitrary approximation order in space and a third-order Runge-Kutta method on non-uniform meshes \cite{zhang-shu10,zhang-shu14}. The case of $\delta$-singularities in initial condition and source term of a scalar hyperbolic equation was then investigated in \cite{yang-shu13} where superconvergence in negative-order norms outside the pollution region was proved and sharp estimates over the whole domain were given.

The objective of this work is the theoretical analysis of stationary discrete shock solutions for the DG method discretizing a scalar conservation law with a general convex flux. We will mainly consider the Godunov method to evaluate the numerical flux for which explicit solutions can be derived, but we will also focus on other numerical fluxes widely used in the context of DG methods. To this end, we first give general results about the structure of the stationary discrete profiles. These results are local in the sense that they consider profiles that are small perturbations to the exact solution. In the case of the Godunov flux, we establish the mean value of the solution in the cell containing the exact shock position and the complete solution in other cells for an arbitrary approximation order. The analysis also predicts exponential decay of oscillations of the discrete profile on both sides of the shock for a certain class of monotone numerical fluxes. Then, the linear stability of these profiles is investigated and eigenvalues are characterized for the Godunov flux. As a result, the shock profiles may become unstable at points where the numerical flux is not differentiable. These points contain the situation of a shock at interface which is counter-intuitive since the exact solution is a piecewise constant function over all cells and is included in the function space of the DG method.

Considering the inviscid Burgers equation we derive exact discrete shock profiles for the DG scheme up to fourth-order of accuracy in the spirit of the work of Lerat \cite{lerat_13}. The results allow to evaluate quantitatively the structure of the solution within elements and predict situations where the numerical solution violates locally an entropy condition at the cell level. The theoretical stability analysis predicts the occurrence of unstable profiles when the exact shock position is close enough to an interface of the mesh. Numerical experiments will suggest that many features of the DG method obtained with the Godunov flux still hold for other numerical fluxes. In particular, for a given space discretization, the oscillating and unstable characters of the solution depend on the strength and exact position of the shock. As this latter feature is generally unknown, it strongly complicates the analysis for enhancing stability and robustness of the DG method. These results may support approaches based on {\it a posteriori} limitation techniques such as the MOOD method \cite{clain_etal11,berthon_desveaux14}. Artificial viscosity represent another attractive approach providing that the amount of viscosity adapts itself to the regularity of the solution. Theoretical results in this work will suggest the addition of artificial viscosity to the highest modes in the DG space spanned by hierarchical basis functions. As an illustration, we apply the spectral vanishing viscosity method \cite{maday_tadmor89} with a selective filter in the Legendre basis as proposed in  \cite{maday_etal93}.


The paper is organized as follows. Section \ref{sec:numer_approach} presents the model problem and the numerical approach for the discretization. In section \ref{sec:steady_shock_sol_conv}, we consider a general convex flux and analyze discrete shock profiles for the DG method. These profiles are explicitly constructed in the case of the inviscid Burgers equation in section \ref{sec:steady_shock_sol_burgers}. A linear stability analysis is performed in the vicinity of these solutions in section~\ref{sec:ana_stab}. These results are assessed by several numerical experiments in section \ref{sec:num_xp} and a first attempt for stabilizing the numerical scheme is proposed in section~\ref{sec:SVM}. Finally, concluding remarks about this work are given in section \ref{sec:conclusions}.

\section{Model problem and discretization}\label{sec:numer_approach}

\subsection{Nonlinear scalar equation}\label{sec:model_eqn}

The discussion in this paper focuses on the discretization of scalar nonlinear hyperbolic equations in one space dimension with a DG method. Let $\Omega=\mathbb{R}$ be the space domain and consider the following problem
\begin{subeqnarray}\label{eq:model_pb}
 \partial_tu + \partial_xf(u) &=& 0, \hspace{1.2cm} \mbox{in }\Omega\times(0,\infty),\\
  u(x,0) &=& u_0(x),\hspace{0.55cm} \mbox{in } \Omega. 
\end{subeqnarray}

The physical flux $f$ in ${\cal C}^2(\Omega_{a})$ is assumed to be coercive and strictly convex over the set of admissible states $\Omega_{a}\subset\mathbb{R}$: $\lim_{u\rightarrow\pm\infty}f(u)=+\infty$ and $f''(u)>0$ for all $u$ in $\Omega_a$. We are particularly interested in steady shock solutions to (\ref{eq:model_pb}). Such solutions consist in stationary discontinuities between two states 
\begin{equation}\label{eq:exact_sol}
 u(x):= \lim_{t\rightarrow\infty}u(x,t) = \left\{
 \begin{array}{rcl}
  u_L & \mbox{if} & x< x_c, \\
  u_R & \mbox{if} & x> x_c,
 \end{array}
 \right.
\end{equation}

\noindent where $x_c$ denotes the shock position and the states satisfy the Rankine-Hugoniot relation 
\begin{equation}\label{eq:RH_cond}
 f(u_L)=f(u_R)=f_\infty,
\end{equation}
\noindent and the Lax entropy condition 
\begin{equation}\label{eq:Lax_cond}
 f'(u_L)>0>f'(u_R).
\end{equation}

The above equation may be written in the equivalent form 
\begin{equation}\label{eq:Lax_cond_hat-u}
 u_L>\hat{u}>u_R,
\end{equation}

\noindent where $\hat{u}$ is the unique state such that $f'(\hat{u})=0$. Integrating (\ref{eq:model_pb}a) in space over $\Omega$, one obtains $d_t\int_\Omega udx=f(u_L)-f(u_R)=0$, which induces
\begin{equation}\label{eq:global_cons_edp}
 \int_\Omega u dx=\int_\Omega u_0 dx.
\end{equation}

\subsection{Discontinuous Galerkin formulation}\label{sec:DG_discr}

The DG method consists in defining a discrete weak formulation of problem (\ref{eq:model_pb}). The domain is discretized with a uniform grid $\Omega_h=\cup_{j\in\mathbb{Z}}\kappa_j$ with cells $\kappa_j = [x_{j-1/2}, x_{j+1/2}]$, $x_{j+1/2}=(j+\tfrac{1}{2})h$ and $h>0$ the space step (see Figure \ref{fig:stencil_1D}).

\subsubsection{Numerical solution and Legendre polynomials} We look for approximate solutions in the function space of discontinuous polynomials
\begin{equation}\label{eq:Vhp-space}
 {\cal V}_h^p=\{v_h\in L^2(\Omega_h):\;v_h|_{\kappa_{j}}\in{\cal P}_p(\kappa_{j}),\; \kappa_j\in\Omega_h\},
\end{equation}

\noindent where ${\cal P}_p(\kappa_{j})$ denotes the space of polynomials of degree at most $p$ in the element $\kappa_{j}$. The approximate solution to problem (\ref{eq:model_pb}) is sought under the form
\begin{equation}\label{eq:num_sol}
 u_h(x,t)=\sum_{l=0}^{p}\phi_j^l(x)U_j^{l}(t), \quad \forall x\in\kappa_{j},\, \kappa_j\in\Omega_h,\, t\geq0,
\end{equation}

\noindent where $U_j^l$ are the degrees of freedom (DOFs) in the element $\kappa_j$. The subset $(\phi_j^0,\dots,\phi_j^{p})$ constitutes a basis of ${\cal V}_h^p$ restricted onto a given element. In this work we will use the Legendre polynomials $L_{0\leq k\leq p}$. The basis functions in a given element $\kappa_j$ thus write $\phi_j^k(x)=L_k(2(x-x_j)/h)$ where $x_j=(x_{j+1/2}+x_{j-1/2})/2$ denotes the center of the element. Orthogonality of the Legendre polynomials induces
\begin{equation}\label{eq:ortho_basis}
 \int_{\kappa_j} \phi_j^k(x)\phi_j^l(x)dx = \frac{h}{2k+1}\delta_{k,l},\quad \forall \kappa_j\in\Omega_h,\, 0\leq k,l\leq p,
\end{equation}

\noindent where $\delta_{k,l}$ denotes the Kronecker symbol, and from (\ref{eq:num_sol}) we obtain the following expression for the mean value of the numerical solution
\begin{equation}\label{eq:mean_sol}
 \langle u_h\rangle_j(t):=\frac{1}{h}\int_{\kappa_j} u_h(x,t)dx = U_j^0(t), \quad \forall \kappa_j\in\Omega_h,\, t\geq0.
\end{equation}

Likewise, the properties $L_k(\pm1)=(\pm1)^k$ induce the following expressions for the left and right traces of the numerical solution at interfaces $x_{j\pm1/2}$ of a given element:
\begin{subeqnarray}\label{eq:LR_traces}
 u_{j+1/2}^-(t) &:=& u_h(x_{j+1/2}^-,t) = \sum_{l=0}^p U_j^l(t), \quad \forall t\geq0,\\
 u_{j-1/2}^+(t) &:=& u_h(x_{j-1/2}^+,t) = \sum_{l=0}^p (-1)^lU_j^l(t), \quad \forall t\geq0,
\end{subeqnarray}

\noindent where $v^\pm$ denote the right and left traces of the quantity $v$ at a given position (see Figure \ref{fig:stencil_1D}). Finally, the derivatives of the Legendre polynomials may be defined from the recurrence relation \cite{abramowitz_stegun}
\begin{equation}\label{eq:Leg_rec}
 d_sL_{k}=d_sL_{k-2}+(2k-1)L_{k-1}(s), \quad k>1, \quad L_0(s)=1, \quad L_1(s)=s. 
\end{equation}

\begin{figure}
\begin{center}
\epsfig{figure=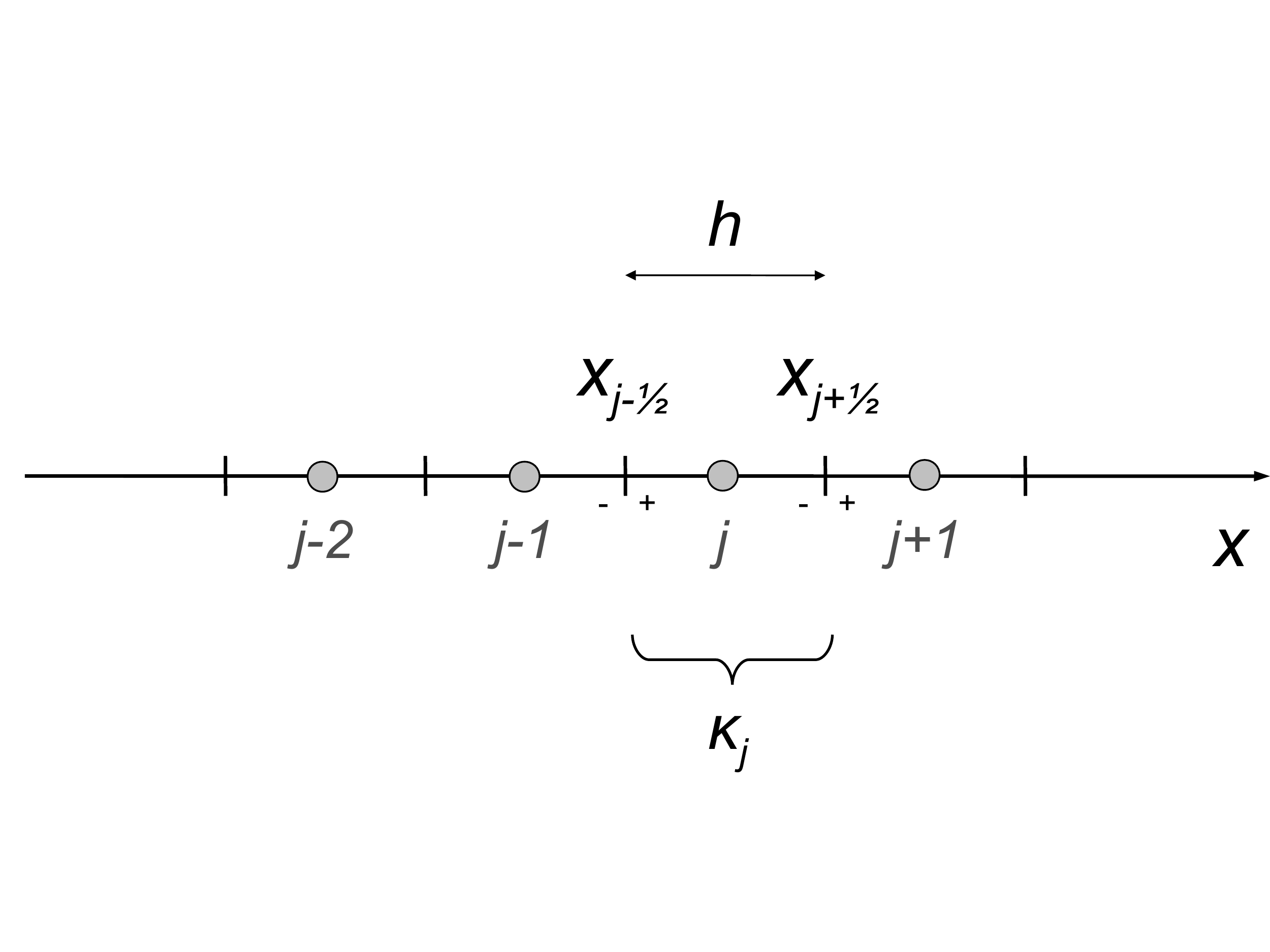,width=6cm,clip=true,trim=0cm 3.5cm 0cm 3.cm}
\caption{Mesh with definition of left and right traces at interfaces $x_{j\pm1/2}$.}
\label{fig:stencil_1D}
\end{center}
\end{figure}

Multiplying the above relation with $L_l(s)$, integrating over $[-1,1]$ and using orthogonality, one obtains the recursion
\begin{equation}\label{eq:matN_kl_rec}
 N_{k,l} = N_{k-2,l} + 2\delta_{k-1,l}, \quad N_{0,l}=0, \quad N_{1,l}=2\delta_{0,l}, \quad 1<k\leq p,\; 0\leq l\leq p,
\end{equation}

\noindent and hence $N_{k,l}=1-(-1)^{k+l}$ if $k>l\geq0$ and $N_{k,l}=0$ if $0\leq k\leq l$ for the entries of the matrix ${\bf N}$ defined by
\begin{equation}\label{eq:matN_kl}
 N_{k,l} = \int_{-1}^1L_l(s)d_sL_kds, \quad 0\leq k,l \leq p.
\end{equation}

Note that we know the explicit projection of the derivatives of the Legendre polynomials into the Legendre basis. Indeed, applying the recurrence relation (\ref{eq:Leg_rec}) successively one obtains
\begin{equation*}
 d_sL_{2k} = \sum_{l=1}^k(4l-1)L_{2l-1}(s), \quad d_sL_{2k+1} = \sum_{l=0}^k(4l+1)L_{2l}(s),
\end{equation*}
\noindent which may be rewritten under the general form
\begin{equation}\label{eq:matM_kl}
 d_sL_{k} = \sum_{l=0}^{k-1} M_{k,l}L_l(s),
\end{equation}
\noindent where the entries of the matrix ${\bf M}$ are defined by $M_{k,l}=(2l+1)\tfrac{1-(-1)^{k+l}}{2}$ for $k>l\geq0$ and $M_{k,l}=0$ for $0\leq k\leq l$.

\subsubsection{Space discretization} The semi-discrete form of the DG discretization in space of problem (\ref{eq:model_pb}) reads: find $u_h$ in ${\cal V}_h^p$ such that
\begin{equation}\label{eq:discr_var_form}
 \int_{\Omega_h} v_h \partial_tu_h dx - \sum_{\kappa_j\in\Omega_h}{\cal R}_j(u_h,v_h) = 0, \quad \forall v_h\in{\cal V}_h^p.
\end{equation}

\noindent together with the initial condition
\begin{equation}\label{eq:discr_CI}
 \int_{\Omega_h} v_h(x) u_h(x,0) dx = \int_{\Omega_h} v_h(x) u_0(x) dx, \quad \forall v_h\in{\cal V}_h^p.
\end{equation}

The discretization of the elementwise explicit residuals in (\ref{eq:discr_var_form}) reads
\begin{equation}\label{eq:local_residuals}
 {\cal R}_j(u_h,v_h) =  \int_{\kappa_j} f(u_h)\partial_xv_h dx - \oint_{\partial\kappa_j} \hat{h}(u_h^-,u_h^+)v_h dS,
\end{equation}

\noindent where $\hat{h}:\Omega_a\times\Omega_a\rightarrow\mathbb{R}$ denotes a numerical flux consistent with the physical flux: $\hat{h}(u,u)=f(u)$. We will consider the Godunov flux for the theoretical analysis of sections \ref{sec:steady_shock_sol_burgers} and \ref{sec:ana_stab_god}:
\begin{equation}\label{eq:godunov_flux}
 \hat{h}(u_h^-,u_h^+) = \left\{
\begin{array}{ll}
  \min\{f(v):\;v\in[u_h^-,u_h^+]\}, &\mbox{if } u_h^- \leq u_h^+,\\
  \max\{f(v):\,v\in[u_h^+,u_h^-]\}, &\mbox{if } u_h^- > u_h^+.
\end{array}
\right.
\end{equation}

The Godunov flux results from the solution of the Riemann problem for (\ref{eq:model_pb}) with piecewise constant initial data consisting in states $u_h^-$ and $u_h^+$ at the left and right of the interface. The values of (\ref{eq:godunov_flux}) in the set of states are depicted in Figure~\ref{fig:solution_PR_Godunov_burgers}(a). 

Comparisons will also be given in the numerical experiments for the local Lax-Friedrichs (LLF) flux:
\begin{equation}\label{eq:llf_flux}
 \hat{h}(u_h^-,u_h^+) = \frac{f(u_h^-)+f(u_h^+)}{2} + \frac{\alpha}{2}(u_h^--u_h^+),
\end{equation}
\noindent where the constant $\alpha$ is a stabilization parameter defined by $\alpha>\max\{|f'(v)|:\; \min(u_h^-,u_h^+)\leq v\leq\max(u_h^-,u_h^+)\}$, and the Engquist-Osher flux (see Figure~\ref{fig:solution_PR_Godunov_burgers}(b)):
\begin{equation}\label{eq:OSH_flux}
 \hat{h}(u_h^-,u_h^+) = \frac{f(u_h^-)+f(u_h^+)}{2} - \frac{1}{2}\int_{u_h^-}^{u_h^+}|f'(v)|dv,
\end{equation}

We recall that fluxes (\ref{eq:godunov_flux}) to (\ref{eq:OSH_flux}) are Lipschitz continuous and monotone 
\begin{equation}\label{eq:monotone_flux}
 \partial_{u^-}\hat{h}(a,b)\geq0, \quad \partial_{u^+}\hat{h}(a,b)\leq0, \quad \forall a,b \in \Omega_a,
\end{equation}

\noindent where $\partial_{u^-}\hat{h}$ and $\partial_{u^+}\hat{h}$ denote the partial derivatives of $\hat{h}$ with respect to its first and second arguments. Numerical experiments tend to show that the effect of the flux on the quality of the approximation decreases as the polynomial degree $p$ increases \cite{cockburn-shu01,qui_et-al06}. The analysis in section \ref{sec:steady_shock_sol_conv} and numerical experiments in section \ref{sec:num_xp} will support this trend and confirm the relevance of the theoretical analysis with the Godunov flux in sections~\ref{sec:steady_shock_sol_burgers} and \ref{sec:ana_stab}.

\begin{figure}
\begin{center}
\subfigure[]{\epsfig{figure=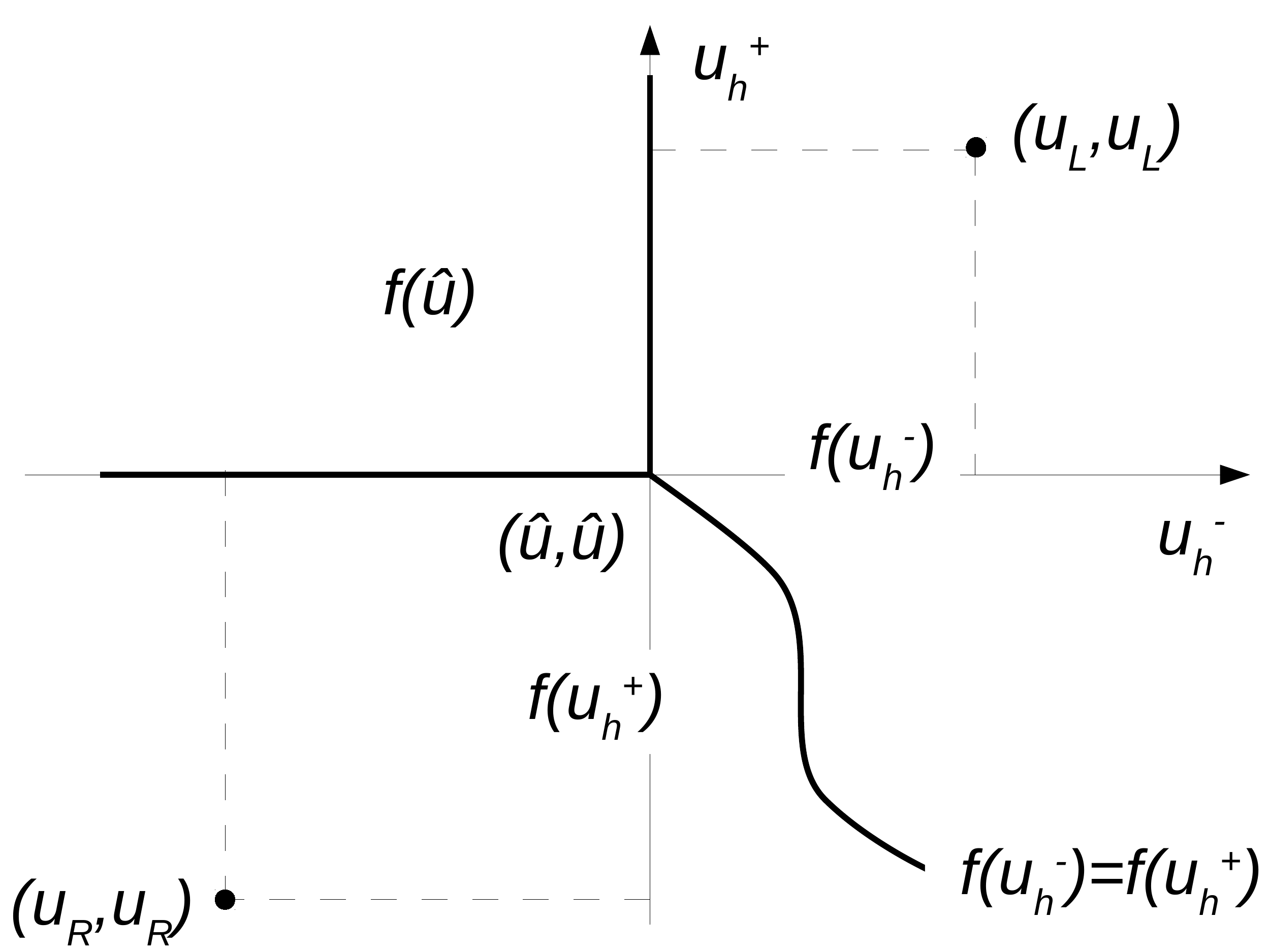,width=6cm}}
\subfigure[]{\epsfig{figure=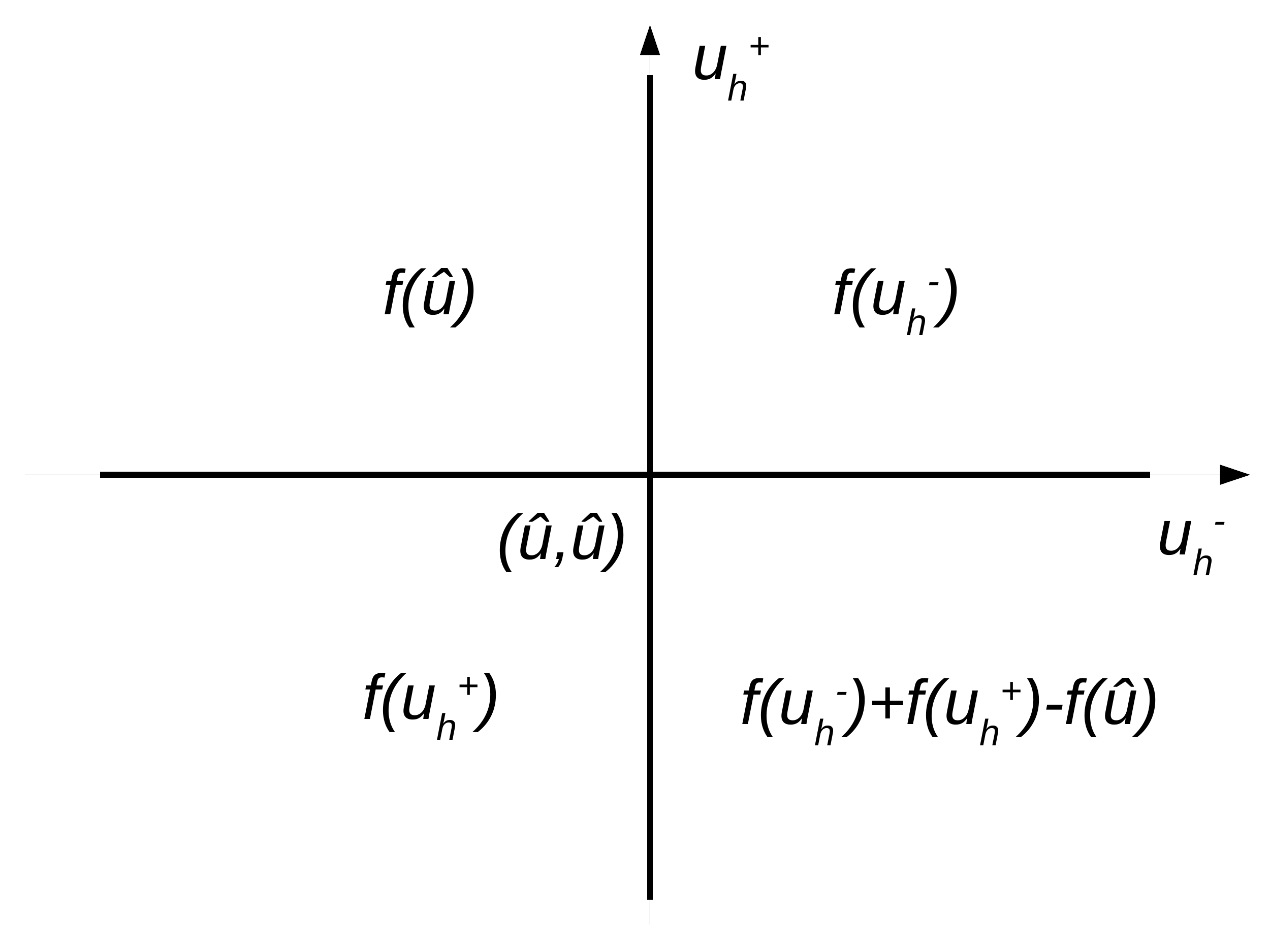,width=6cm}}
\caption{Values of (a) the Godunov flux (\ref{eq:godunov_flux}) and (b) the Engquist-Osher flux (\ref{eq:OSH_flux}) in the set of states for a convex scalar physical flux.}
\label{fig:solution_PR_Godunov_burgers}
\end{center}
\end{figure}

\subsection{Steady-state scheme}

Looking for steady-state solutions
\begin{equation}\label{eq:_steady_num_sol}
 u_h(x)=\sum_{l=0}^{p}\phi_j^l(x)U_j^{l}, \quad \forall x\in\kappa_j,\, \kappa_j\in\Omega_h,
\end{equation}

\noindent of the numerical scheme (\ref{eq:discr_var_form}) and projecting the discrete scheme into the function basis, one obtains for all $\kappa_j$ in $\Omega_h$ and $0\leq k\leq p$:
\begin{eqnarray}\label{eq:steady_discrete_scheme_def}
 {\cal R}_j^k(u_h) &:=& {\cal R}_j(u_h,\phi_j^k) \nonumber\\
 &=& \int_{\kappa_j} f(u_h)d_x\phi_j^k dx - \hat{h}\big(u_{j+\frac{1}{2}}^-,u_{j+\frac{1}{2}}^+\big) + (-1)^k\hat{h}\big(u_{j-\frac{1}{2}}^-,u_{j-\frac{1}{2}}^+\big).
\end{eqnarray}

The local conservation property
\begin{equation}\label{eq:steady_local_cons}
\hat{h}\big(u_{j+\frac{1}{2}}^-,u_{j+\frac{1}{2}}^+\big) = \hat{h}\big(u_{j-\frac{1}{2}}^-,u_{j-\frac{1}{2}}^+\big), \quad \forall\kappa_j\in\Omega_h,
\end{equation}

\noindent is obtained from ${\cal R}_j(u_h,1_{\kappa_j})=0$ where $1_{\kappa_j}$ denotes the indicator function of $\kappa_j$. We are interested in discrete solutions of the numerical scheme such that
\begin{equation}\label{eq:discrete_shock_at_infty}
 \lim_{x\rightarrow-\infty} u_h(x) = u_L,\quad  \lim_{x\rightarrow\infty} u_h(x) = u_R.
\end{equation}

Therefore, we obtain $\lim_{x\rightarrow \pm\infty}\hat{h}(u_h^-,u_h^+)=f_\infty$ and the elementwise residuals in cells $\kappa_j$ in $\Omega_h$ may be rewritten under the more convenient forms
\begin{subeqnarray}\label{eq:steady_discrete_scheme}
 {\cal R}_j^0(u_h) &=& \hat{h}\big(u_{j+\frac{1}{2}}^-,u_{j+\frac{1}{2}}^+\big) - f_\infty \\
                   &=& \hat{h}\big(u_{j-\frac{1}{2}}^-,u_{j-\frac{1}{2}}^+\big) - f_\infty = 0, \\
 {\cal R}_j^k(u_h) &=& \int_{\kappa_j} f(u_h)d_x\phi_j^k dx - \big(1-(-1)^k\big)f_\infty = 0,\; 1\leq k\leq p.
\end{subeqnarray}

%
%
\section{Steady discrete shock solutions for general convex fluxes}\label{sec:steady_shock_sol_conv}

In this section, we present some preliminary results obtained for a convex flux $f$ that will be used in the next sections. We assume that the exact shock position $x_c$ is known and we set $j_c=0$ the cell index containing the shock without loss of generality. Cells $j<0$ and $j>0$ will be referred as to supersonic and subsonic cells, respectively. Our objective is here to determine the structure of the numerical solution in cells upstream and downstream of the shock position and to provide information about the solution in the shock cell. 

The analysis will mainly focus on the use of the Godunov flux (\ref{eq:godunov_flux}). For a steady-state solution to hold, the conservation property $\hat{h}(u_h^-,u_h^+)=f_\infty$, with a coercive and convex flux, has one of the following two solutions
\begin{subeqnarray}\label{eq:sol_godunov_flux}
 u_{h}^-=u_L &\mbox{and}& u_{h}^+\geq u_R, \\
 u_{h}^-\leq u_L &\mbox{and}& u_{h}^+= u_R.
\end{subeqnarray}

Indeed, equation $f(u)=f_\infty$ has two roots $u=u_L$ or $u=u_R$ (see section~\ref{sec:model_eqn}). Hence, to solve $\hat{h}(u_h^-,u_h^+)=f_\infty$, it is suficient to consider the solutions of Riemann problems with initial conditions either $u_h^-$ such that $f(u_h^-)=f_\infty$ and $u_h^+$ in $\Omega_a$, or $u_h^-$ in $\Omega_a$ and $u_h^+$ such that $f(u_h^+)=f_\infty$. Consider the first initial condition, then $u_h^-=u_L$ or  $u_h^-=u_R$. Since $\hat{h}$ is continuous, we now consider successive situations when $u_h^+$ varies. Consider first $u_h^-=u_L$. The following situations (i) $u_L\leq u_h^+$, (ii) $u_R\leq u_h^+< u_L$, and (iii) $u_h^+<u_R$ lead to the following expressions of the Godunov flux (\ref{eq:godunov_flux}): (i) $\hat{h}(u_L,u_h^+)=f(u_L)$, (ii) $\hat{h}(u_L,u_h^+)=f(u_L)$, and (iii) $\hat{h}(u_L,u_h^+)=f(u_h^+)>f_\infty$. The other situation, $u_h^-=u_R$, leads to $\hat{h}(u_R,u_h^+)=\min\{f(\hat{u}),f(u_h^+)\}<f_\infty$ when $u_R<u_h^+$, where $\hat{u}$ is defined by (\ref{eq:Lax_cond_hat-u}), and $\hat{h}(u_R,u_h^+)=f(u_h^+)>f_\infty$ when $u_h^+<u_R$. We therefore obtain (\ref{eq:sol_godunov_flux}a) and a similar step applies to infer (\ref{eq:sol_godunov_flux}b).

In the following, we shall assume a stronger condition on the traces of the numerical solution at left and right faces of the shock cell:
\begin{subeqnarray}\label{eq:sign_condition_conv-flx}
 u_{-1/2}^-=u_L &\mbox{and}& u_{-1/2}^+ > u_R, \\
 u_{1/2}^- < u_L &\mbox{and}& u_{1/2}^+= u_R.
\end{subeqnarray}

Indeed, the Godunov flux is no longer differentiable at points $(u_h^-,u_h^+)$ with $u_h^-\geq\hat{u}$ such that $f(u_h^-)=f(u_h^+)$ (see Figure~\ref{fig:solution_PR_Godunov_burgers}(a)) and the strict inequalities in (\ref{eq:sign_condition_conv-flx}) ensure its differentiability in our analysis. We will give some comments on the implications of assumption (\ref{eq:sign_condition_conv-flx}) in section~\ref{sec:comments_trace_ineq}. We stress that the present results will show that the above assumption is not a restriction on the validity of our analysis.

As a consequence of (\ref{eq:sign_condition_conv-flx}), the Godunov flux (\ref{eq:godunov_flux}) reduces to a fully upwind flux at interfaces $x_{\pm1/2}$:
\begin{subeqnarray}\label{eq:godunov-flx_interf_jc}
 \hat{h}(u_{-1/2}^-,u_{-1/2}^+) &=& f(u_{-1/2}^-), \\
 \hat{h}(u_{1/2}^-,u_{1/2}^+) &=& f(u_{1/2}^+).
\end{subeqnarray}

From (\ref{eq:godunov-flx_interf_jc}a), the residuals in cell $\kappa_{-1}$ no longer depend on DOFs in cell $\kappa_0$. Hence, the DOFs in cells $j<0$ now become independent of DOFs in cells $j\geq0$. Likewise, (\ref{eq:godunov-flx_interf_jc}b) implies that DOFs in cells $j>0$ are independent of DOFs in cells $j\leq0$. This is a key result for the theoretical analysis below and follows from the choice of the Godunov flux.  


%
\subsection{Numerical solution in the subsonic and supersonic regions}\label{sec:steady_shock_sol_sup-sub}


The first result concerns existence of discrete shock profiles in cells $j<0$ in the supersonic region when using the Godunov numerical flux. This result is local in the sense that it is based on the inverse function theorem in each cell and is valid in a neighborhood of the exact solution (\ref{eq:exact_sol}). The same analysis is then applied in the subsonic region. In the case of more general monotone numerical fluxes, we also give conditions that allow an exponential decay of perturbations away from the shock position. 

\begin{theorem}[Solution in the supersonic region]
\label{th:sol_super_zone_conv_flx}
Assume that the shock is strictly contained in cell $j_c=0$. Let $p\geq0$, $h>0$ and assume that the discrete solution satisfies (\ref{eq:sign_condition_conv-flx}a). Then in each cell $j<0$, the constant discrete function in ${\cal P}_p(\kappa_j)$ defined by $w_h:\kappa_j\ni x\mapsto w_h(x)=u_L$ is the unique solution to the numerical scheme (\ref{eq:steady_discrete_scheme}) over a sufficiently small neighborhood ${\cal U}_j\subset\mathbb{R}^{p+1}$.
\end{theorem}

\begin{proof}
We first note that, by consistency, $w_h$ trivially satisfies the numerical scheme (\ref{eq:steady_discrete_scheme}): ${\cal R}_j^k(w_h)=0$, for all $0\leq k\leq p$ and $j<0$. In other words, $w_h$ is a stationary solution in the supersonic region. Let us prove that this is the unique solution in a sufficiently small neighborhood. Since $u_L>\hat{u}$ (see Figure~\ref{fig:solution_PR_Godunov_burgers}a), for each interface $j+\tfrac{1}{2}<-\tfrac{1}{2}$, there exist neighborhoods ${\cal V}_j\times{\cal V}_{j+1}\subset\mathbb{R}^{p+1}\times\mathbb{R}^{p+1}$ of the solution $w_h$ over ${\kappa_j\cup\kappa_{j+1}}$ where the numerical flux in (\ref{eq:steady_discrete_scheme}a) reduces to the upwind flux:
\begin{equation}\label{eq:proof_sup1}
 {\cal R}_j^0(u_h) = f(u_{j+\frac{1}{2}}^-) - f_\infty = 0.
\end{equation}

By assumption (\ref{eq:sign_condition_conv-flx}a), (\ref{eq:proof_sup1}) holds also at interface $x_{-1/2}$. Now let $j<0$, define ${\bf U}_j=(U_j^l)_{0\leq l\leq p}^\top$ and let ${\bf G}_{j+1/2}:\mathbb{R}^{p+1}\rightarrow\mathbb{R}^{p+1}$ be the application whose components are the elementwise residuals in (\ref{eq:proof_sup1}) and (\ref{eq:steady_discrete_scheme}c): 
\begin{equation}\label{eq:matG}
 G_{j+1/2}^k({\bf U}_j):={\cal R}_j^k(u_h)=0, \quad 0\leq k\leq p. 
\end{equation}

Setting ${\bf U}_L=(u_L,0,\dots,0)^\top$ in $\mathbb{R}^{p+1}$, we have ${\bf G}_{j+1/2}({\bf U}_L)=0$ by consistency and it is sufficient to prove that the mapping ${\bf G}_{j+1/2}$ is invertible over a neighborhood ${\cal W}_j\subset\mathbb{R}^{p+1}$ of ${\bf U}_L$. The partial derivatives of ${\bf G}_{j+1/2}$ take values

\begin{subeqnarray*}
 \partial_{U_j^l} G_{j+1/2}^0 &=& f'(u_{j+1/2}^-) = f'(u_L), \quad 0\leq l\leq p,\\
 \partial_{U_j^l} G_{j+1/2}^k &=&\int_{\kappa_j}f'(u_h)\phi_j^l d_x\phi_j^k dx = f'(u_L)N_{k,l}, \quad 0<k\leq p,\;0\leq l\leq p,
\end{subeqnarray*}


\noindent at the point ${\bf U}_L$, where the coefficients $N_{k,l}$ are defined by (\ref{eq:matN_kl}). The application ${\bf G}_{j+1/2}$ is continuously differentiable over $\mathbb{R}^{p+1}$ and its Jacobian $|\nabla_{{\bf U}_j}{\bf G}_{j+1/2}|$ has the following form at the point ${\bf U}_L$:
\begin{equation*}
\left|
 \begin{array}{ccccc}
  f'(u_L) & f'(u_L) & f'(u_L) & \dots & f'(u_L) \\
  2f'(u_L) & 0 & 0 & \dots & 0 \\
  0 & 2f'(u_L) & 0 & \dots & 0 \\
  2f'(u_L) & 0 & 2f'(u_L) & \dots & 0 \\
  \vdots & \vdots & \vdots & \ddots & \vdots \\
  (1-(-1)^p)f'(u_L) & (1+(-1)^p)f'(u_L) & (1-(-1)^p)f'(u_L) & \dots & 0
 \end{array}
 \right|
\end{equation*}

\noindent and reduces to $|\nabla_{{\bf U}_j}{\bf G}|_{u_L}=(-2)^p(f'(u_L))^{p+1}\neq0$ because $f'(u_L)>0$. Applying the inverse function theorem, there exists a neighborhood ${\cal W}_j\subset\mathbb{R}^{p+1}$ of ${\bf U}_L$ where the mapping ${\bf G}_{j+1/2}$ is invertible. We conclude with ${\cal U}_j={\cal V}_j\cap{\cal W}_j$
\end{proof}

A similar result holds for the subsonic region that we formulate below and whose proof uses the same method as for Theorem \ref{th:sol_super_zone_conv_flx}.

\begin{theorem}[Solution in the subsonic region]
\label{th:sol_sub_zone_conv_flx}
Assume that the shock is strictly contained in cell $j_c=0$. Let $p\geq0$, $h>0$ and assume that the discrete solution satisfies (\ref{eq:sign_condition_conv-flx}b). Then in each cell $j>0$, the constant discrete function in ${\cal P}_p(\kappa_j)$ defined by $w_h:\kappa_j\ni x\mapsto w_h(x)=u_R$ is the unique solution to the numerical scheme (\ref{eq:steady_discrete_scheme}) over a sufficiently small neighborhood ${\cal U}_j\subset\mathbb{R}^{p+1}$.
\end{theorem}

In the case of a general numerical flux, the upwind property (\ref{eq:godunov-flx_interf_jc}) is no longer valid with the consequence that the DOFs in supersonic and subsonic regions are no longer uncoupled from other DOFs and oscillations may appear in cells close to the shock position. However, we now show that, under some assumptions on the numerical flux, these oscillations decay exponentially fast from the shock. This trend will be illustrated in the numerical experiments of section \ref{sec:num_xp} with the LLF flux and supports the relevance of the choice of the Godunov flux for the theoretical analysis.

\begin{theorem}
\label{th:exp_decay}
 Suppose $\hat{h}:\Omega_a\times\Omega_a\rightarrow\mathbb{R}$ is a monotone (\ref{eq:monotone_flux}), Lipschitz continuous numerical flux consistent with the physical flux $f(u)$ and satisfies
\begin{equation}\label{eq:hyp_exp_decay}
 0<-\frac{\partial_{u^+}\hat{h}(u_L,u_L)}{\partial_{u^-}\hat{h}(u_L,u_L)}<1, \quad 0<-\frac{\partial_{u^-}\hat{h}(u_R,u_R)}{\partial_{u^+}\hat{h}(u_R,u_R)}<1.
\end{equation}

Assume that there exist $J^-<0$ and $J^+>0$ such that the discrete solution satisfies $u_h(x)=u_X+\sum_{l=0}^p\varsigma_{j}^l\phi_{j}^l(x)$ in cells $j\in\{J^-,J^+\}$ with $\sum_{l=0}^p|\varsigma_{j}^l|\leq\epsilon|u_X|$, $\epsilon\ll1$ and $u_X=u_L$ if $j<0$ or $u_X=u_R$ if $j>0$. Then, the amplitude of oscillations of the numerical solution around the exact solution decays exponentially fast as $|j|\rightarrow\infty$.
\end{theorem}

\begin{proof}
We first consider supersonic cells $j\leq J^-$ and proceed by induction. Let an interface $j+\tfrac{1}{2}<J^-+\tfrac{1}{2}$ and assume that the property holds true in cell $j+1$: $u_h(x)=u_L+\sum_{l=0}^p\varsigma_{j+1}^l\phi_{j+1}^l(x)$ for all $x\in\kappa_{j+1}$ with $\sum_{l=0}^p|\varsigma_{j+1}^l|\leq\epsilon|u_L|$. Now define ${\bf G}_{j+1/2}:\mathbb{R}^{p+1}\times\mathbb{R}^{p+1}\rightarrow\mathbb{R}^{p+1}$ the application whose components are the elementwise residuals in (\ref{eq:steady_discrete_scheme}a), or equivalently (\ref{eq:proof_sup1}), and (\ref{eq:steady_discrete_scheme}c): 
\begin{equation*}
 G_{j+1/2}^k({\bf U}_j,{\bf U}_{j+1}):={\cal R}_j^k(u_h)=0, \quad 0\leq k\leq p. 
\end{equation*}

Note that ${\bf G}_{j+1/2}$ also vanishes at $({\bf U}_L,{\bf U}_L)$ by consistency. By assumption of monotonicity of $\hat{h}$ over $\Omega_a\times\Omega_a$, there exists a neighborhood ${\cal V}_{j+1/2}\subset\mathbb{R}^{p+1}\times\mathbb{R}^{p+1}$ of $({\bf U}_L,{\bf U}_L)$ where ${\bf G}_{j+1/2}$ is continuously differentiable and has partial derivatives
\begin{subeqnarray*}
 \partial_{U_j^l} G_{j+1/2}^0 &=& \partial_{u^-}\hat{h}(u_{j+1/2}^-,u_{j+1/2}^+) = \partial_{u^-}\hat{h}(u_L,u_L), \quad 0\leq l\leq p,\\
 \partial_{U_{j+1}^l} G_{j+1/2}^0 &=& (-1)^l\partial_{u^+}\hat{h}(u_{j+1/2}^-,u_{j+1/2}^+) = (-1)^l\partial_{u^+}\hat{h}(u_L,u_L), \quad 0\leq l\leq p,\\
 \partial_{U_j^l} G_{j+1/2}^k &=&\int_{\kappa_j}f'(u_h)\phi_j^ld_x\phi_j^k dx = f'(u_L)N_{k,l}, \quad 0<k\leq p,\;0\leq l\leq p,\\
 \partial_{U_{j+1}^l} G_{j+1/2}^k &=& 0, \quad 0<k\leq p,\; 0\leq l\leq p,
\end{subeqnarray*}


\noindent at $({\bf U}_L,{\bf U}_L)$, where the coefficients $N_{k,l}$ are defined by (\ref{eq:matN_kl}). In the same way as in proof of Theorem~\ref{th:sol_super_zone_conv_flx}, the Jacobian becomes
%
\begin{equation*}
 |\nabla_{{\bf U}_j}{\bf G}_{j+1/2}|({\bf U}_L,{\bf U}_L)=(-2f'(u_L))^p \partial_{u^-}\hat{h}(u_L,u_L) \neq0 
\end{equation*}

\noindent from $f'(u_L)>0$, monotonicity of $\hat{h}$ and the first inequality in (\ref{eq:hyp_exp_decay}). Applying the implicit function theorem, there exists a neighborhood ${\cal U}_{j+1/2}\subset\mathbb{R}^{p+1}$ about ${\bf U}_L$ and a continuously differentiable application $\varphi:{\cal U}_{j+1/2}\ni{\bf U}_{j+1}\mapsto{\bf U}_{j}=\varphi({\bf U}_{j+1})\in\mathbb{R}^{p+1}$ such that ${\bf G}_{j+1/2}(\varphi({\bf U}_{j+1}),{\bf U}_{j+1})=0$. Setting $u_h|_{\kappa_{j}}\equiv u_L+\sum_{l=0}^p\varsigma_{j}^l\phi_{j}^l$, a Taylor development of $G_{j+1/2}^0$ about ${\bf U}_L$ reads
\begin{eqnarray*}
 G_{j+1/2}^0({\bf U}_j,{\bf U}_{j+1}) &=& G_{j+1/2}^0(u_L,u_L)+\partial_{u^-}\hat{h}(u_L,u_L)\sum_{l=0}^p\varsigma_{j}^l \nonumber\\
 && + \partial_{u^+}\hat{h}(u_L,u_L)\sum_{l=0}^p(-1)^l\varsigma_{j}^l+{\cal O}(\epsilon^2u_L^2)\nonumber\\
  &\simeq& \partial_{u^-}\hat{h}(u_L,u_L)(u_{j+1/2}^--u_L)  + \partial_{u^+}\hat{h}(u_L,u_L)(u_{j+1/2}^+-u_L).
\end{eqnarray*}

We thus obtain that
\begin{equation}\label{eq:sign_jump-a}
 u_{j+1/2}^--u_L \simeq -\frac{\partial_{u^+}\hat{h}(u_L,u_L)}{\partial_{u^-}\hat{h}(u_L,u_L)}(u_{j+1/2}^+-u_L).
\end{equation}

From (\ref{eq:hyp_exp_decay}), the left and right traces of the numerical solution have the same sign. Likewise, a Taylor development of the applications $G_{j+1/2}^k$ with $0<k\leq p$ about ${\bf U}_L$ reads
\begin{equation*}
 G_{j+1/2}^k({\bf U}_j,{\bf U}_{j+1}) = 0 = -f'(u_L)\sum_{l=0}^pN_{k,l}\varsigma_j^l+{\cal O}(\epsilon^2u_L^2), \quad 0< k\leq p,
\end{equation*}

\noindent which induces $\varsigma_j^l={\cal O}(\epsilon^2u_L^2)$ for $0\leq l<p$. The last coefficient is defined through the relation (\ref{eq:sign_jump-a}) and reads
\begin{equation*}
 \varsigma_j^p \simeq -\frac{\partial_{u^+}\hat{h}(u_L,u_L)}{\partial_{u^-}\hat{h}(u_L,u_L)}(u_{j+1/2}^+-u_L).
\end{equation*}

From (\ref{eq:hyp_exp_decay}), we have $|\partial_{u^+}\hat{h}(u_L,u_L)/\partial_{u^-}\hat{h}(u_L,u_L)|<1$ and the last coefficient of the perturbation, $\varsigma_j^p$, thus decay exponentially fast as $j\rightarrow-\infty$, while other coefficients are always smaller than $\varsigma_j^p$. Therefore, the induction holds true in cell $\kappa_j$ and the proof is complete since it is assumed to be true in cell $J^-$.

The proof is similar for subsonic cells. Introducing the application 
\begin{equation*}
 G_{j-1/2}^k({\bf U}_{j-1},{\bf U}_{j}):={\cal R}_j^k(u_h)=0, \quad 0\leq k\leq p,
\end{equation*}
\noindent defined by (\ref{eq:steady_discrete_scheme}b) and (\ref{eq:steady_discrete_scheme}c) for $j>J^+$ and assuming small perturbations about ${\bf U}_R=(u_R,0,\dots,0)^\top$ in $\mathbb{R}^{p+1}$, the implicit function theorem and a Taylor development about $({\bf U}_R,{\bf U}_R)$ give $\varsigma_{j}^k={\cal O}(\epsilon^2u_R^2)$ for $0\leq k<p$ and
\begin{equation*}
 \varsigma_{j}^p \simeq -\frac{\partial_{u^-}\hat{h}(u_R,u_R)}{\partial_{u^+}\hat{h}(u_R,u_R)}(u_{j-1/2}^--u_R).
\end{equation*}
\end{proof}

Note that the assumption (\ref{eq:hyp_exp_decay}) holds for the LLF flux. Indeed, we have $\partial_{u^-}\hat{h}(u^-,\\ u^+) = (f'(u^-)+\alpha)/2$ and $\partial_{u^+}\hat{h}(u^-,u^+)=(f'(u^+)-\alpha)/2$ with $-\alpha<f'(u_R)<0<f'(u_L)<\alpha$. Numerical experiments of section \ref{sec:num_xp_prof} will illustrate the results of Theorem~\ref{th:exp_decay}. In \cite{smyrlis90}, Smyrlis used similar arguments to demonstrate the exponential decay of oscillations in the vicinity of a stationary shock for the Lax-Wendroff scheme observed in precedent numerical experiments \cite{harten_et_al76}. In this case, monotonicity is lost and oscillations change sign from one cell to another.

The proof of Theorem~\ref{th:exp_decay} shows that the oscillations are transmitted from a cell to its neighboring cell $\kappa_j$ by the numerical fluxes and that only the highest DOF is affected. In particular, setting $U_j^p=0$ one recovers the uniform solution up to first order: $u_h|_{\kappa_j}=u_X+{\cal O}(\epsilon^2u_X^2)$. This may motivate limiting techniques of the solution to suppress spurious oscillations and a first attempt in this direction will be shown in section~\ref{sec:SVM}.

\subsection{Numerical solution in the shock region with the Godunov flux}\label{sec:steady_shock_sol_god}

In the case of the Godunov flux, the mean value of the numerical solution in the shock cell is explicitly known. This is the object of the following lemma.

\begin{lemma}[Global conservation]\label{th:shock_sol_DOF0}
Let $h>0$ and $p\geq0$ and use the Godunov flux (\ref{eq:godunov_flux}), then under the assumptions (\ref{eq:sign_condition_conv-flx}) the first DOF in cell $\kappa_{0}$ is defined by
\begin{equation}\label{eq:shock_sol_DOF0}
 U_{0}^0 = \langle u\rangle_0 = \frac{(1+s_{c})u_L+(1-s_c)u_R}{2},
\end{equation}
\noindent where $\langle u\rangle_0$ denotes the mean value (\ref{eq:mean_sol}) of the exact solution (\ref{eq:exact_sol}) in cell $\kappa_0$, and
\begin{equation}\label{eq:def_sc}
 s_{c}=\frac{2}{h}(x_c-x_{0})
\end{equation}

\noindent is the relative shock position in $\kappa_{0}$ and has values $s_{c}=\pm1$ when the shock is at interfaces $x_{\pm1/2}$.
\end{lemma}

\begin{proof}
Setting $v_h=1_{\kappa_j}$ into (\ref{eq:discr_var_form}) and summing up the results over all cells $\kappa_j$ in $\Omega_h$, one obtains
\begin{equation*}
 h\frac{d}{dt}\Big(\sum_{\kappa_j\in\Omega_h}U_j^0\Big) = \sum_{\kappa_j\in\Omega_h}{\cal R}_j^0 = -\sum_{\kappa_j\in\Omega_h} \hat{h}_{j+\frac{1}{2}}-\hat{h}_{j-\frac{1}{2}} = f(u_L)-f(u_R) = 0,
\end{equation*}
\noindent according to (\ref{eq:discrete_shock_at_infty}). Integrating (\ref{eq:model_pb}a) in space over $\Omega$, one obtains $d_t\int_\Omega udx=f(u_L)-f(u_R)=0$. Subtracting this result to the above equation, one obtains
\begin{equation*}
 \frac{d}{dt}\Big(h\sum_{\kappa_j\in\Omega_h}U_j^0-\int_{\Omega}u(x) dx\Big) = h\frac{d}{dt}\Big(U_0^0-\langle u\rangle_0\Big) = 0,
\end{equation*}
\noindent from Theorems~\ref{th:sol_super_zone_conv_flx} and \ref{th:sol_sub_zone_conv_flx}, which gives $U_0^0=\langle u\rangle_0$ from (\ref{eq:discr_CI}).
\end{proof}

\subsection{Some comments on assumption (\ref{eq:sign_condition_conv-flx})}\label{sec:comments_trace_ineq}

The above results with the Godunov flux are based on assumption (\ref{eq:sign_condition_conv-flx}) on the traces of the numerical solution at interfaces of the shock cell. However, from (\ref{eq:sol_godunov_flux}) we observe that this assumption may be violated when either (i)  $u_{1/2}^- \leq u_L$ and $u_{-1/2}^+ = u_R$, or (ii) $u_{1/2}^- = u_L$ and $u_{-1/2}^+ \geq u_R$. Our objective is here to show that, up to a shift in the index of the shock cell, $j_c$, (\ref{eq:godunov-flx_interf_jc}) still holds. This is a consequence of the following result: for oscillations in cell $\kappa_0$, situations (i) and (ii) reduce to 
\begin{subeqnarray}\label{eq:sol_violating_hyp_traces}
 u_{-3/2}^+ = u_{-1/2}^-=u_L, & u_{-1/2}^+ = u_{1/2}^-=u_R,\\
 u_{-1/2}^+ = u_{1/2}^- =u_L, & u_{1/2}^+ = u_{3/2}^- = u_R,
\end{subeqnarray}

\noindent respectively. To prove these results, consider the local residuals (\ref{eq:local_residuals}) at steady-state. Setting $v_h=u_h$ and imposing the conservation property (\ref{eq:steady_local_cons}), one obtains
\begin{eqnarray}\label{eq:steady_residual_nrj}
 {\cal R}_j(u_h,u_h) &=& g(u_{j+1/2}^-)-g(u_{j-1/2}^+) - f_\infty(u_{j+1/2}^--u_{j-1/2}^+) \nonumber\\
 &=& \big(f(\xi_j)-f_\infty\big)(u_{j+1/2}^--u_{j-1/2}^+) = 0, \quad \forall j\in\mathbb{Z},
\end{eqnarray}

\noindent where $g(u)=\int^uf(v)dv$ and $\min(u_{j-1/2}^+,u_{j+1/2}^-)\leq\xi_j\leq\max(u_{j-1/2}^+,u_{j+1/2}^-)$ from the mean value theorem. The above equation has two solutions: either $u_{j+1/2}^-=u_{j-1/2}^+$, or $f(\xi_j)=f_\infty$ and $\min(u_{j-1/2}^+,u_{j+1/2}^-)<\xi_j<\max(u_{j-1/2}^+,u_{j+1/2}^-)$.

Now, suppose for instance that case (i) holds so that $u_{-1/2}^+=u_R$. Applying (\ref{eq:steady_residual_nrj}) with $j=0$, we obtain either the trivial solution $u_{1/2}^-=u_R$, or $f(\xi_0)=f_\infty$ with $\xi_0\neq u_R$. In the latter case, we have $u_R<\xi_0=u_L<u_{1/2}^-$ by coercivity and strict convexity of $f$ (see section~\ref{sec:model_eqn}). This latter result cannot hold because of the definition of the Godunov flux which imposes $u_{1/2}^-\leq u_L$ by (\ref{eq:sol_godunov_flux}), so $u_{1/2}^-=u_R$.

Then, every solutions such that $u_{-1/2}^-<u_L$ correspond to situations where the oscillations are in cell $\kappa_{j_c}$ with $j_c\leq-1$. Indeed, by (\ref{eq:godunov_flux}) one would have $\hat{h}(u_{-1/2}^-,u_{-1/2}^+)=f(u_{-1/2}^+)$ and since $\hat{h}(u_{1/2}^-,u_{1/2}^+)=f(u_{1/2}^+)$, there exists a neighborhood of the uniform solution $u_h=u_R$ in cells $j\geq0$ where it is the unique solution from Theorem~\ref{th:sol_sub_zone_conv_flx}. Hence, by setting $\kappa_0$ the cell containing oscillations according to Theorems~\ref{th:sol_super_zone_conv_flx} and \ref{th:sol_sub_zone_conv_flx}, only $u_{-1/2}^-=u_L$ violates assumption (\ref{eq:sign_condition_conv-flx}). Using again (\ref{eq:steady_residual_nrj}), $u_{-1/2}^-=u_L$ imposes $u_{-3/2}^+=u_L$. Similar arguments hold for proving that (\ref{eq:sol_violating_hyp_traces}b) corresponds to the only situation of case (ii) violating (\ref{eq:sign_condition_conv-flx}).

To sum up, it is convenient to introduce an index of faces of the shock cell, $i_c+\tfrac{1}{2}=\pm\tfrac{1}{2}$. Then, the situations violating (\ref{eq:sign_condition_conv-flx}) reduce to one of the following cases
\begin{equation*}
 u_{i_c-1/2}^+=u_{i_c+1/2}^-=u_L, \quad u_{i_c+1/2}^+=u_{i_c+3/2}^-=u_R, \quad i_c+\tfrac{1}{2}=\pm\tfrac{1}{2}.
\end{equation*}

Therefore, the Godunov flux is not differentiable at $x_{i_c+1/2}$ only, $\hat{h}_{i_c+1/2} = f(u_{i_c+1/2}^\pm)$: one solution satisfies (\ref{eq:godunov-flx_interf_jc}) and the analysis of section~\ref{sec:steady_shock_sol_conv} remains valid, but the other one corresponds to oscillations in a neighboring cell. In the following, we will only consider solutions for which (\ref{eq:godunov-flx_interf_jc}) holds without loss of generality. Among these solutions, the situations violating (\ref{eq:sign_condition_conv-flx}) will be analyzed in the linear stability analysis of section~\ref{sec:ana_stab} and numerical experiments of section~\ref{sec:num_xp}.


%
%
\section{Steady discrete shock solutions for the Burgers equation}\label{sec:steady_shock_sol_burgers}

We are now interested in steady-state solutions of the inviscid Burgers equation, \ie, (\ref{eq:model_pb}) with $f(u)=\tfrac{1}{2}u^2$. According to  Theorems~\ref{th:sol_super_zone_conv_flx} and \ref{th:sol_sub_zone_conv_flx}, the uniform numerical solutions $u_h\equiv u_L$ in the supersonic region and $u_h\equiv u_R$ in the subsonic region hold with the Godunov flux. We thus restrict our analysis to the solution of the numerical scheme in the cell containing the shock. The result now depends on the order of the numerical scheme and we give the solutions for $1\leq p\leq3$ in the following theorem.

\begin{theorem}\label{th:shock_sol_p0123}
Consider the discrete DG scheme with the Godunov flux (\ref{eq:godunov_flux}). Let the exact position of the shock be in cell $\kappa_0$ and assume that the internal traces satisfy (\ref{eq:sign_condition_conv-flx}). Then, for $h>0$, according to the polynomial degree, $p\leq3$, and the relative position of the shock in the cell, $s_c$, defined by (\ref{eq:def_sc}), the solution of the discrete scheme (\ref{eq:steady_discrete_scheme}) reads: 
\begin{equation*}
 u_h(x) = u_L\sum_{l=0}^pu_l\phi_{0}^l(x), \quad \forall x \in \kappa_{0},
\end{equation*}
\noindent with $u_0=s_c$ and if $p=1$:
\begin{equation}\label{eq:shock_sol_p1}
 -1<s_{c}<1, \quad u_1 = -\sqrt{3(1-s_{c}^2)};
\end{equation}
\noindent if $p=2$: 
\begin{subeqnarray}\label{eq:shock_sol_p2}
  -1<s_{c}<-\frac{2}{3}, &\quad& \left\{
  \begin{array}{rcl}
   u_1 &=& 0,\\
   u_2 &=& \sqrt{5(1-s_{c}^2)},
  \end{array}
  \right.\\
  -\frac{2}{3}< s_{c}< \frac{2}{3}, &\quad& \left\{
  \begin{array}{rcl}
   u_1 &=& -\sqrt{3(1-\tfrac{9}{4}s_{c}^2)},\\
   u_2 &=& -\frac{5}{2}s_{c},
  \end{array}
  \right.\\
  \frac{2}{3}< s_{c}<1, &\quad& \left\{
  \begin{array}{rcl}
   u_1 &=& 0,\\
   u_2 &=& -\sqrt{5(1-s_{c}^2)};
  \end{array}
  \right.
\end{subeqnarray}
\noindent if $p=3$: 
 \begin{subeqnarray}
  s_c\in{\cal D}_1, && \left\{
  \begin{array}{rcl}
   u_1 &=& -\tfrac{1}{5}\sqrt{\frac{21}{2}}\sqrt{5-54s_{c}^2},\\
   u_2 &=& -7s_{c},\\
   u_3 &=& -u_1,
 \end{array}
  \right.\\
  s_c\in{\cal D}_2, && \left\{
  \begin{array}{rcl}
   u_1 &=& \frac{-1}{20\sqrt{3s_c\ol{\Delta}_2}}\big[-7(17s_{c}^2-6)^3+49s_{c}(17s_{c}^2-6)^2\ol{\Delta}_3^{1/3}+ \\
             && 26s_{c}^2(17s_{c}^2-6)\ol{\Delta}_3^{2/3}-s_{c}(463s_{c}^2-246)\ol{\Delta}_3-\\
             && 26s_{c}^2\ol{\Delta}_2\ol{\Delta}_3^{1/3}+s_{c}\ol{\Delta}_3^{5/3}+\ol{\Delta}_2^2\big]^{1/2},\\
   u_2 &=& \frac{1}{12}\big[-7s_{c}-\frac{7(17s_{c}^2-6)}{\ol{\Delta}_3^{1/3}}+\ol{\Delta}_3^{1/3}\big],\\
   u_3 &=& \frac{u_1}{3888s_{c}(s_{c}^2-1)\ol{\Delta}_2}\big[-35(17s_{c}^2-6)^3+\\
             && 4207s_{c}(17s_{c}^2-6)^2\ol{\Delta}_3^{1/3}-420s_ {c}^2(17s_{c}^2-6)\ol{\Delta}_3^{2/3}+\\
             && \ol{\Delta}_2\big(2s_{c}(5969s_{c}^2-2889)+420s_{c}^2\ol{\Delta}_3^{1/3}+\\
             && 60s_{c}\ol{\Delta}_3^{2/3}+5\ol{\Delta}_2\big)\big],
  \end{array}
  \right.
 \end{subeqnarray}
\setcounter{equation}{2}
\begin{subeqnarray}\label{eq:shock_sol_p3}
 \setcounter{subequation}{3}
  s_c\in{\cal D}_3, && \left\{
  \begin{array}{rcl}
   u_1 &=& \frac{-1}{20\sqrt{\Delta_2}}\Big[\big(108252-386051s_{c}^2+335619s_{c}^4\big)^2+ \\
             &&  \sqrt{3}\Delta_1\Big(1164-2403s_{c}^2+\Delta_3^{1/3}(26s_{c}+7\Delta_3^{1/3})\Big)+\\
             &&  \Delta_3^{1/3}(588-914s_{c}^2+1089s_{c}^4)+\\
             &&  \Delta_3^{2/3}(703-1125s_{c}^2)\Big]^{1/2},\\
   u_2 &=& \frac{1}{12}\big[-7s_{c}-\frac{7(17s_{c}^2-6)}{\Delta_3^{1/3}}-\Delta_3^{1/3}\big],\\
   u_3 &=& \frac{7u_1}{108(1-s_{c}^2)\Delta_3^{2/3}}\big[35(12-161s_{c}^2+236s_{c}^4)+\\
             && 5\sqrt{3}\Delta_1(\Delta_3^{1/3}-7)+5\Delta_3^{1/3}s(79-100s_{c}^2)+\\
             && 2\Delta_3^{2/3}(32s_{c}^2-17)\big],
  \end{array}
  \right.
 \end{subeqnarray} 
\noindent where ${\cal D}_1=(-\tfrac{1}{6},\tfrac{1}{6})$, ${\cal D}_2=(-\sqrt{\frac{6}{17}},-\tfrac{1}{6})$, ${\cal D}_3=(-1,-\sqrt{\frac{6}{17}})\cup(\tfrac{1}{6},1)$; $\Delta_3=7\Delta_2$, $\ol{\Delta}_3=-7\Delta_2$, $\ol{\Delta}_2=-\Delta_2$ and
 \begin{subeqnarray}\label{eq:prop_def_Deltai}
  \Delta_2 &=& 3\sqrt{3}\Delta_1 - 419s_{c}^3 + 279s_{c}, \\
  \Delta_1 &=& \big[7776s_{c}^6-10008s_{c}^4+3359s_{c}^2-56\big]^{\tfrac{1}{2}}.
 \end{subeqnarray}

\end{theorem}

\begin{proof}
 See Appendix~\ref{app:proof_sol_p123}.
\end{proof}

Several conclusions may be inferred from Theorem~\ref{th:shock_sol_p0123}. First, for a polynomial degree $p$, \ie, a given DG scheme, the solutions are parametrized by the relative shock position in the cell only. In particular, the numerical solution is independent of the mesh size and the amplitude of the oscillations are proportional to the shock strength $u_L-u_R$. These oscillations occur for $p\geq1$ and vanish if and only if the shock is at an interface: $s_c=\pm1$. Then, the coefficients are displayed in Figure~\ref{fig:evol_coeffs_p1-3}(a-c) and present continuous evolution with $s_c$. The validity of the assumption (\ref{eq:sign_condition_conv-flx}) is illustrated in Figure~\ref{fig:evol_coeffs_p1-3}(d) and is satisfied according to the ranges of solutions in Theorem \ref{th:shock_sol_p0123}. We observe that for $p=2$ (resp. $p=3$), the positions $s_c=\pm\tfrac{2}{3}$ (resp. $s_c=\pm\tfrac{1}{6}$) correspond to situations where $u_{-1/2}^+=u_{1/2}^-=u_L$ or $u_R$ (see section~\ref{sec:comments_trace_ineq}). It is remarkable that assumption (\ref{eq:sign_condition_conv-flx}) selects the solutions over the whole range $s_c$ in $[-1,1]$. Finally, the oscillatory behavior of the solution may lead to entropy violating solutions where the sign of eigenvalues changes locally upstream and/or downstream of the shock position. Since the numerical solution can potentially change sign $p$ times in a cell, this may lead to nonphysical situations. Numerical experiments in section \ref{sec:num_xp} will show that those situations may exist for $p\geq2$. This feature is different from what has been shown in \cite{lerat_13} for the high-order RBC schemes where the oscillations were seen to monotonically decrease for a shock position moving from the center of the cell to its faces.

\begin{figure}
\begin{center}
\subfigure[$p=1$]{\epsfig{figure=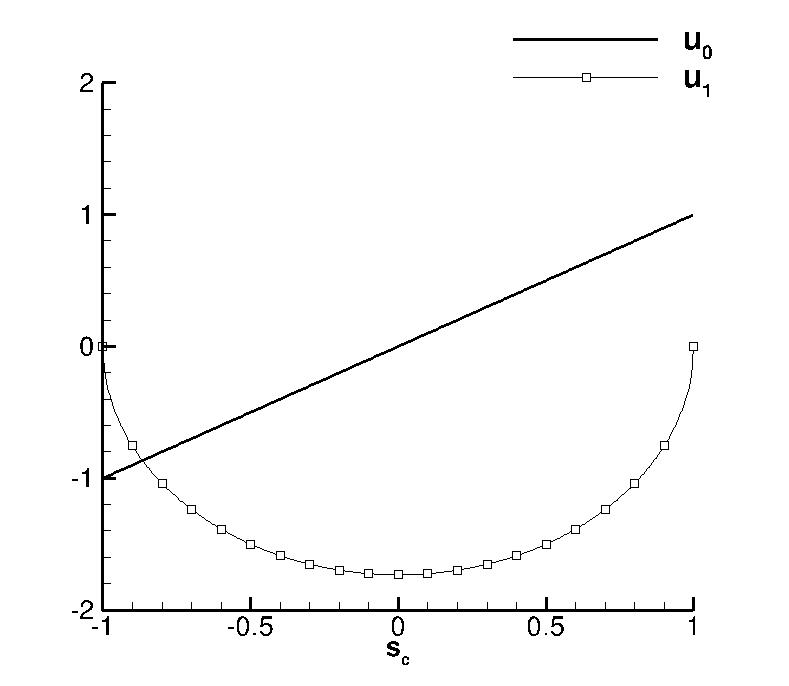 ,width=6cm}}
\subfigure[$p=2$]{\epsfig{figure=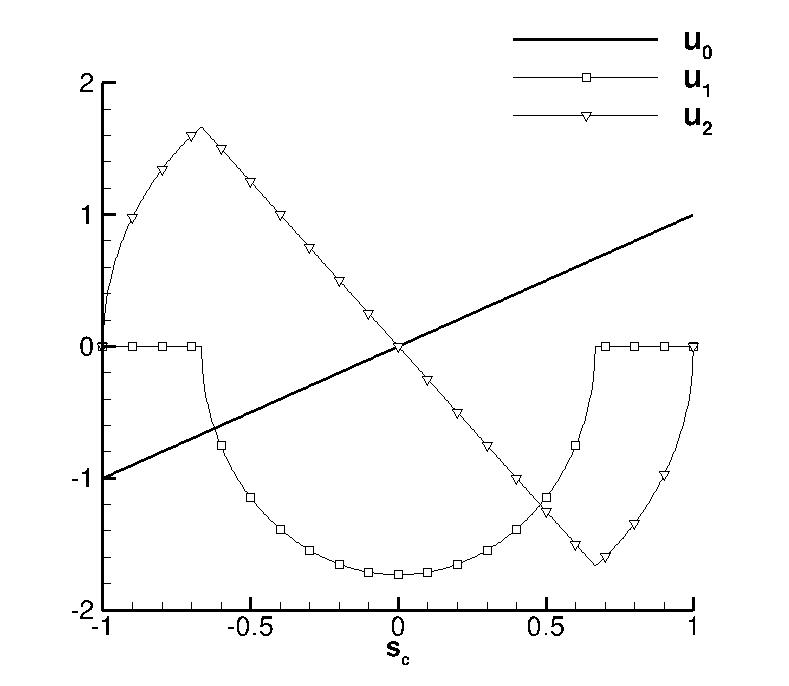 ,width=6cm}}\\
\subfigure[$p=3$]{\epsfig{figure=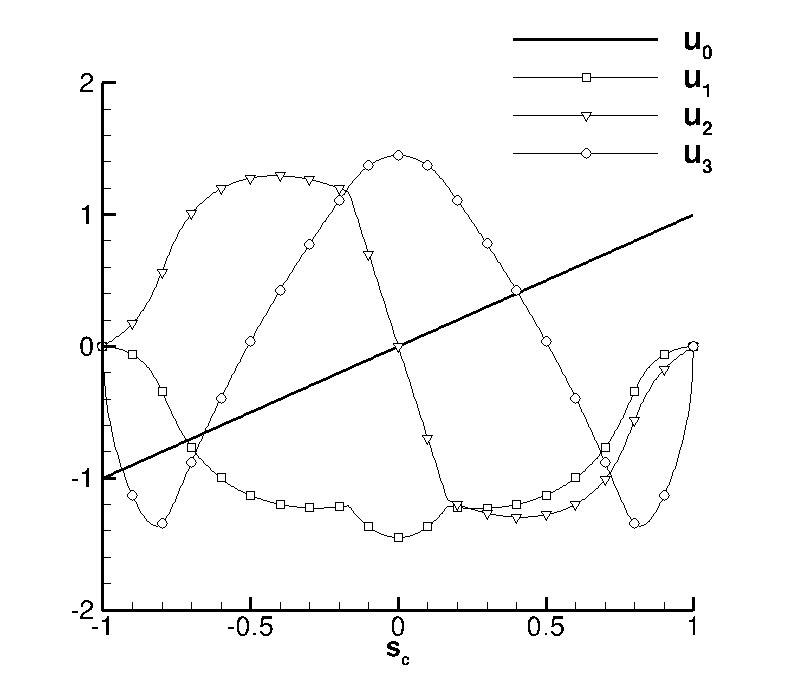 ,width=6cm}}
\subfigure[]     {\epsfig{figure=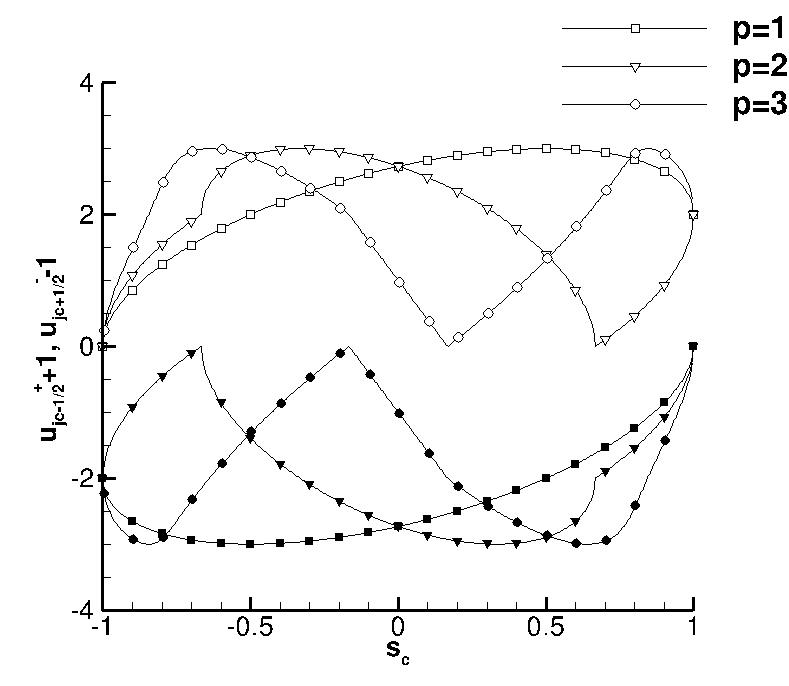 ,width=6cm}}
\caption{From (a) to (c): evolution of coefficients in the function basis of the solution in shock cell $\kappa_{0}$ given by Theorem~\ref{th:shock_sol_p0123}. In (d): evolution of the quantities $u_{-1/2}^++u_L$ (open symbols) and $u_{1/2}^-+u_R$ (full symbols) with $u_L=-u_R=1$}
\label{fig:evol_coeffs_p1-3}
\end{center}
\end{figure}

%
%
\section{Linear stability of steady shock profiles}\label{sec:ana_stab}
\subsection{Linearized operator}

In this section, we are interested in the linear stability of the  DG scheme around steady shock solutions $u_h$ satisfying (\ref{eq:steady_discrete_scheme}). We will mainly focus on sufficient conditions for instability of the numerical scheme. These conditions will then be used to analyze non-convergence to steady-state of the DG scheme observed under certain conditions in the experiments of section~\ref{sec:num_xp}.

To this end, we first consider a forward Euler method for the time discretization, setting $v_h\equiv\phi_j^k$ in (\ref{eq:discr_var_form}), we obtain 
\begin{equation}\label{eq:discrete_scheme}
 U_j^{k(n+1)} = U_j^{k(n)} + \lambda_k{\cal R}_j^k(u_h^{(n)}), \quad \forall j\in\mathbb{Z}, \quad 0\leq k \leq p,
\end{equation}

\noindent where $U_j^{k(n)}=U_j^{k}(n\Delta t)$ is a function of time, $\lambda_k=(2k+1)\lambda$ and $\lambda=\tfrac{\Delta t}{h}$ with $\Delta t>0$ the time step. We shall assume that the numerical flux is differentiable at every point $x_{j+1/2}$ with respect to the left and right traces $u_{j+1/2}^\pm$ of the steady solution. 

Let $\phi:{\cal V}_h^p\rightarrow{\cal V}_h^p;u_h^{(n)}\mapsto u_h^{(n+1)}=\phi(u_h^{(n)})$ be the map defined by (\ref{eq:discrete_scheme}) where $u_h^{(n)}$ is defined by (\ref{eq:_steady_num_sol}) with components $U_j^{k(n)}$. Now define the G\^ateaux derivative of $\phi$ at point $u_h$ in the direction $w_h$ by $L:{\cal V}_h^p\rightarrow{\cal V}_h^p;w_h\mapsto Lw_h=d\phi(u_h;w_h)$. Using the vector notation ${\bf Lw}$ for the components in ${\cal V}_h^p$ of $Lw_h$, the G\^ateaux derivative of the DG operator (\ref{eq:discrete_scheme}) along the direction $w_h$ reads
\begin{equation*}
 ({\bf L}{\bf w})_j^k = W_j^k + \lambda_k \lim_{\epsilon\rightarrow0} \frac{{\cal R}_j^k(u_h+\epsilon w_h)-{\cal R}_j^k(u_h)}{\epsilon}, \quad \forall j\in\mathbb{Z}, \quad 0\leq k \leq p.
\end{equation*}

Using expression (\ref{eq:steady_discrete_scheme_def}) for the local residuals, we obtain
\begin{eqnarray*}
 {\cal R}_j^k(u_h+\epsilon w_h) &=& \int_{\kappa_j} f(u_h+\epsilon w_h)d_x\phi_j^k dx - \hat{h}\big(u_{j+\frac{1}{2}}^-+\epsilon w_{j+\frac{1}{2}}^-,u_{j+\frac{1}{2}}^++\epsilon w_{j+\frac{1}{2}}^+\big) \\
&& + (-1)^k\hat{h}\big(u_{j-\frac{1}{2}}^-+\epsilon w_{j-\frac{1}{2}}^-,u_{j-\frac{1}{2}}^++\epsilon w_{j-\frac{1}{2}}^+\big).
\end{eqnarray*}

Substracting ${\cal R}_j^k(u_h)$, dividing by $\epsilon$, and letting $\epsilon$ tend to zero, the components of ${\bf L}{\bf w}$ read
\begin{eqnarray}
 ({\bf L}{\bf w})_j^k &=& W_j^k + \lambda_k \Big( \int_{\kappa_j}w_hf'(u_h)d_x\phi_j^kdx - w_{j+\frac{1}{2}}^-\partial_{u^-}\hat{h}_{j+\frac{1}{2}}-w_{j+\frac{1}{2}}^+\partial_{u^+}\hat{h}_{j+\frac{1}{2}} \nonumber\\
 & & + (-1)^k\big(w_{j-\frac{1}{2}}^-\partial_{u^-}\hat{h}_{j-\frac{1}{2}}+w_{j-\frac{1}{2}}^+\partial_{u^+}\hat{h}_{j-\frac{1}{2}}\big)\Big).
\end{eqnarray}

The linearized operator ${\bf L}$ is thus made of three diagonals of blocks of size $(p+1)\times(p+1)$ with entries
\begin{subeqnarray}\label{eq:linear_operator}
 ({\bf L}_{j,j-1})_{k,l} &=& (-1)^k\lambda_k\partial_{u^-}\hat{h}_{j-\frac{1}{2}},\\
 ({\bf L}_{j,j})_{k,l}   &=& \delta_{k,l} \hspace{-0.05cm}+\hspace{-0.05cm} \lambda_k\big(\hspace{-0.15cm}\int_{\kappa_j}\hspace{-0.2cm}f'(u_h)\phi_j^ld_x\phi_j^kdx - \partial_{u^-}\hat{h}_{j+\frac{1}{2}}-(-1)^{k}\partial_{u^+}\hat{h}_{j-\frac{1}{2}} \big),\\
 ({\bf L}_{j,j+1})_{k,l} &=& -(-1)^l\lambda_k\partial_{u^+}\hat{h}_{j+\frac{1}{2}},
\end{subeqnarray}
\noindent for $0\leq k,l\leq p$ and $j$ in $\mathbb{Z}$.

For the shock profile $u_h$ to be linearly stable, it is necessary that the spectrum of ${\bf L}$ contains only eigenvalues $\mu$ with modulus lower than unity, $|\mu|\leq1$, and semisimple eigenvalues with unit modulus, $|\mu|=1$ \cite{bultelle_et-al98}. In the following, we will focus on sufficient conditions for instability of shock profiles. 

\subsection{The case of the Godunov numerical flux}\label{sec:ana_stab_god}

The upwind character of the Godunov flux allows to specify the eigenvalues of the linearized operator in the following proposition.

\begin{proposition}\label{th:Godunov_spectrum}
Under the assumptions of Lemma~\ref{th:shock_sol_DOF0}, the spectrum of the linearized operator (\ref{eq:linear_operator}) for the Godunov flux (\ref{eq:godunov_flux}) reduces to the spectra of the following matrices of size $p+1$:
\begin{subeqnarray}\label{eq:linear_operator_god}
 ({\bf L}_{j<0})_{k,l} &=& \delta_{k,l} + \lambda_kf'(u_L)(N_{k,l}-1), \\
 ({\bf L}_{0})_{k,l}   &=& \delta_{k,l} + \lambda_k\int_{\kappa_{0}}f'(u_h)\phi_0^ld_x\phi_0^kdx,\\
 ({\bf L}_{j>0})_{k,l} &=& \delta_{k,l} + \lambda_kf'(u_R)(N_{k,l}+(-1)^{k+l}),
\end{subeqnarray}

\noindent for $0\leq k,l\leq p$, where $u_h$ denotes the solution of (\ref{eq:steady_discrete_scheme}) and $N_{k,l}$ is defined by (\ref{eq:matN_kl}).
\end{proposition}

\begin{proof}
Under assumption (\ref{eq:sign_condition_conv-flx}), the Godunov flux reduces to the upwind flux (\ref{eq:godunov-flx_interf_jc}). Therefore, for $j\neq 0$ the integral in (\ref{eq:linear_operator}b) reads 
\begin{equation*}
 \int_{\kappa_j}f'(u_h)\phi_j^ld_x\phi_j^kdx = f'(u_X)N_{k,l},
\end{equation*}

\noindent with $u_X=u_L$ for $j<0$ and $u_X=u_R$ for $j>0$. After simplification and omitting the double subscript for diagonal blocks, the linearized operator reduces to
\begin{equation*}
 ({\bf L}_{j,j-1})_{k,l} = (-1)^k\lambda_kf'(u_L), \; ({\bf L}_{j})_{k,l} = \delta_{k,l} + \lambda_k f'(u_L)(N_{k,l}-1), \;  {\bf L}_{j,j+1} = 0, \; j<0,
\end{equation*}
\begin{equation*}
 ({\bf L}_{0,-1})_{k,l} = (-1)^k\lambda_kf'(u_L), \; ({\bf L}_{0})_{k,l} = \delta_{k,l} + \lambda_k\int_{\kappa_{0}}f'(u_h)\phi_0^ld_x\phi_0^kdx,
\end{equation*}
\noindent $({\bf L}_{0,1})_{k,l}= -(-1)^l\lambda_kf'(u_R)$, and
\begin{equation*}
 {\bf L}_{j,j-1} = 0, \, ({\bf L}_{j})_{k,l} = \delta_{k,l} + \lambda_k f'(u_R)(N_{k,l}+(-1)^{k+l}), \,  ({\bf L}_{j,j+1})_{k,l} = -(-1)^l\lambda_kf'(u_R),
\end{equation*}
\noindent for $j>0$. Therefore, ${\bf L}$ is lower block triangular for $j\leq 0$ and upper block triangular for $j>0$. The eigenvalues of the matrix are thus the eigenvalues of the diagonal blocks which are constant and equal to (\ref{eq:linear_operator_god}).
\end{proof}


%
\subsection{Application to the Burgers equation}\label{sec:ana_stab_god_burgers}

As an application, Table~\ref{tab:eigenvalues} gives the eigenvalues of the three different blocks in (\ref{eq:linear_operator_god}) for polynomial approximations $0\leq p\leq2$ in the case of the Burgers equation $f(u)=\tfrac{u^2}{2}$. We stress that eigenvalues in blocks $j\neq0$ remain valid for a general physical flux $f$.

\begin{table}
     \begin{center}
     \caption{Eigenvalues of the iteration matrix (\ref{eq:linear_operator}) associated to the DG scheme with the Godunov numerical flux and a forward Euler time integration for the discretization of the Burgers equation. Here, $\lambda_L=\lambda f'(u_L)$, $\lambda_R=\lambda f'(u_R)$, $\ol{\lambda}=u_L\lambda$, $\gamma_1=3^{\tfrac{2}{3}}-3^{\tfrac{1}{3}}$, and $\gamma_2=3^{\tfrac{7}{6}}+3^{\tfrac{5}{6}}$.}
     \begin{tabular}{|c|c|c|c|c|}
	\hline
	$p$ & $s_c$ & $j<0$ & $j=0$ & $j>0$ \\
	\hline
	$0$ & $(-1,1)$ & $1-\lambda_L$ & $1$ & $1+\lambda_R$ \\
	\hline
	$1$ & $(-1,1)$ & $1-\lambda_L(2\pm i\sqrt{2})$ & $1$, $1-2\ol{\lambda}\sqrt{3(1-s_c^2)}$ & $1+\lambda_R(2\pm i\sqrt{2})$ \\
	\hline
	\multirow{7}{*}{$2$} & \multirow{2}{*}{$(-1,-\tfrac{2}{3})$} & \multirow{2}{*}{$1-\lambda_L(3+\gamma_1)$,} & $1$, $1\pm2i\ol{\lambda}\times\dots$ & \multirow{2}{*}{$1+\lambda_R(3+\gamma_1)$,} \\
	 & & & $\sqrt{-3s_c\sqrt{5(1-s_c^2)}-6(1-s_c^2)}$ & \\
	\cline{2-2} \cline{4-4}
	 & \multirow{2}{*}{$(-\tfrac{2}{3},\tfrac{2}{3})$} & \multirow{2}{*}{$1-\frac{\lambda_L}{2}\big(6-\gamma_1$} & $1$, $1-\ol{\lambda}\sqrt{3(4-9s_c^2)}$, & \multirow{2}{*}{$1+\frac{\lambda_R}{2}\big(6-\gamma_1$} \\
	 & & & $1-2\ol{\lambda}\sqrt{3(4-9s_c^2)}$ & \\
	\cline{2-2} \cline{4-4}
	 & \multirow{2}{*}{$(\tfrac{2}{3},1)$} & \multirow{2}{*}{$\pm\gamma_2 i\big)$} & $1$, $1\pm2i\ol{\lambda}\times\dots$ & \multirow{2}{*}{$\pm\gamma_2 i\big)$} \\
	 & & & $\sqrt{3s_c\sqrt{5(1-s_c^2)}-6(1-s_c^2)}$ & \\
	\hline
    \end{tabular}
    \label{tab:eigenvalues}
    \end{center}
\end{table}


Stability in the shock cell imposes $\ol{\lambda}=u_L\lambda\leq(3(1-s_c^2))^{-1/2}$ for $p=1$, $\ol{\lambda}\leq(3(4-9s_c^2))^{-1/2}$ for $p=2$ and $|s_c|<\tfrac{2}{3}$, but the linearized operator (\ref{eq:linear_operator}) is unconditionally unstable for $p=2$ and $|s_c|>\tfrac{2}{3}$. The DG method is indeed unstable for $p\geq1$ and a forward Euler method. These modes may be stabilized by using Runge-Kutta schemes of sufficient order (see below and section~\ref{sec:num_xp_stab}). For $p=2$ and $s_c\in(-1,-\tfrac{2}{3})$, the real part of eigenvectors associated to the unstable eigenvalues has ${\bf r}=(0,0,2\sqrt{5(1-s_c^2)}+5s_c)^\top$ as components in the basis of ${\cal P}_2(\kappa_{0})$. For $p=2$ and $s_c\in(\tfrac{2}{3},1)$, their real part reads ${\bf r}=(0,0,2\sqrt{5(1-s_c^2)}-5s_c)^\top$. They only affect the highest DOF and reach largest values at faces of the shock cell.

Let us consider the linear stability of situations violating assumption (\ref{eq:sign_condition_conv-flx}) according to section~\ref{sec:comments_trace_ineq}. For these points, Table~\ref{tab:eigenvalues_sc_pm1} displays the eigenvalues and eigenvectors in the shock cell ${\bf L}_0$. We observe that for $s_c=\pm1$ and $1\leq p\leq2$, $\mu=1$ is a non semisimple eigenvalue of ${\bf L}$. Indeed, $\lambda=\tfrac{\Delta t}{h}>0$ hence $\mu=1$ is not an eigenvalue of blocks ${\bf L}_{j<0}$ or ${\bf L}_{j>0}$ (see Table~\ref{tab:eigenvalues}), but it is an eigenvalue of ${\bf L}_{0}$ with algebraic multiplicity $p+1$ and geometric multiplicity of $1$ as indicated in Table~\ref{tab:eigenvalues_sc_pm1}. This property holds for $p=3$ because eigenvalues of ${\bf L}_{j<0}$ and ${\bf L}_{j>0}$ satisfy $\mu\neq1$. Indeed, they are of the form $1-\lambda_L\mu_l$ and $1+\lambda_R\mu_l$, respectively, with $\mu_l\in\{-4+(\gamma_3-4)^{1/2}\pm\gamma_4^+,-4-(\gamma_3-4)^{1/2}\pm\gamma_4^-\}$, $\gamma_3=10^{2/3}(i\sqrt{6}-2)^{-1/3}+(10(i\sqrt{6}-2))^{1/3}$ and $\gamma_4^\pm=i(8+\gamma_3\pm8(\gamma_3-4)^{-1/2})^{1/2}$. Note that the eigenvector associated to $\mu=1$ has components along the highest DOF only. This property will be used in section~\ref{sec:SVM} for stabilizing the DG scheme.

Table~\ref{tab:eigenvalues_sc_pm1} also gives the eigenvalues and eigenvectors for shock positions $-1<s_c<1$ where the strict inequality in assumption (\ref{eq:sign_condition_conv-flx}) is violated. This situation occurs for $p=2$ and $p=3$ as highlighted in Theorem~\ref{th:shock_sol_p0123}. Though different from results for $s_c=\pm1$, these results show that $\mu=1$ is also a non semisimple eigenvalue and the DG scheme will be unstable. At this time, a general result about instability has not been achieved and is beyond the scope of the present study.

\begin{table}
     \begin{center}
     \caption{Eigenvalues and eigenvectors in the shock cell of the iteration matrix associated to the DG scheme with the Godunov numerical flux and a forward Euler time integration for the discretization of the Burgers equation. Here, $\ol{\lambda}=u_L\lambda$}
     \begin{tabular}{|c|c|c|c|}
	\hline
	$p$ & $s_c$ & eigenvalues & eigenvectors \\
	\hline
	$1$ & $\pm1$ & $\{1,1\}$ & $\begin{pmatrix}0&0\\1&0\end{pmatrix}$ \\
	\hline
	\multirow{4}{*}{$2$} & $\pm1$ & $\{1,1,1\}$ & $\begin{pmatrix}0&0&0\\0&0&0\\1&0&0\end{pmatrix}$ \\
	\cline{2-4}
	 & $\pm\tfrac{2}{3}$ & $\{1,1,1\}$ & $\begin{pmatrix}\tfrac{1}{2}&0&0\\0&1&0\\1&0&0\end{pmatrix}$ \\
	\hline
	\multirow{5}{*}{$3$} & $\pm1$ & $\{1,1,1,1\}$ & $\begin{pmatrix}0&0&0&0\\ 0&0&0&0 \\0&0&0&0 \\1&0&0&0\end{pmatrix}$ \\
	\cline{2-4}
	& $\pm\tfrac{1}{6}$ & $\{1,1,1,1-\tfrac{10\ol{\lambda}}{\sqrt{3}}\}$ & 
$\begin{pmatrix}\mp\tfrac{5}{9\sqrt{3}}&0&0&0\\ \tfrac{8}{27}&\mp\tfrac{\sqrt{3}}{3}&0&0 \\0&1&0&\mp\tfrac{3\sqrt{3}}{7} \\1&0&0&1\end{pmatrix}$ \\
	\hline
    \end{tabular}
    \label{tab:eigenvalues_sc_pm1}
    \end{center}
\end{table}

Stability of the DG scheme requires high-order Runge-Kutta schemes \cite{cockburn-shu01}. However, the scheme may remain linearly unstable at points violating assumption (\ref{eq:sign_condition_conv-flx}). As an example, consider the second-order and strong stability preserving Heun scheme whose linearized operator reads ${\bf L}_{RK2}={\bf L}+\tfrac{1}{2}({\bf L}-{\bf I})^2$. Solutions of equation $\mu_{RK2}=\mu+\tfrac{1}{2}(\mu-1)^2=1$ are $\mu=\pm1$. We note that ${\bf L}$ is lower triangular by blocks for rows $j<0$ and upper triangular by blocks for rows $j>0$, so is ${\bf L}_{RK2}$. As a consequence, the stability of ${\bf L}_{RK2}$ reduces to the stability of its diagonal blocks. Moreover, the transformation from ${\bf L}$ to ${\bf L}_{RK2}$ results on the same operation of its diagonal blocks. Now, we note that $\mu=\pm1$ is not a root of $\mu_{RK2}=1$ for blocks ${\bf L}_{j<0}$ or ${\bf L}_{j>0}$ for $p\leq3$. Indeed, according to Table~\ref{tab:eigenvalues} and precedent remarks, the only roots are $\lambda_L=\lambda f'(u_L)=2$ for $p=0$ and $\lambda_L=2/(3+\gamma_1)\simeq0.5498$ for $p=2$ which do not satisfy the usual CFL condition $\max\{|\lambda_L|,|\lambda_R|\}<1/(2p+1)$ \cite{cockburn-shu01}. The same analysis holds for $\lambda_R=\lambda f'(u_R)$.


%
%
\section{Numerical experiments}\label{sec:num_xp}

We consider the inviscid Burgers equation, $f(u)=\tfrac{u^2}{2}$, over $\Omega=[0,1]$ with boundary conditions $u(0)=-u(1)=1$. The numerical solution is obtained by using a method of lines. The semi-discrete equation (\ref{eq:discr_var_form}) is advanced in time by means of an explicit third-order and strong stability preserving Runge-Kutta method \cite{shu-osher88}. We look for steady-state solutions of  (\ref{eq:discr_var_form}) of the form $u_h(x)=\lim_{n\rightarrow\infty}u_h(x,n\Delta t)$. The time step is set at
\begin{equation*}
 \Delta t = CFL\times\min\Big\{\frac{h}{\max_{v\in\kappa_j}|f'(v)|}:\;\kappa_j\in\Omega_h\Big\},
\end{equation*}
\noindent where $CFL=1/(2p+1)$ according to \cite{cockburn-shu01} and the maximum eigenvalue is evaluated at quadrature points of the element $\kappa_j$.

As suggested in \cite{lerat_13}, the final shock position for the Burgers equation is set through the following initial condition
\begin{equation*}
 u_0(x) = \left\{
\begin{array}{ll}
  1 - 2(1-\ol{u})x, &\mbox{if } x < \frac{1}{2},\\
  2\ol{u}+1-2(\ol{u}+1)x, &\mbox{if } x \geq \frac{1}{2},
\end{array}
\right.
\end{equation*}

\noindent with $-2<\ol{u}<2$. Using (\ref{eq:global_cons_edp}), we obtain $x_c=1/2+\ol{u}/4$.

We note that the evaluation of the volume integral in (\ref{eq:local_residuals}) is done by using Gauss quadrature which may be inexact for a nonlinear flux. The present results have been obtained by using a numerical quadrature of sufficient order to integrate it exactly in the case of the Burgers equation: $p+1$ Gauss-Legendre points are used for $p\leq2$ and $p+2$ points are used for $p=3$. The extra point for $p=3$ allows to quantitatively compare the solution of the numerical scheme obtained from the theoretical analysis in Theorem~\ref{th:shock_sol_p0123} with the solution obtained from a numerical calculation. 

\subsection{Structure of steady shock profiles}\label{sec:num_xp_prof}

Figures \ref{fig:solutions_p1} to \ref{fig:solutions_p3} display the steady-state solutions to the DG scheme obtained with the Godunov flux (\ref{eq:godunov_flux}) for nine different shock positions (\ref{eq:def_sc}). We compare solutions obtained from the theoretical analysis in Theorem~\ref{th:shock_sol_p0123} with the solution obtained from a numerical calculation with $N=20$ cells and $j_c=11$. Solutions are also compared to the exact solution (\ref{eq:exact_sol}). 


The solution remains uniform in the supersonic and subsonic regions $x<x_{-1/2}$ and $x>x_{1/2}$ whatever the polynomial degree which confirms the conclusions from Theorems~\ref{th:sol_super_zone_conv_flx} and \ref{th:sol_sub_zone_conv_flx}. However, for $p=2$ and $s_c=1$ or for $p=3$ and $s_c=\pm1$, the solution from the numerical calculation appears to be oscillatory outside the shock cell and differs from the expected predictions. In all these situations, the calculations did not succeed in converging to a steady state and therefore do not satisfy the discrete scheme (\ref{eq:steady_discrete_scheme}). The analysis in section~\ref{sec:ana_stab_god_burgers} predicts instability in situations $s_c=\pm1$ for $1\leq p\leq3$, $s_c=\pm\tfrac{2}{3}$ for $p=2$, and  $s_c=\pm\tfrac{1}{6}$ for $p=3$. The analysis was limited to a second-order Runge-Kutta scheme, but results of section~\ref{sec:num_xp_stab} show that it holds for a third-order scheme. We recall that assumption (\ref{eq:sign_condition_conv-flx}) may be violated for $p\geq1$ at these points (see Theorem~\ref{th:shock_sol_p0123} and Appendix~\ref{app:proof_sol_p123}) where the Godunov flux admits two equal values $\hat{h}_{-1/2}=f(u_{-1/2}^\pm)$ or $\hat{h}_{1/2}=f(u_{1/2}^\pm)$ (see Figure~\ref{fig:evol_coeffs_p1-3}(d)). Both values may be solutions of the numerical scheme as long as that the conservation of the scheme (\ref{eq:steady_local_cons}) is respected. Our numerical experiments tend to show that, when converging to the steady state, the Godunov flux changes for a solution to another following a cyclic pattern with the consequence that the flux balance in the cell is periodically modified (see Figure~\ref{fig:residuals_p2}(b)). In some cases, {\it e.g.}, $p=1$ and $s_c=\pm1$ or $p=2$ and $s_c=-1$, the convergence to steady state was reached but at very low speed.   

\begin{figure}
\begin{center}
\subfigure[$s_c=-1$           ]{\epsfig{figure=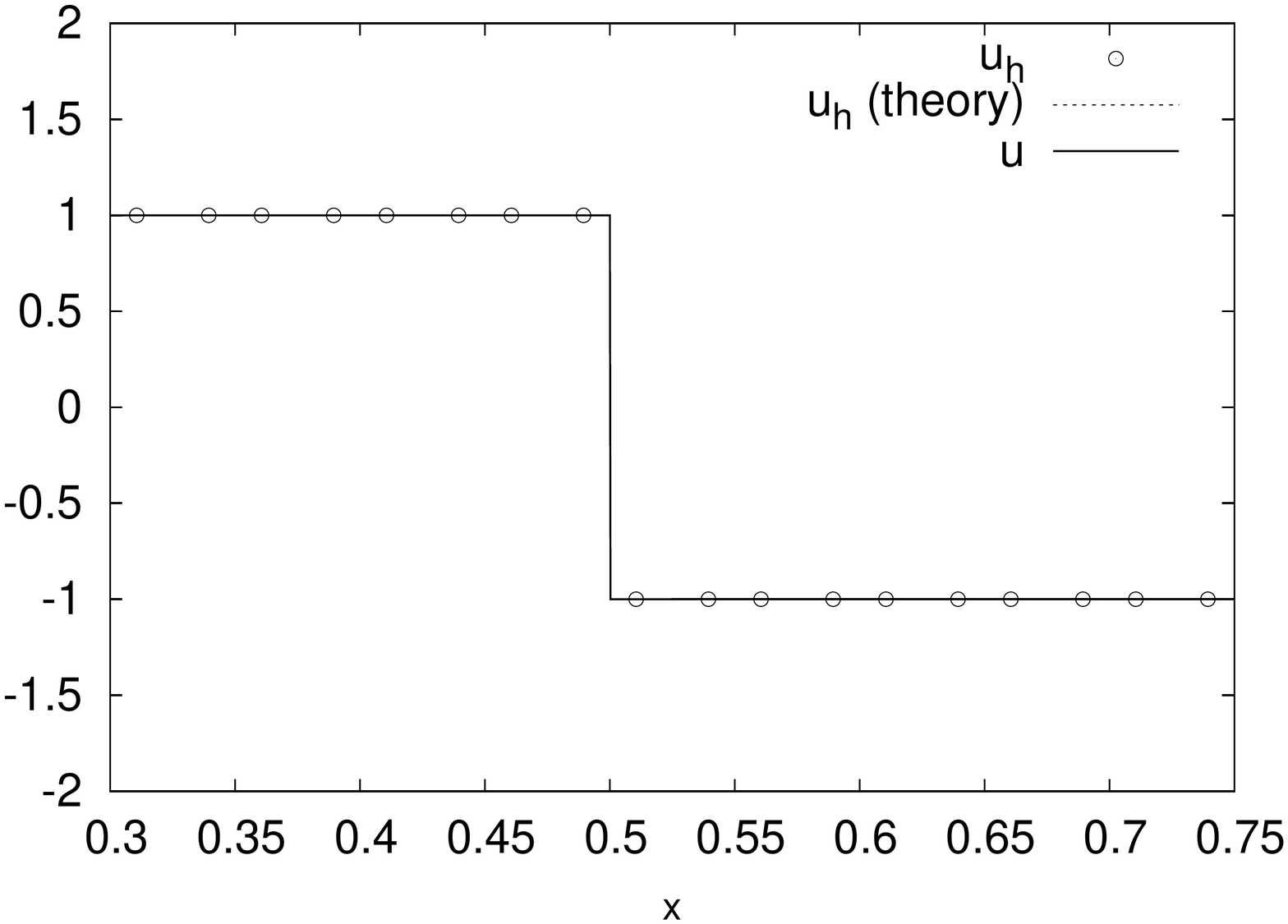 ,width=4.8cm}}\hspace{-0.8cm}
\subfigure[$s_c=-\tfrac{4}{5}$]{\epsfig{figure=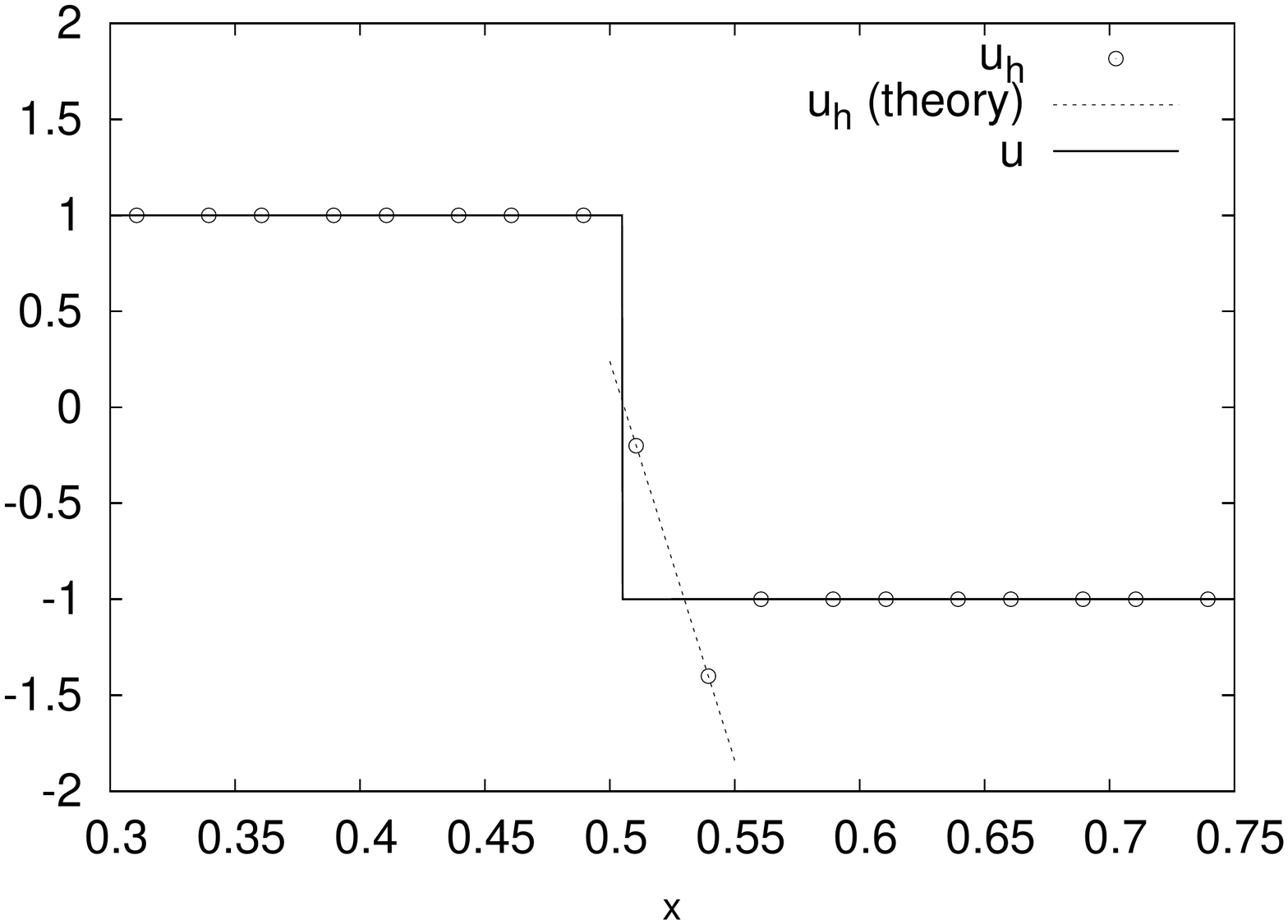 ,width=4.8cm}}\hspace{-0.8cm}
\subfigure[$s_c=-\tfrac{3}{5}$]{\epsfig{figure=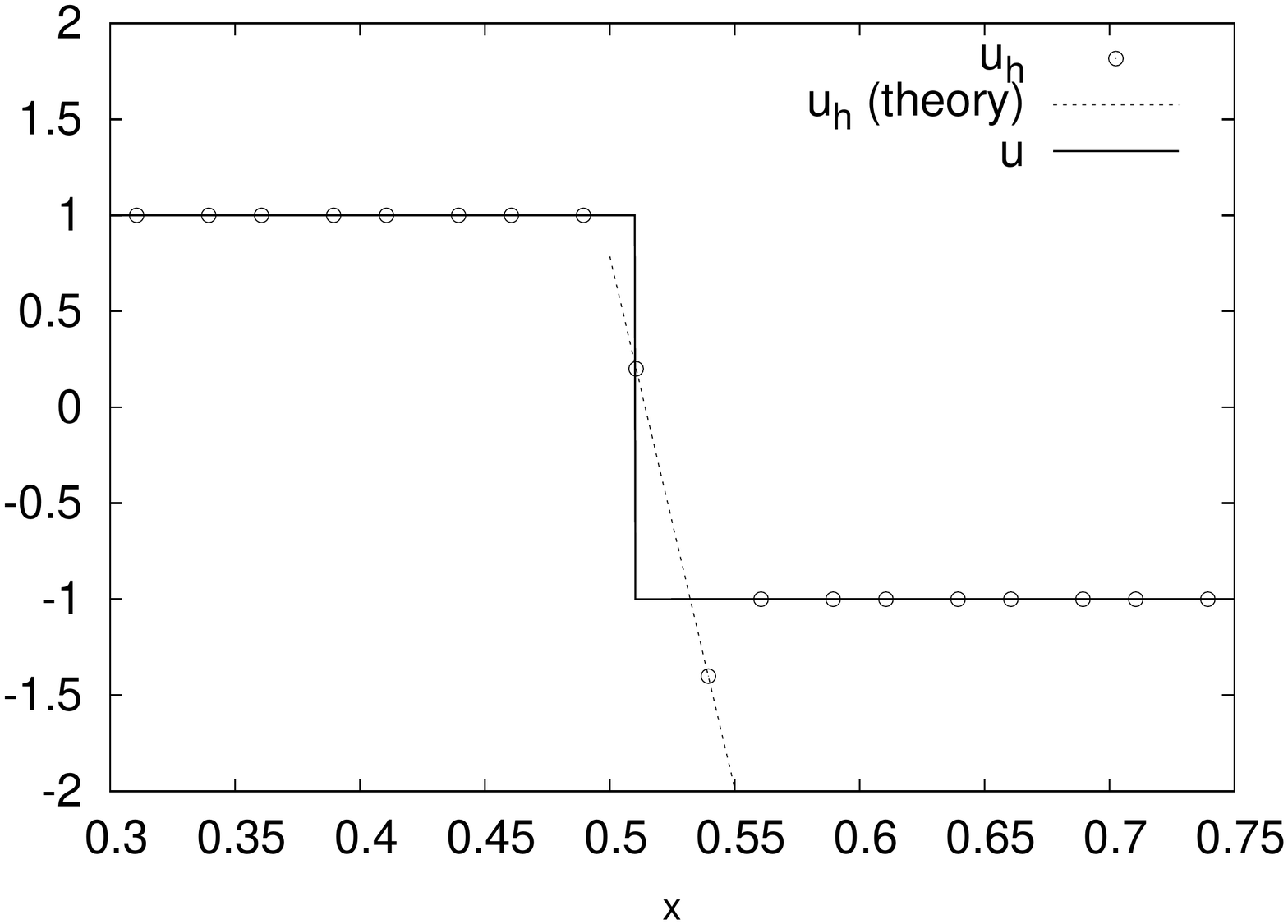 ,width=4.8cm}}\\
\subfigure[$s_c=-\tfrac{2}{5}$]{\epsfig{figure=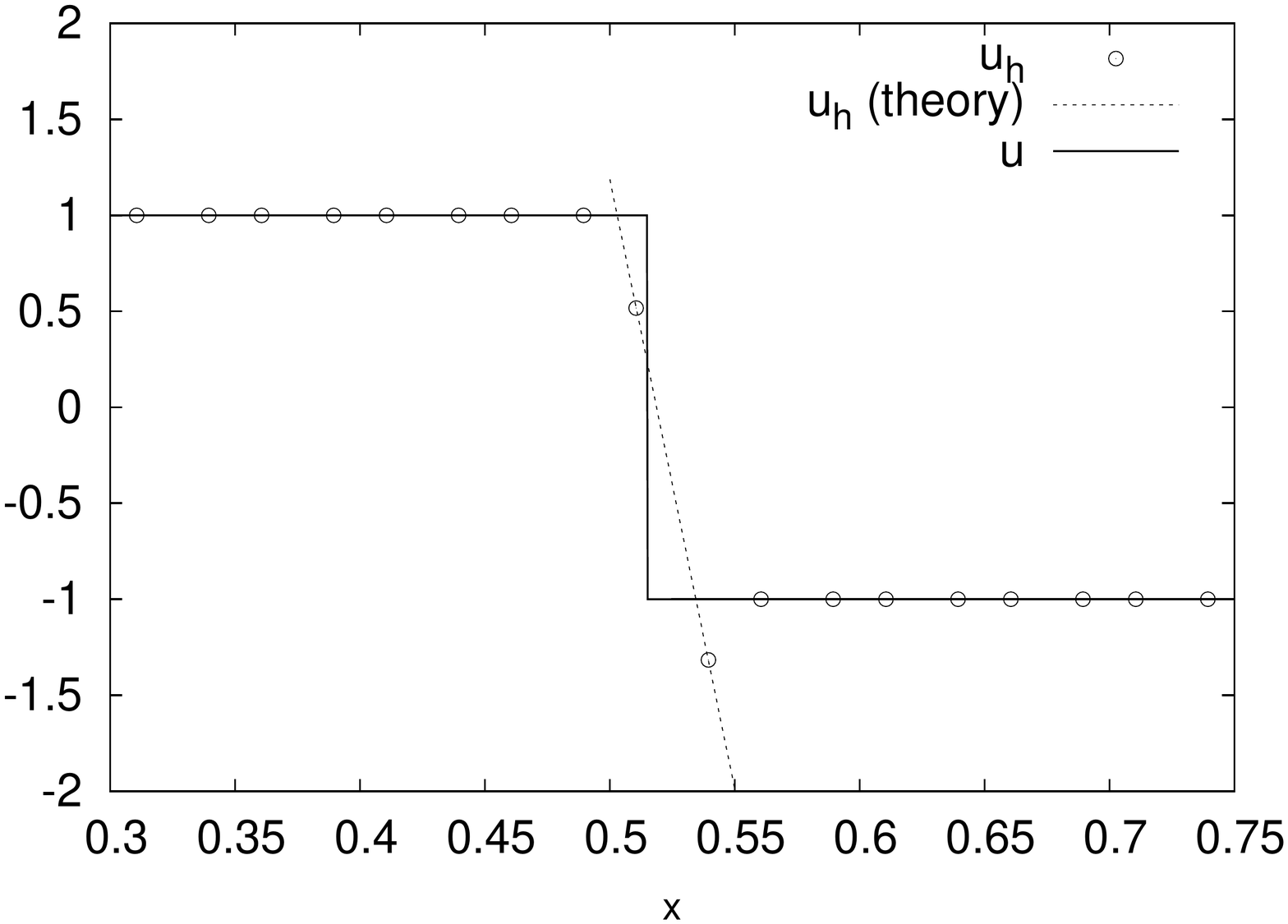 ,width=4.8cm}}\hspace{-0.8cm}
\subfigure[$s_c=0$            ]{\epsfig{figure=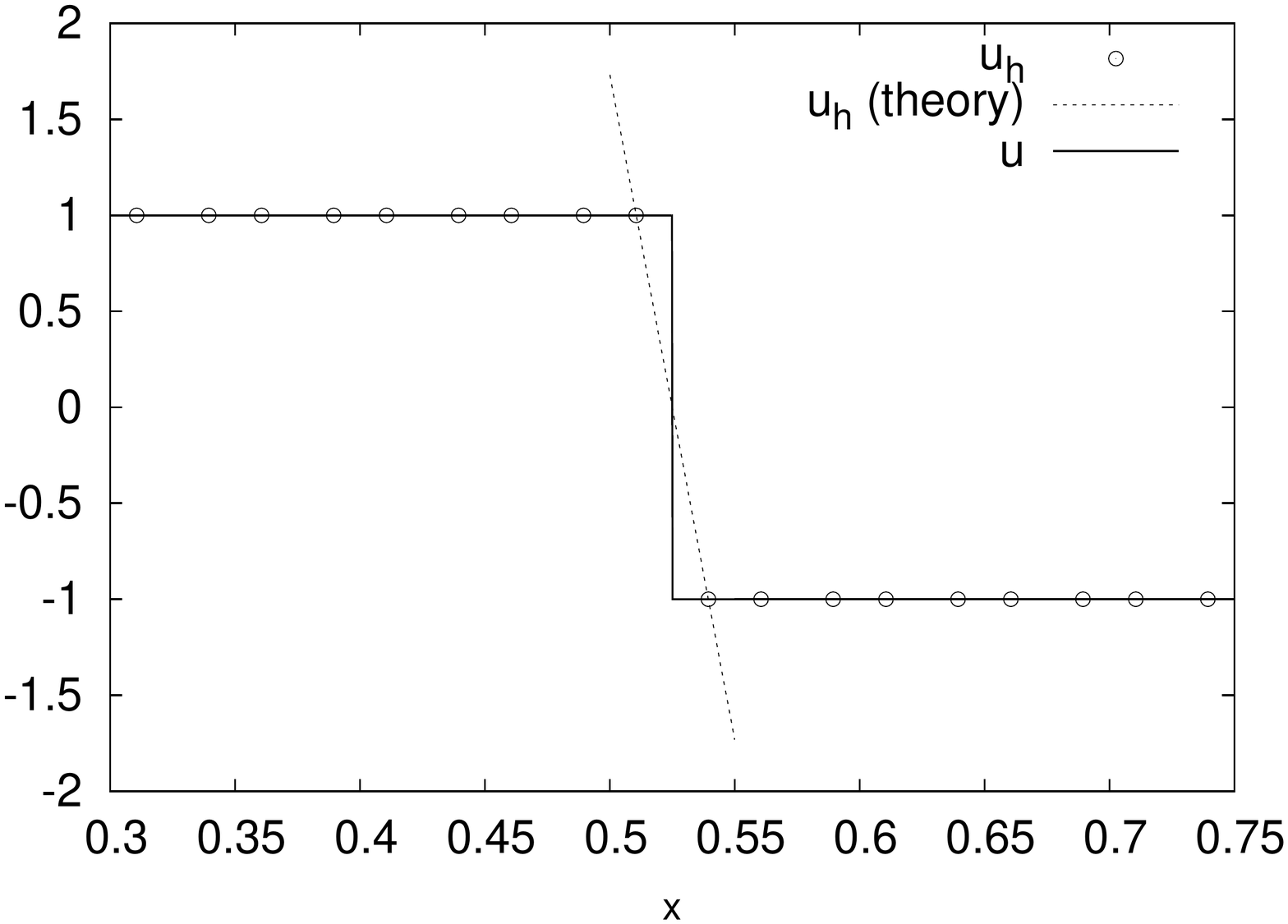 ,width=4.8cm}}\hspace{-0.8cm}
\subfigure[$s_c= \tfrac{2}{5}$]{\epsfig{figure=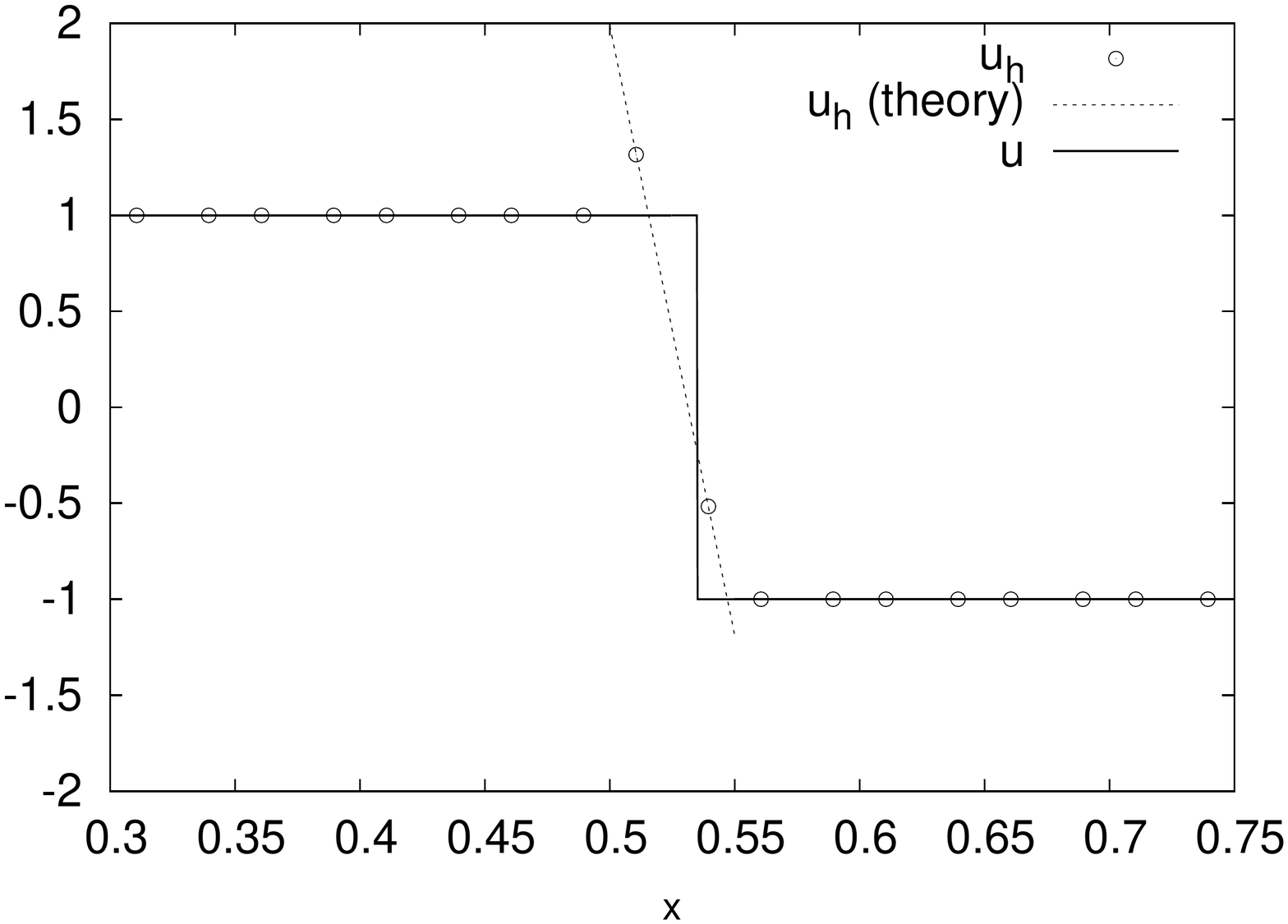 ,width=4.8cm}}\\
\subfigure[$s_c= \tfrac{3}{5}$]{\epsfig{figure=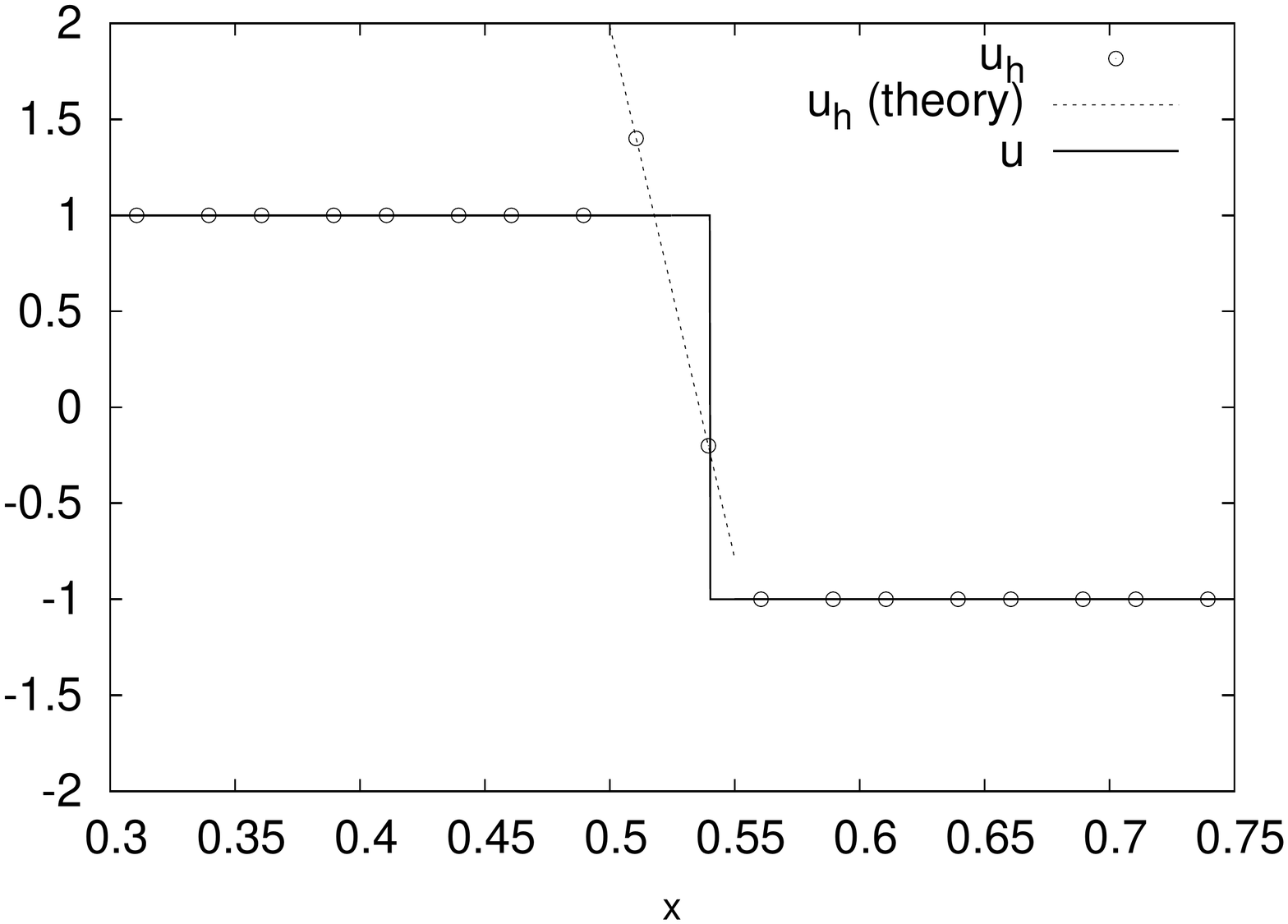 ,width=4.8cm}}\hspace{-0.8cm}
\subfigure[$s_c= \tfrac{4}{5}$]{\epsfig{figure=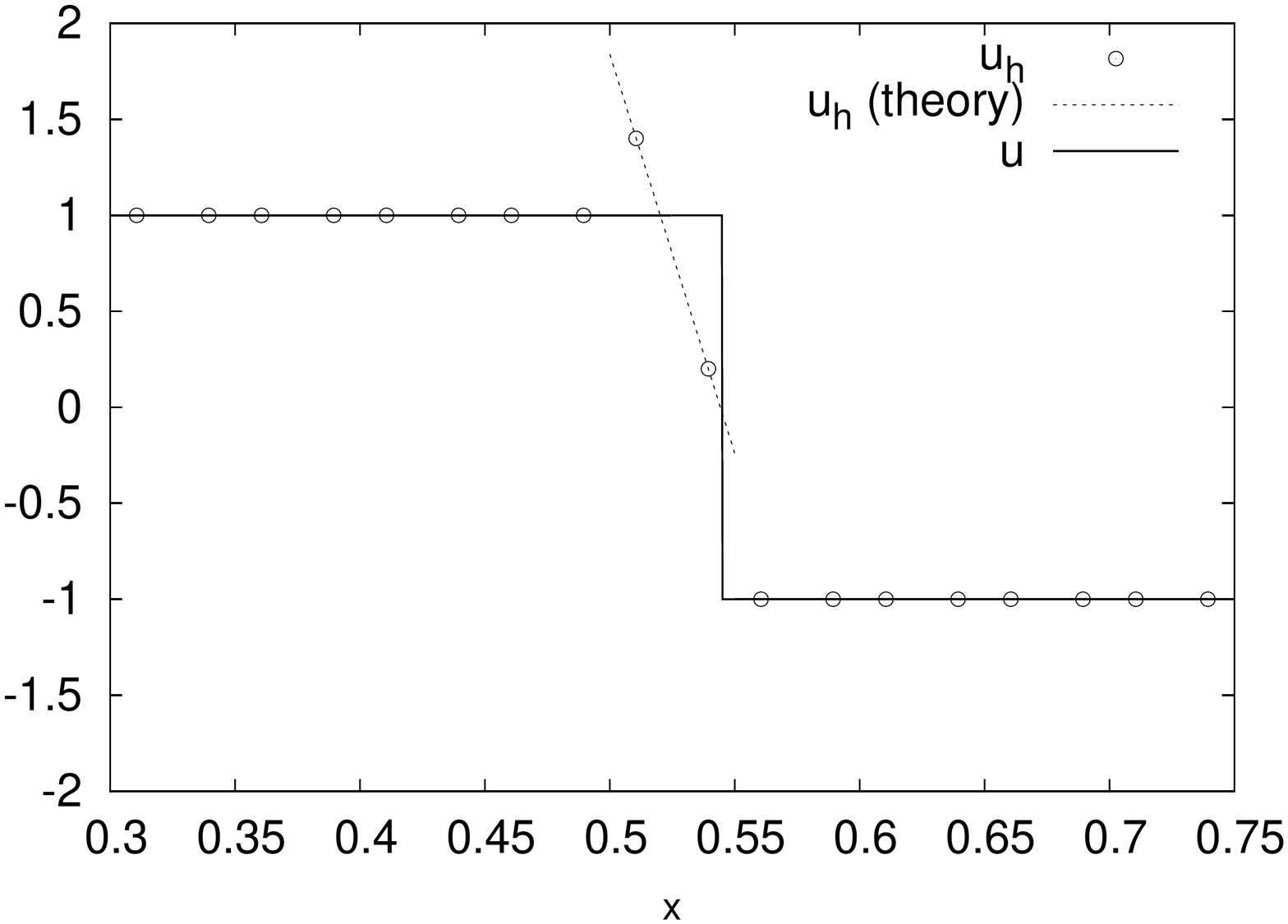 ,width=4.8cm}}\hspace{-0.8cm}
\subfigure[$s_c=1$            ]{\epsfig{figure=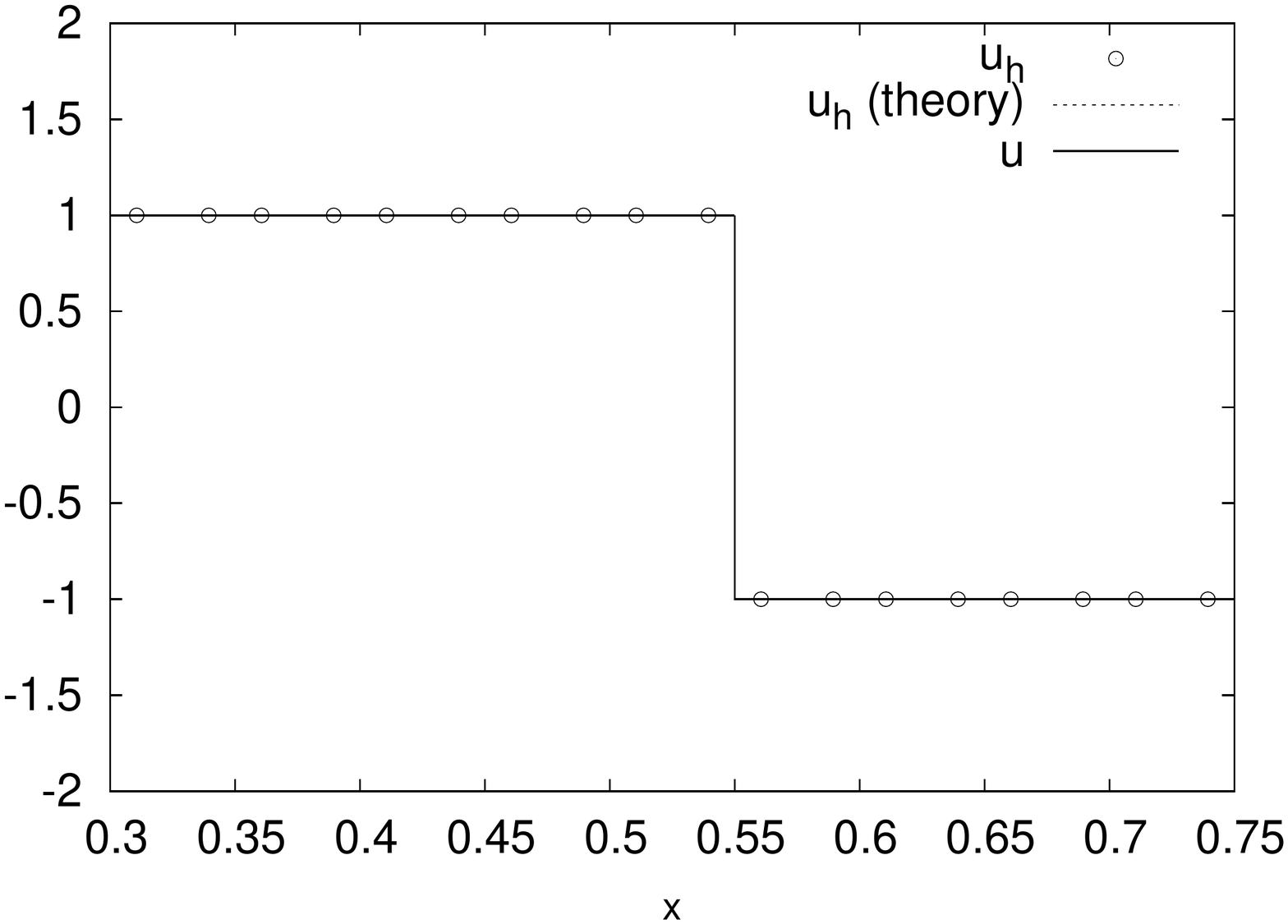 ,width=4.8cm}}
\caption{Steady-state solutions in the shock cell $\kappa_0$ for $p=1$ and the Godunov flux (\ref{eq:godunov_flux}).}
\label{fig:solutions_p1}
\end{center}
\end{figure}

The solution oscillates in the shock region. The theoretical solutions of Theorem~\ref{th:shock_sol_p0123} agree very well with the numerical calculations. The oscillations present amplitude lower than two times the shock strength $u_L-u_R$. Likewise, for $p=2$ and $p=3$ the oscillations of the polynomial solution in the shock cell may lead to nonphysical solutions where the sign of the eigenvalues $f'(u_h)=u_h$ changes sign locally more than one time.

\begin{figure}
\begin{center}
\subfigure[$s_c=-1$           ]{\epsfig{figure=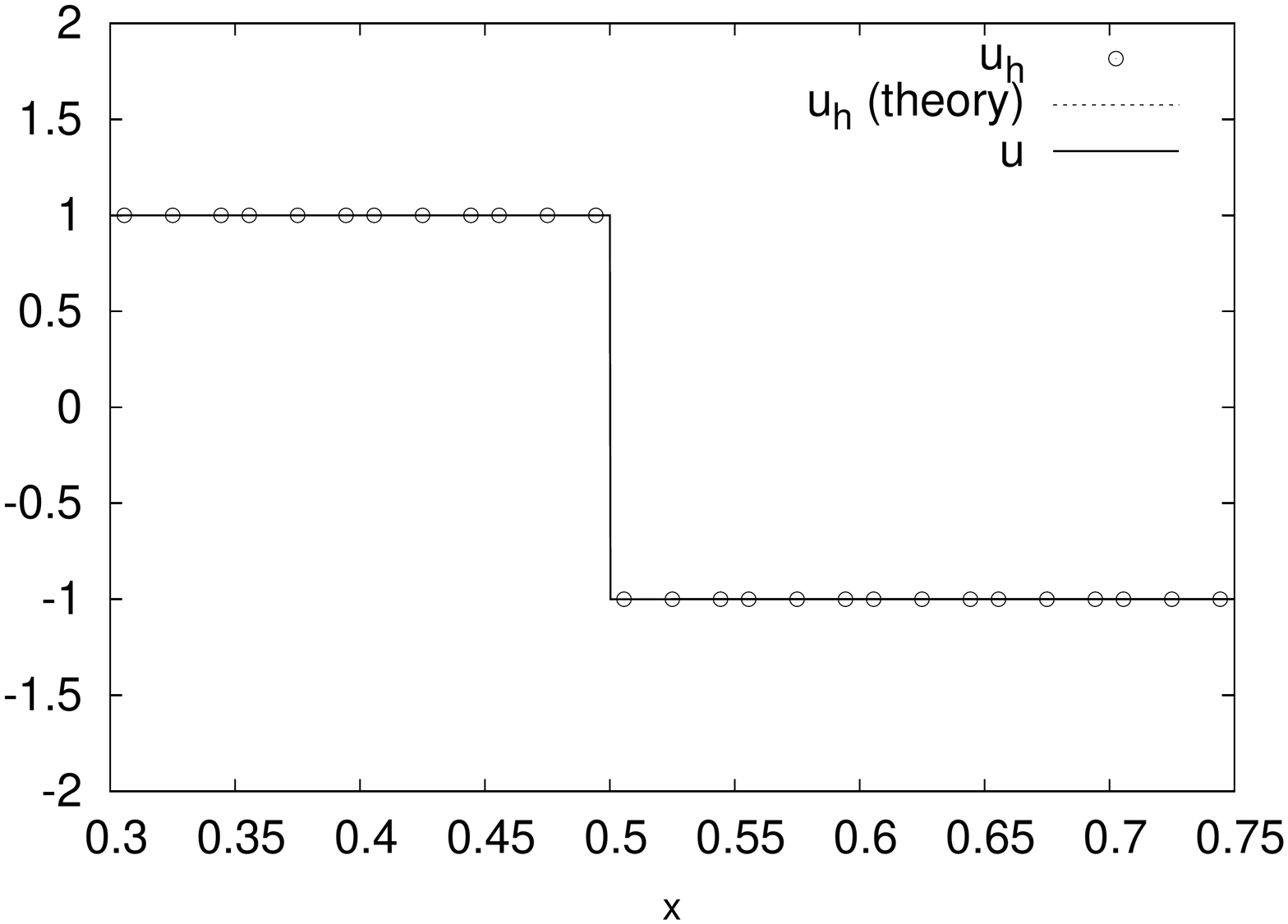 ,width=4.8cm}}\hspace{-0.8cm}
\subfigure[$s_c=-\tfrac{4}{5}$]{\epsfig{figure=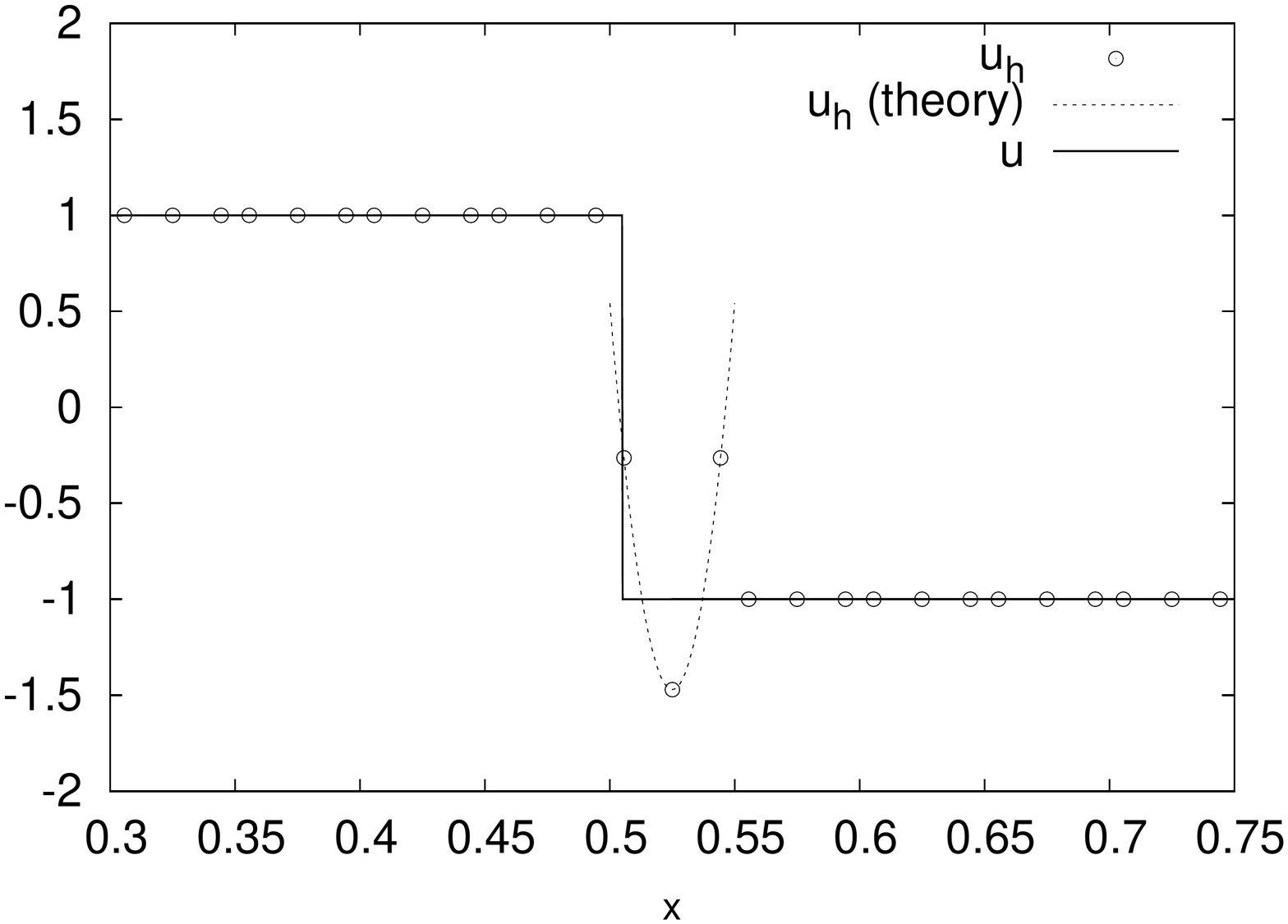 ,width=4.8cm}}\hspace{-0.8cm}
\subfigure[$s_c=-\tfrac{3}{5}$]{\epsfig{figure=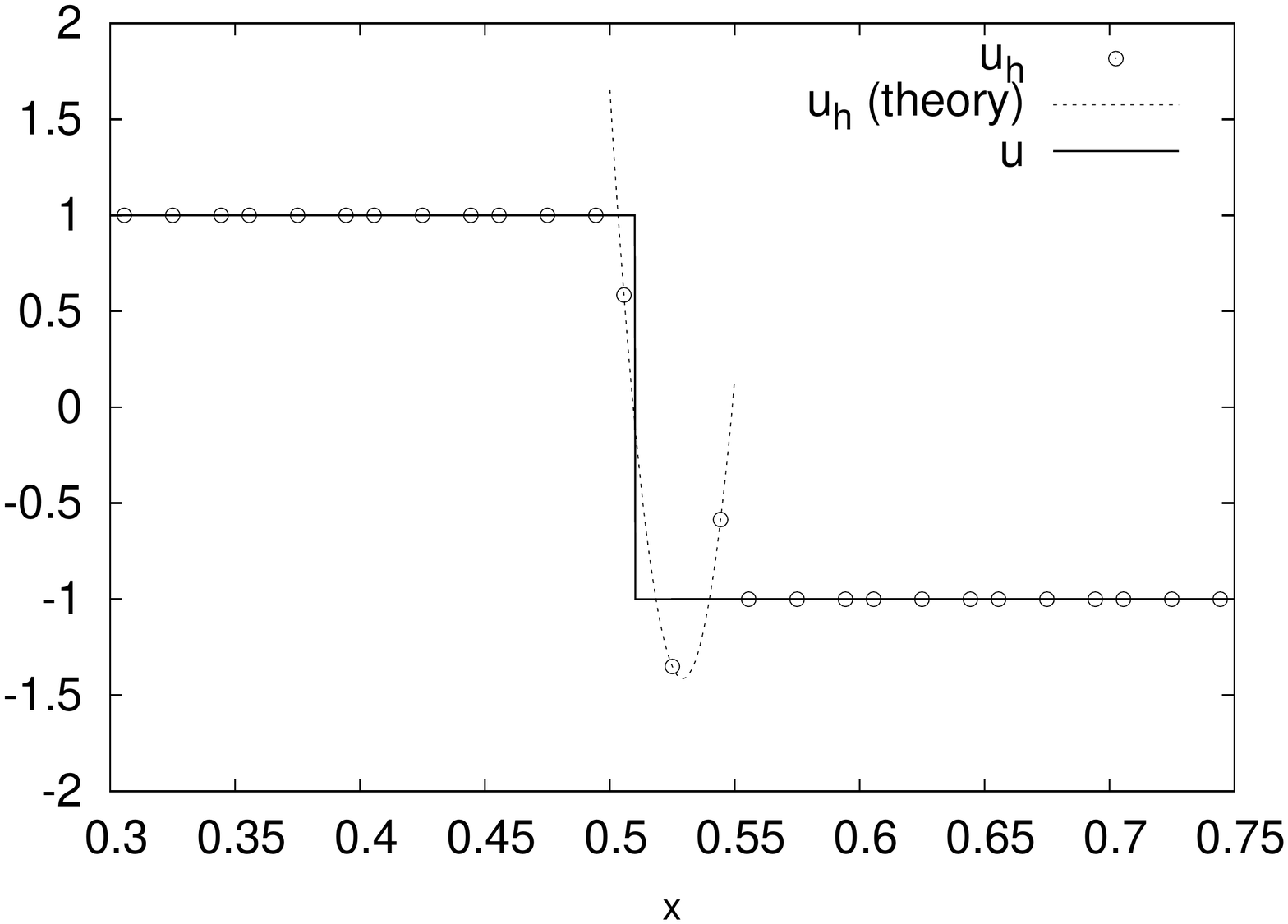 ,width=4.8cm}}\\
\subfigure[$s_c=-\tfrac{2}{5}$]{\epsfig{figure=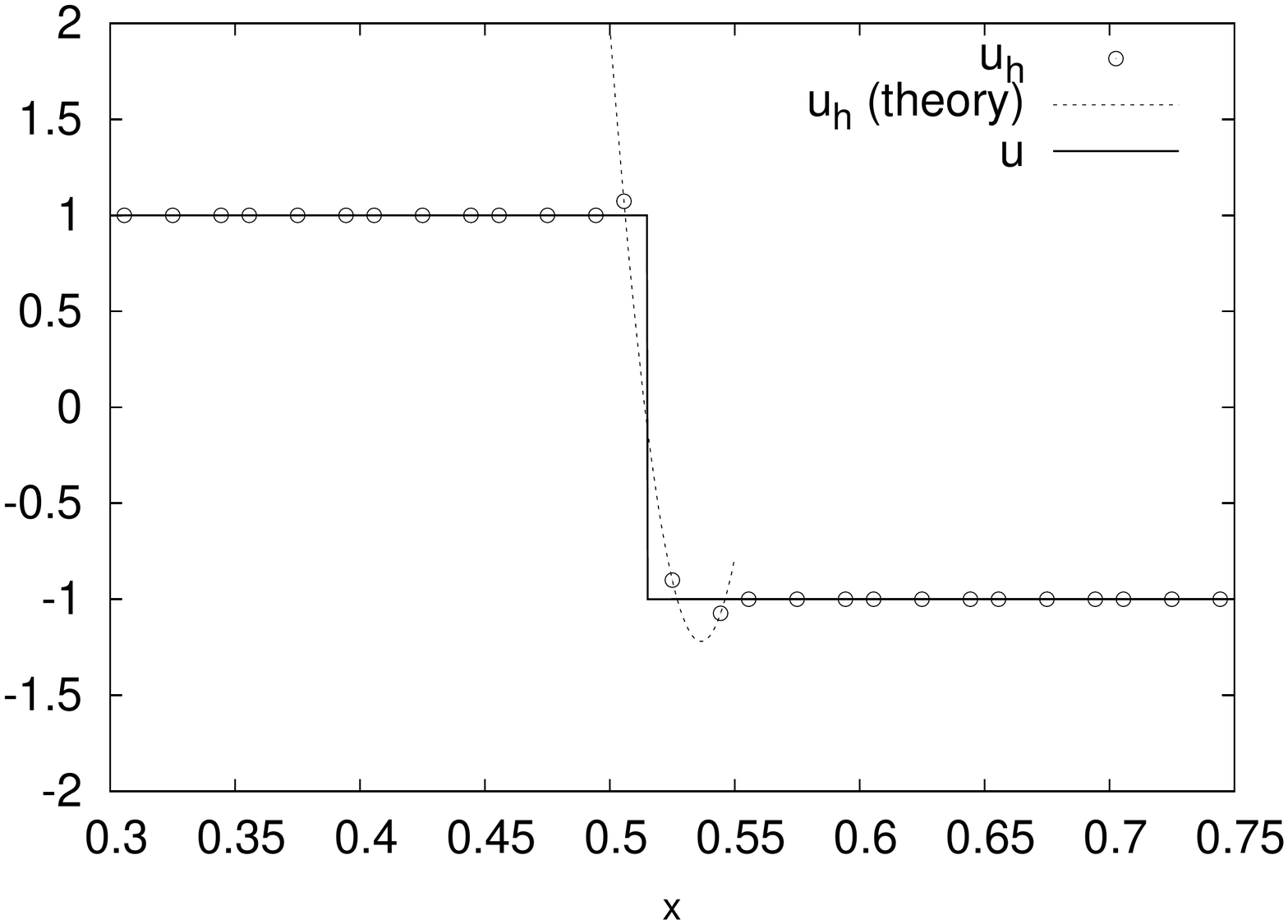 ,width=4.8cm}}\hspace{-0.8cm}
\subfigure[$s_c=0$            ]{\epsfig{figure=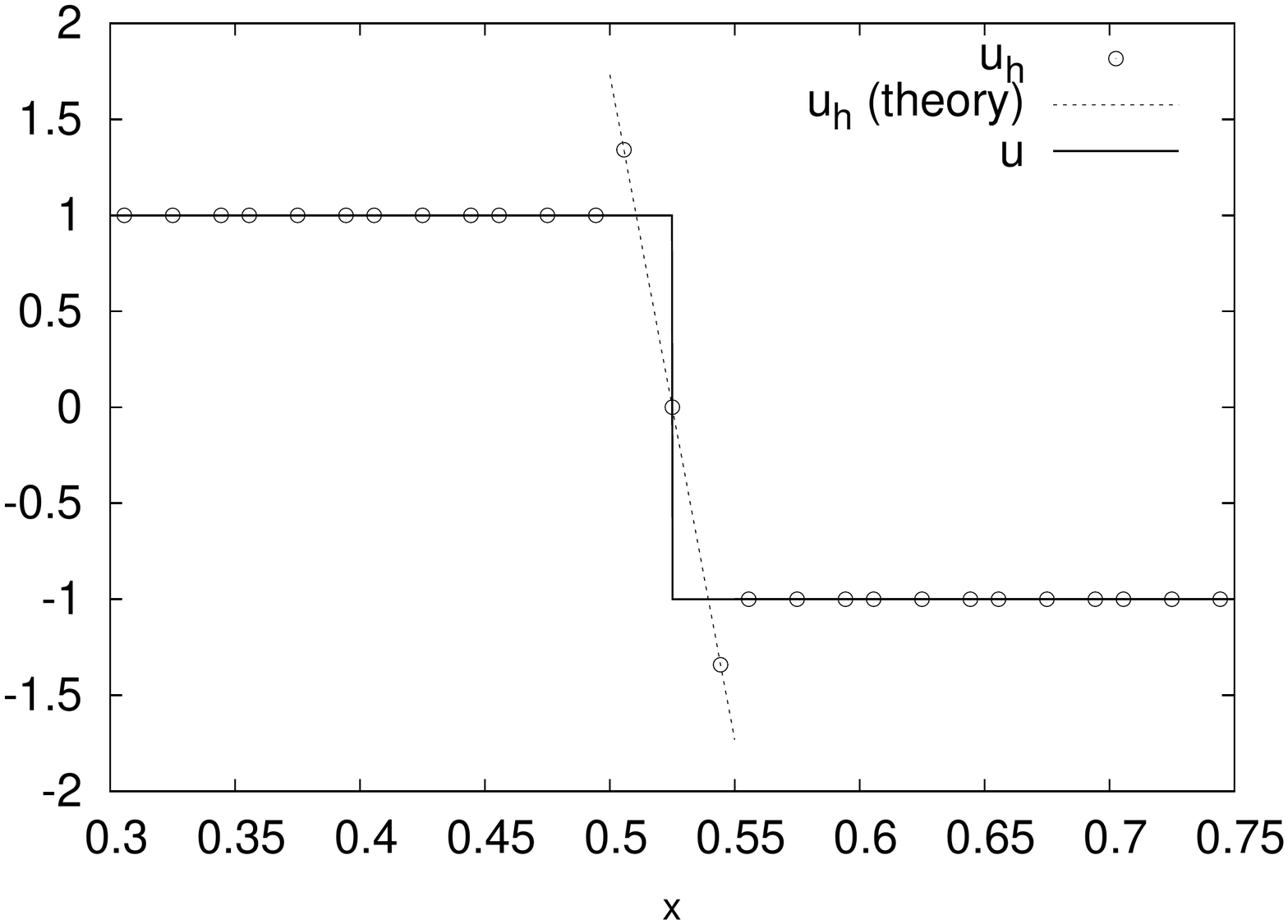 ,width=4.8cm}}\hspace{-0.8cm}
\subfigure[$s_c= \tfrac{2}{5}$]{\epsfig{figure=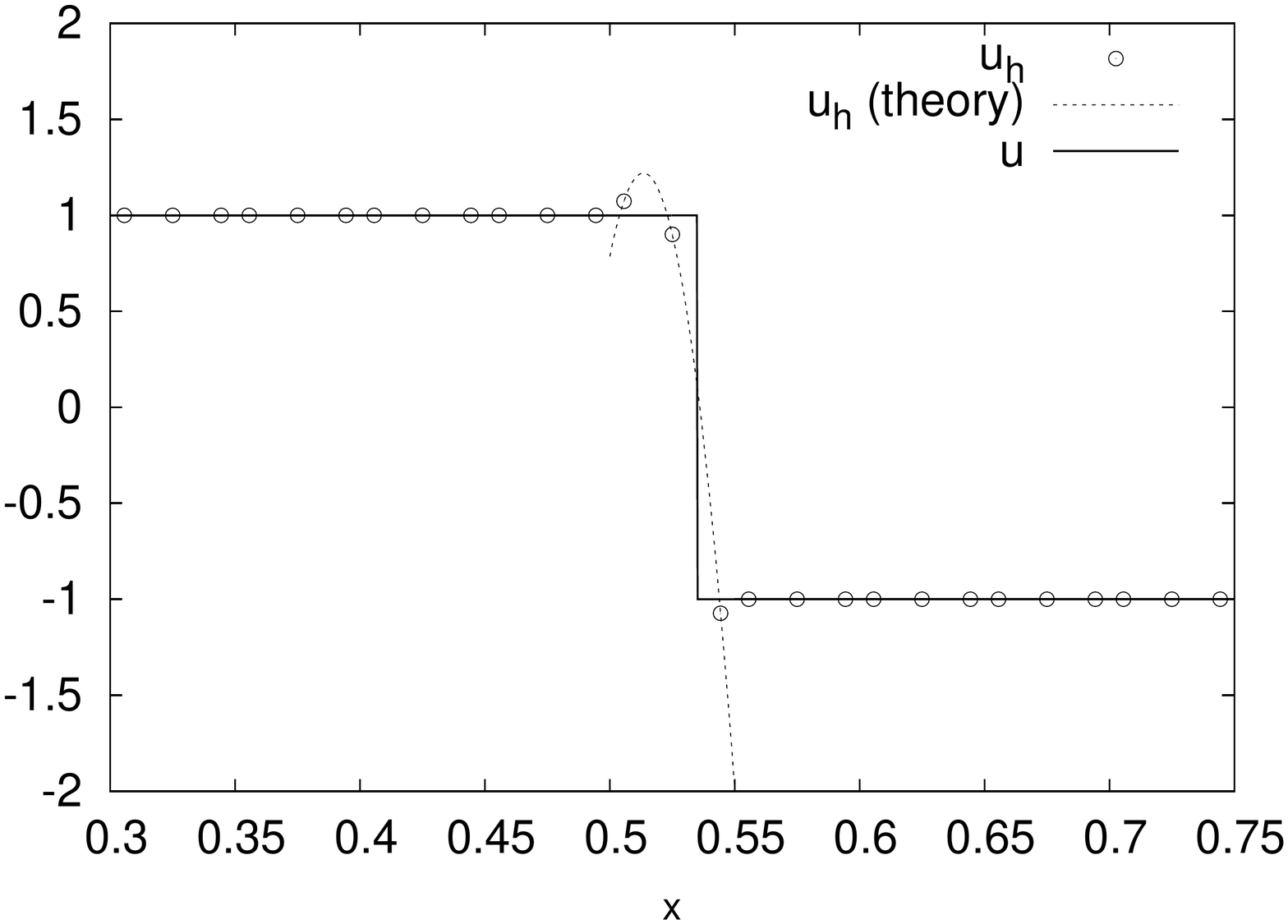 ,width=4.8cm}}\\
\subfigure[$s_c= \tfrac{3}{5}$]{\epsfig{figure=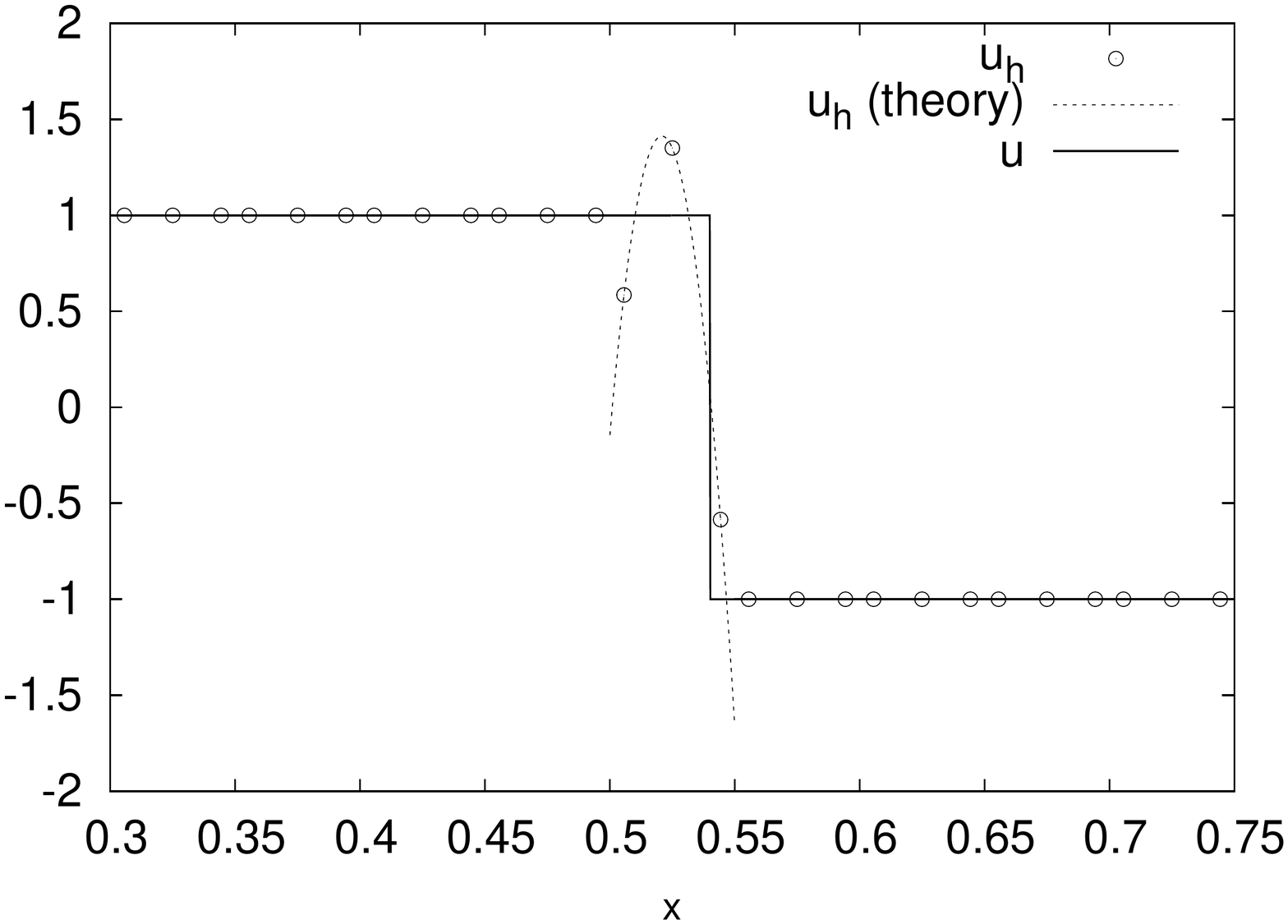 ,width=4.8cm}}\hspace{-0.8cm}
\subfigure[$s_c= \tfrac{4}{5}$]{\epsfig{figure=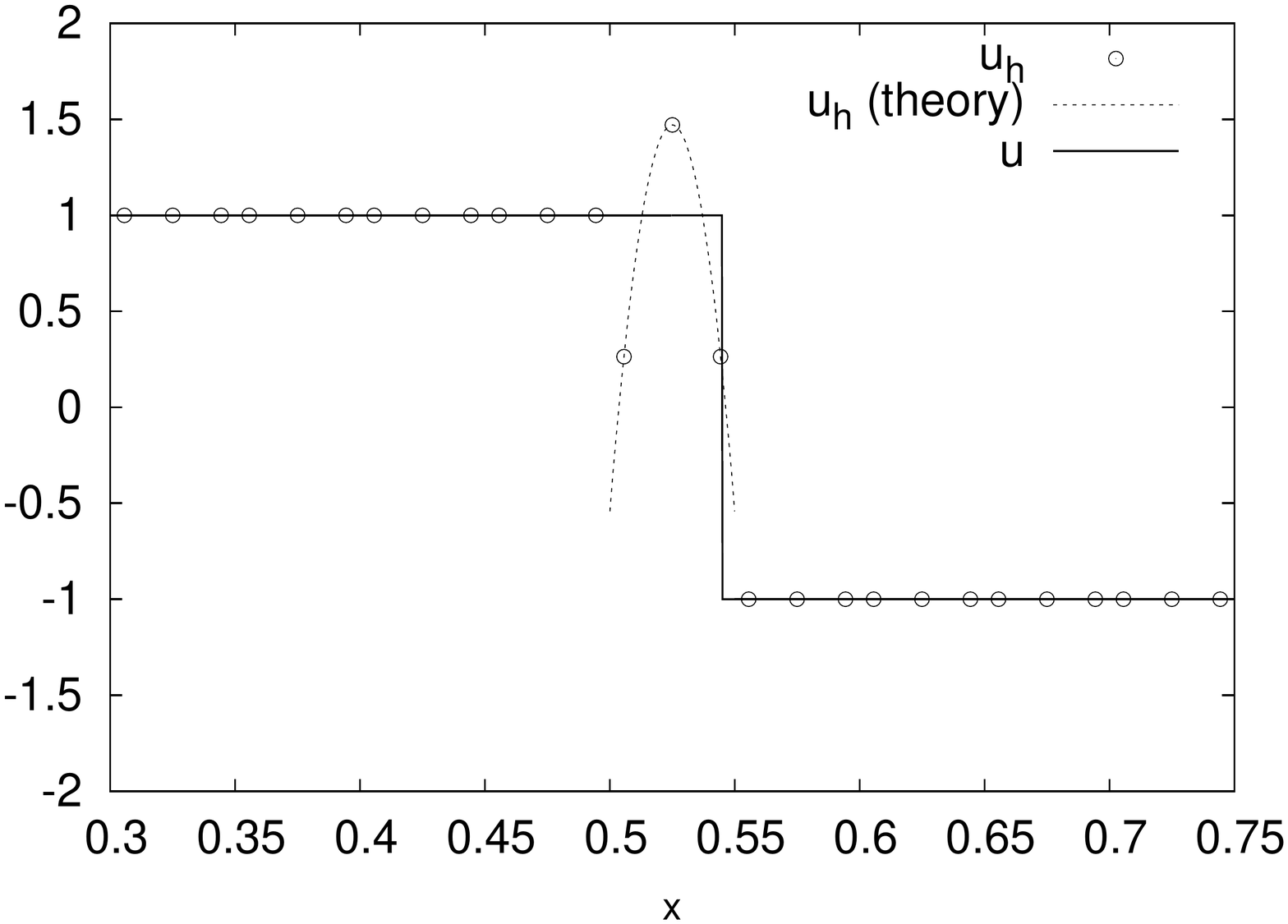 ,width=4.8cm}}\hspace{-0.8cm}
\subfigure[$s_c=1$            ]{\epsfig{figure=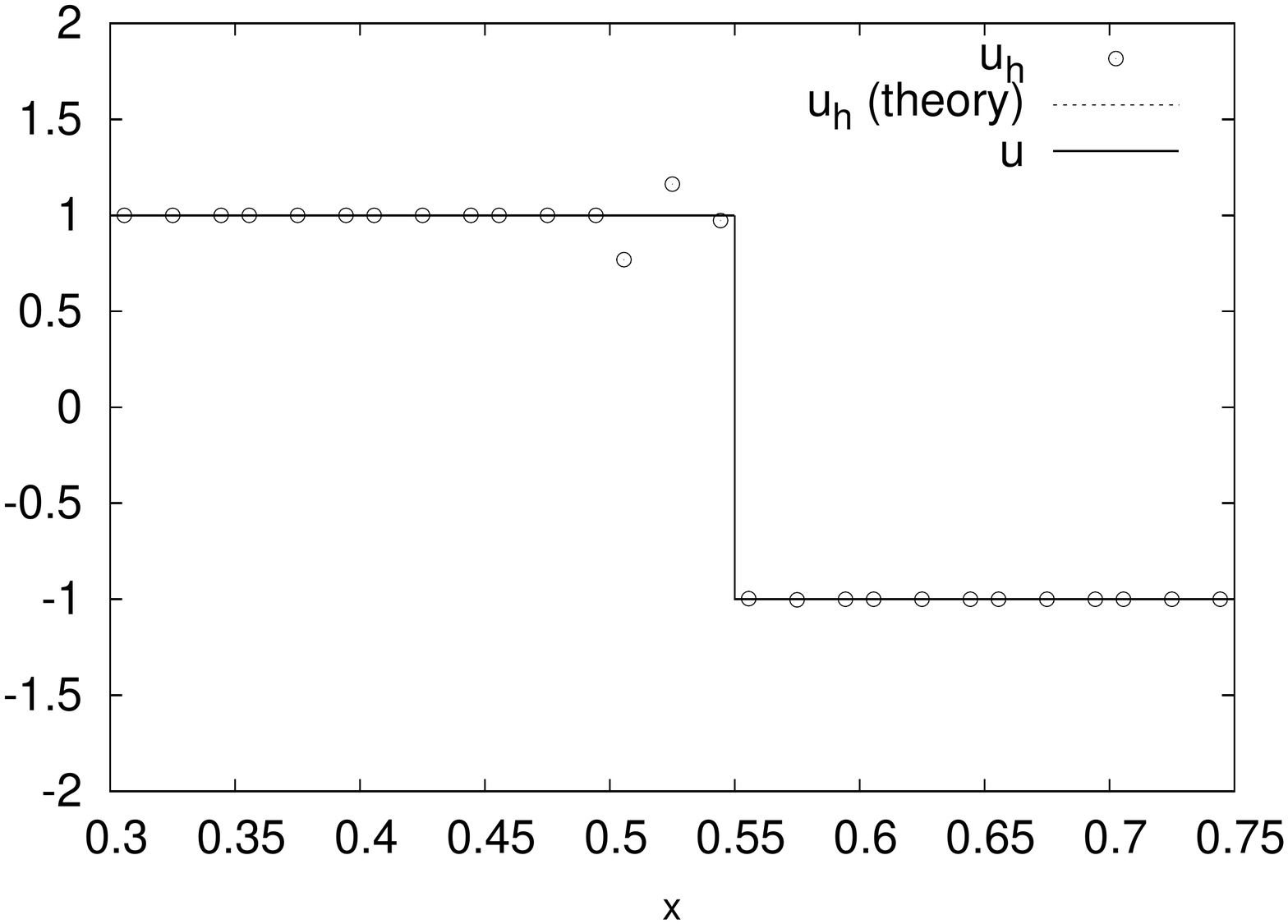 ,width=4.8cm}}
\caption{Steady-state solutions in the shock cell $\kappa_0$ for $p=2$ and the Godunov flux (\ref{eq:godunov_flux}).}
\label{fig:solutions_p2}
\end{center}
\end{figure}

\begin{figure}
\begin{center}
\subfigure[$s_c=-1$           ]{\epsfig{figure=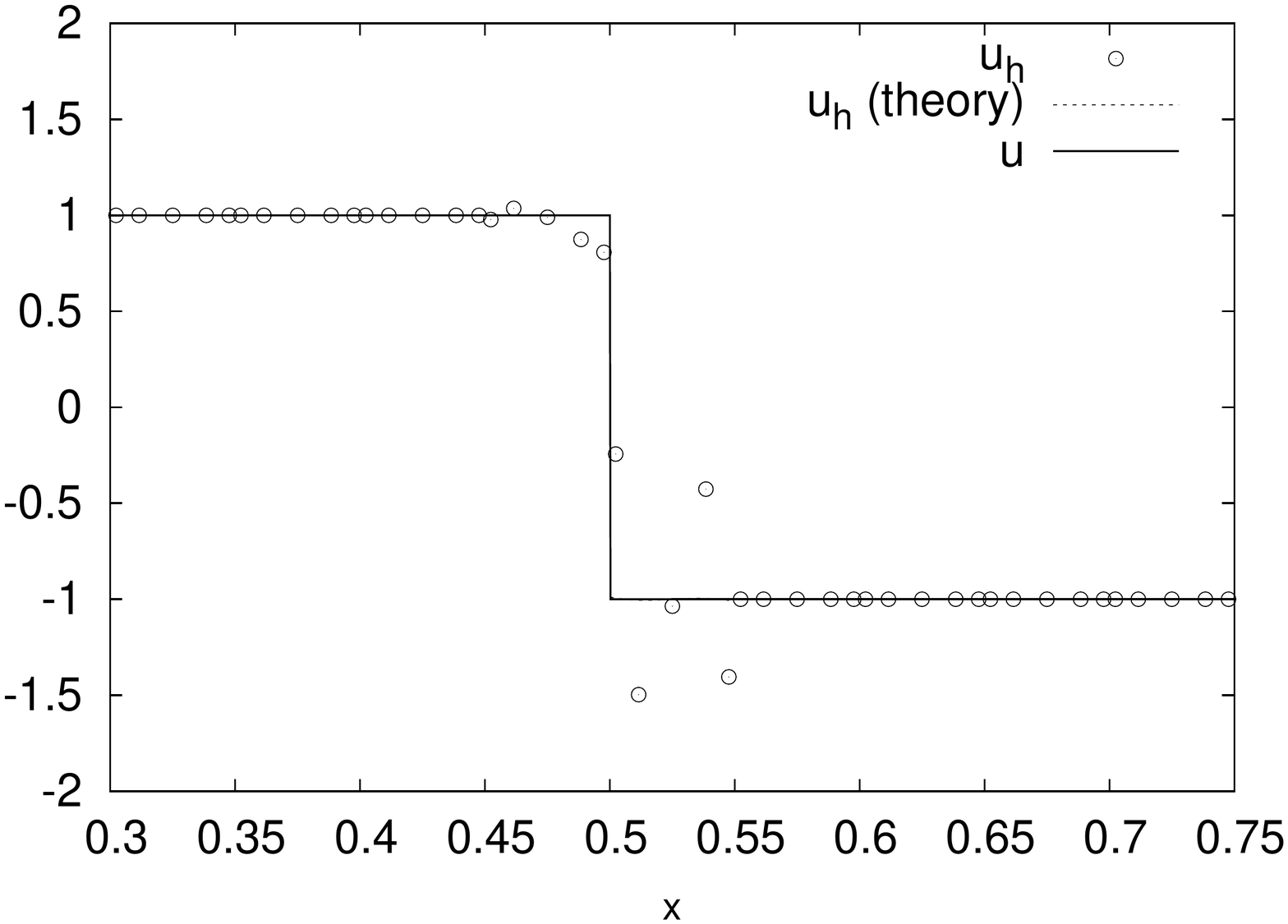 ,width=4.8cm}}\hspace{-0.8cm}
\subfigure[$s_c=-\tfrac{4}{5}$]{\epsfig{figure=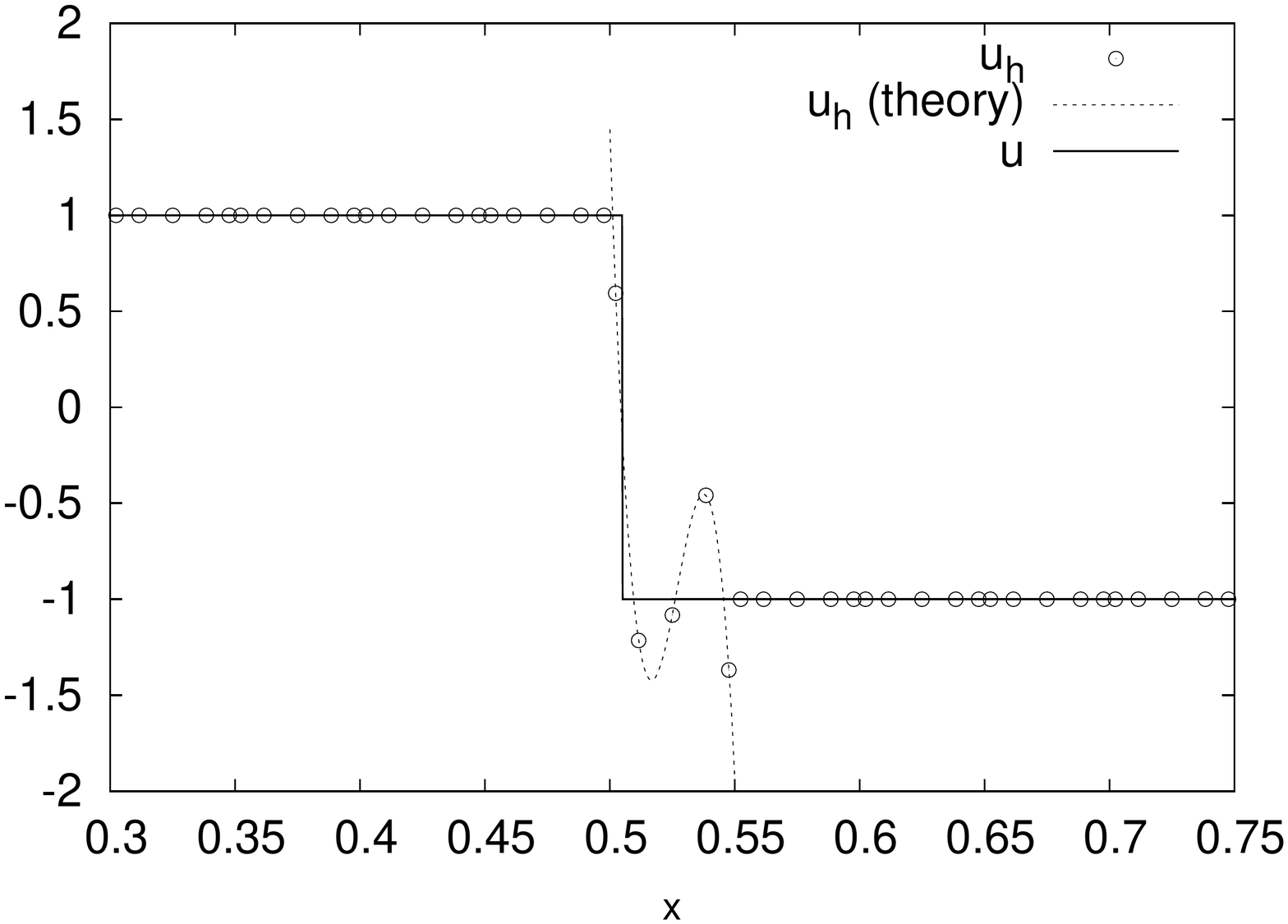 ,width=4.8cm}}\hspace{-0.8cm}
\subfigure[$s_c=-\tfrac{3}{5}$]{\epsfig{figure=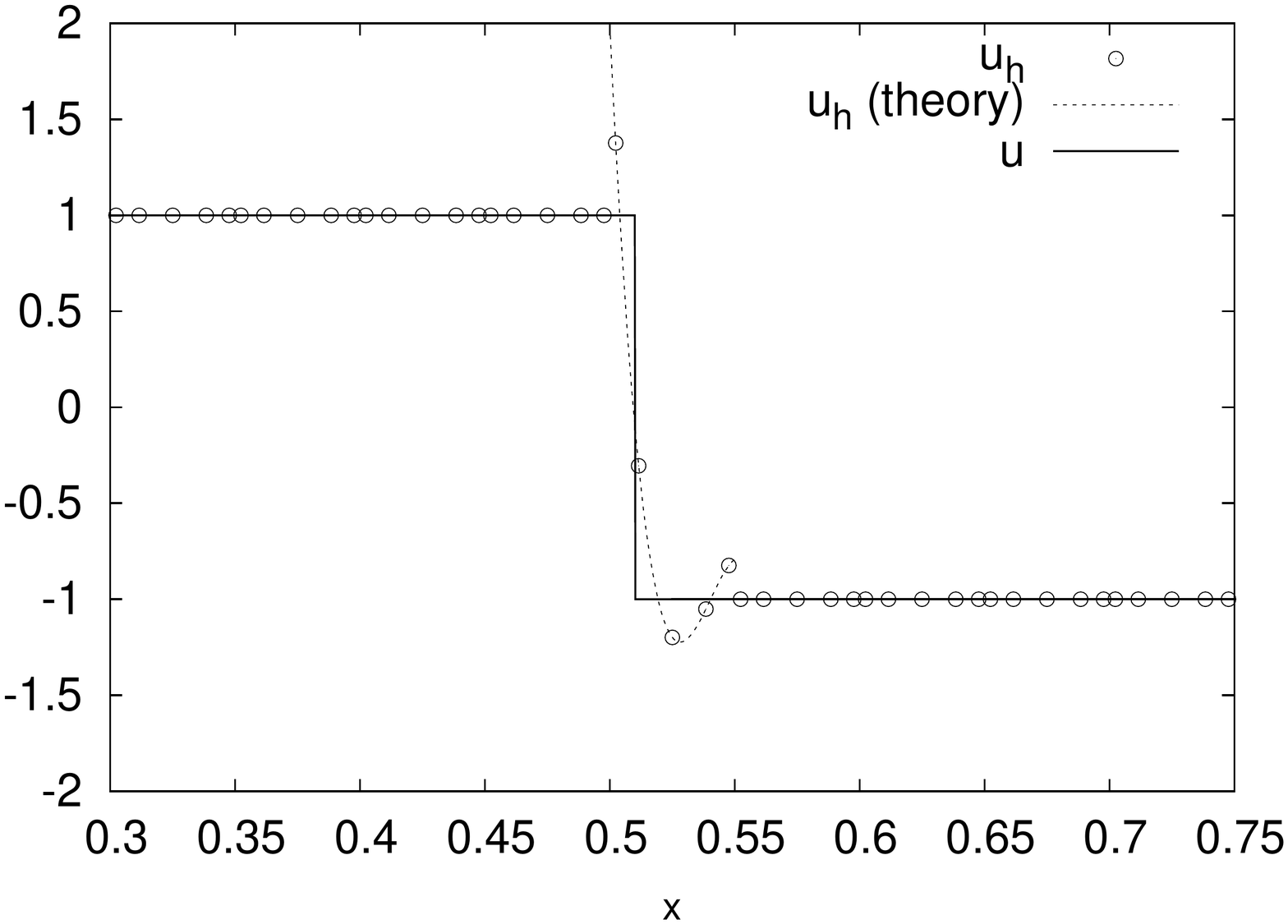 ,width=4.8cm}}\\
\subfigure[$s_c=-\tfrac{2}{5}$]{\epsfig{figure=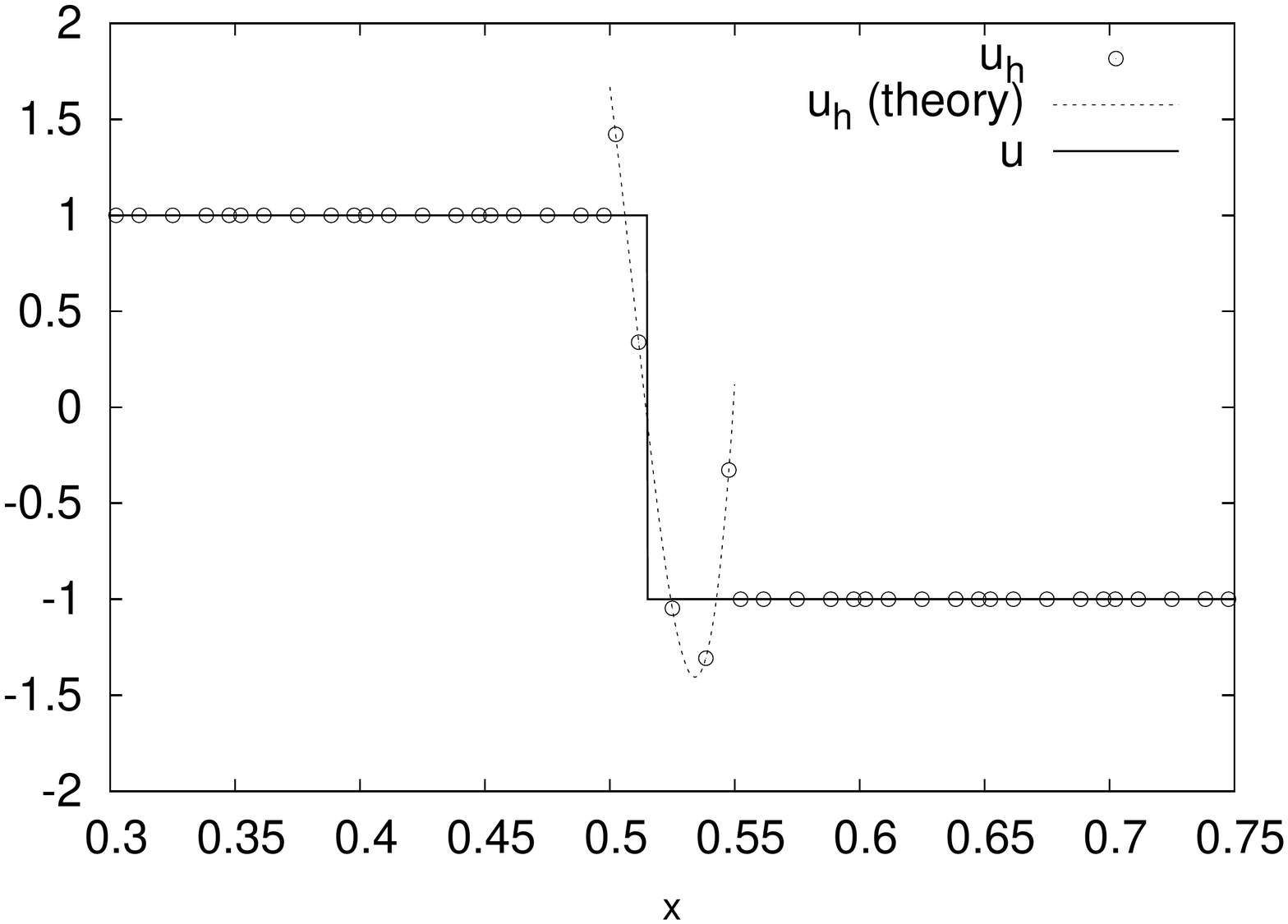 ,width=4.8cm}}\hspace{-0.8cm}
\subfigure[$s_c=0$            ]{\epsfig{figure=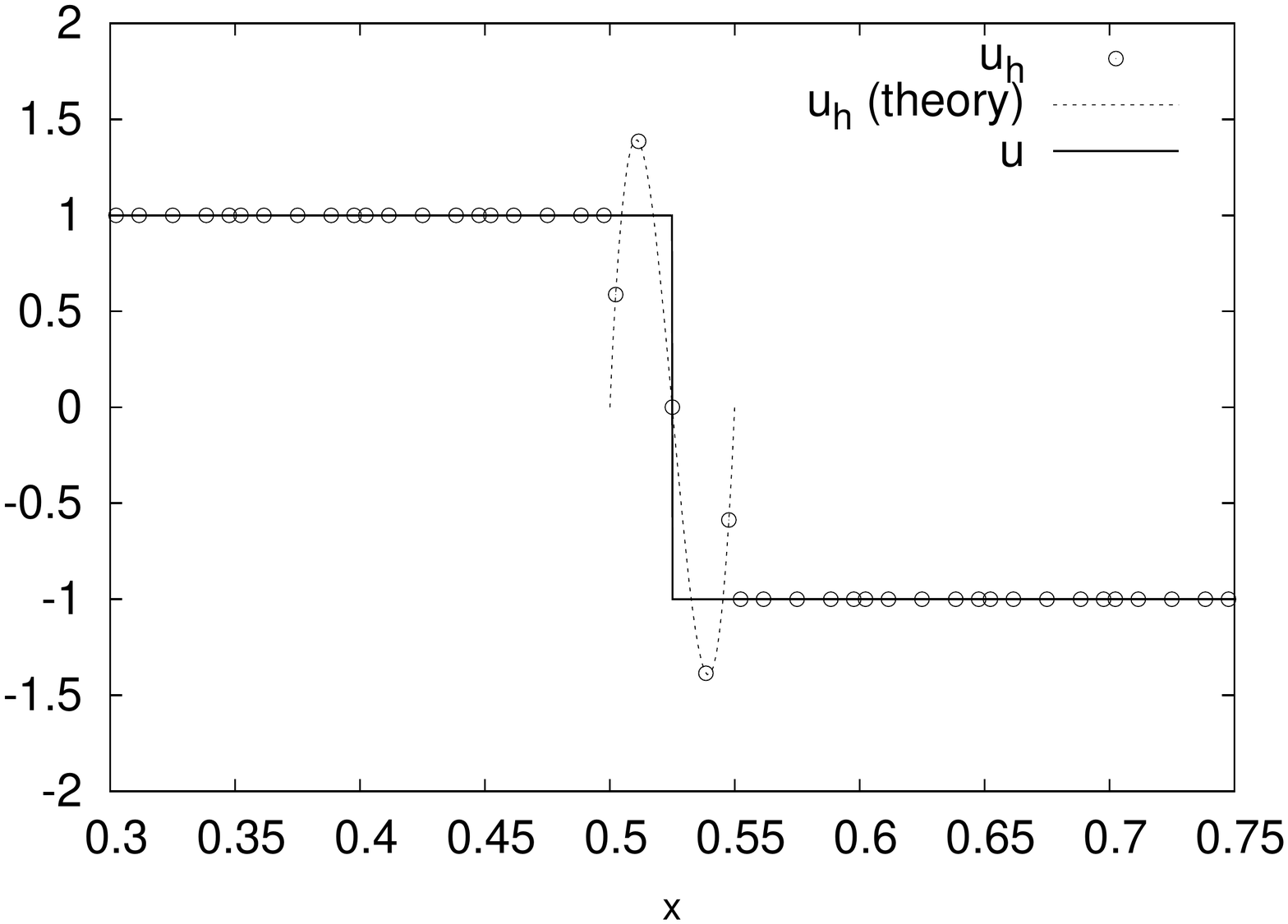 ,width=4.8cm}}\hspace{-0.8cm}
\subfigure[$s_c= \tfrac{2}{5}$]{\epsfig{figure=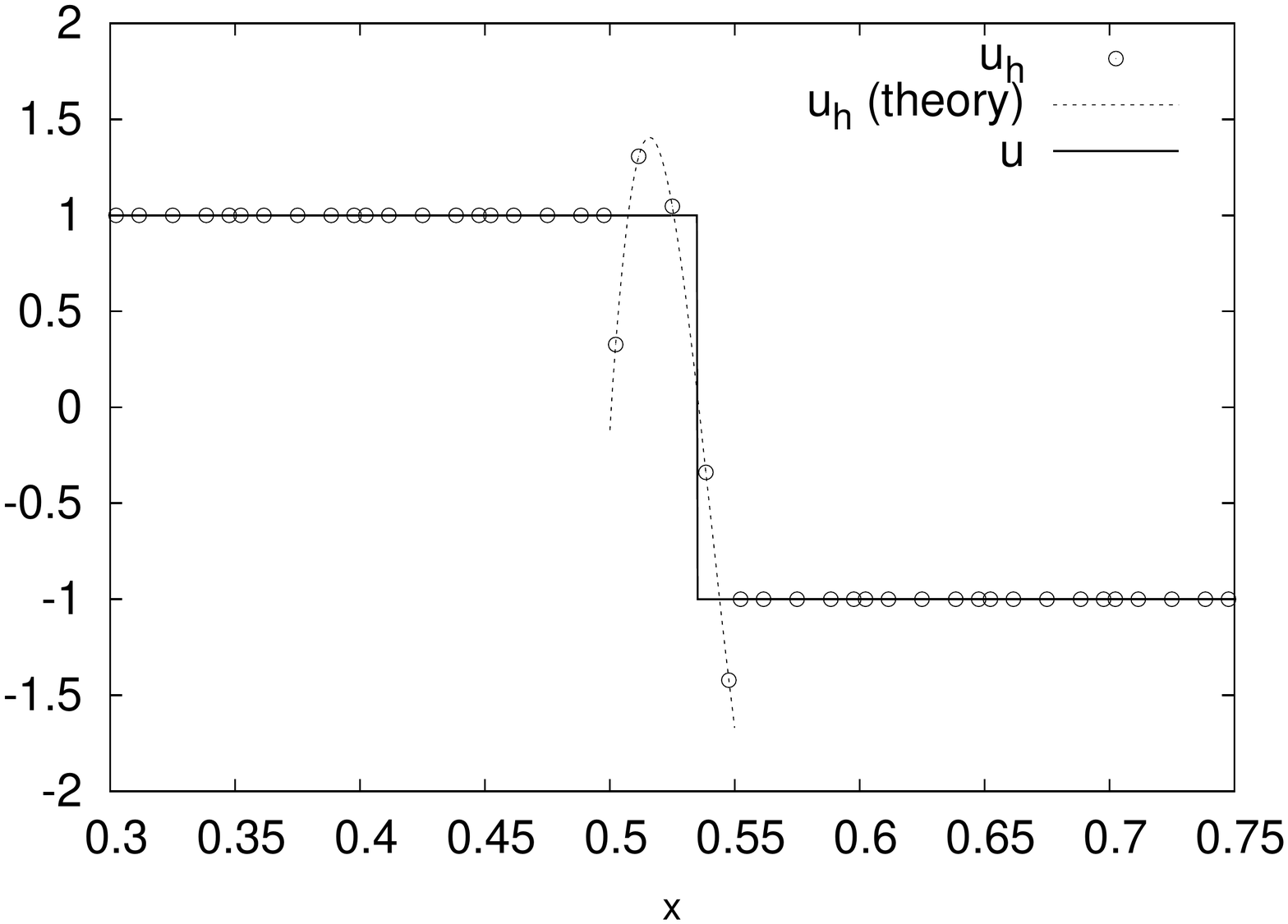 ,width=4.8cm}}\\
\subfigure[$s_c= \tfrac{3}{5}$]{\epsfig{figure=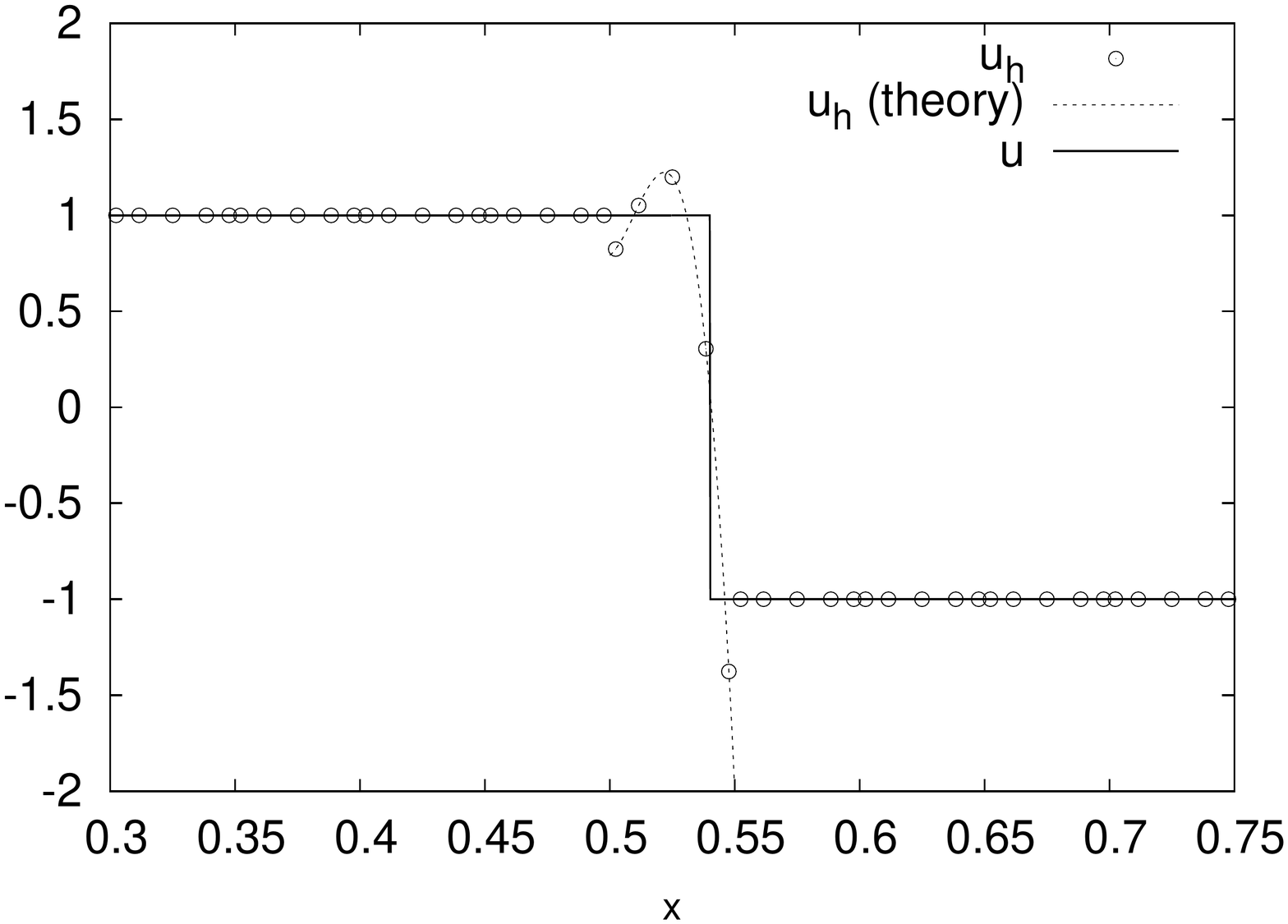 ,width=4.8cm}}\hspace{-0.8cm}
\subfigure[$s_c= \tfrac{4}{5}$]{\epsfig{figure=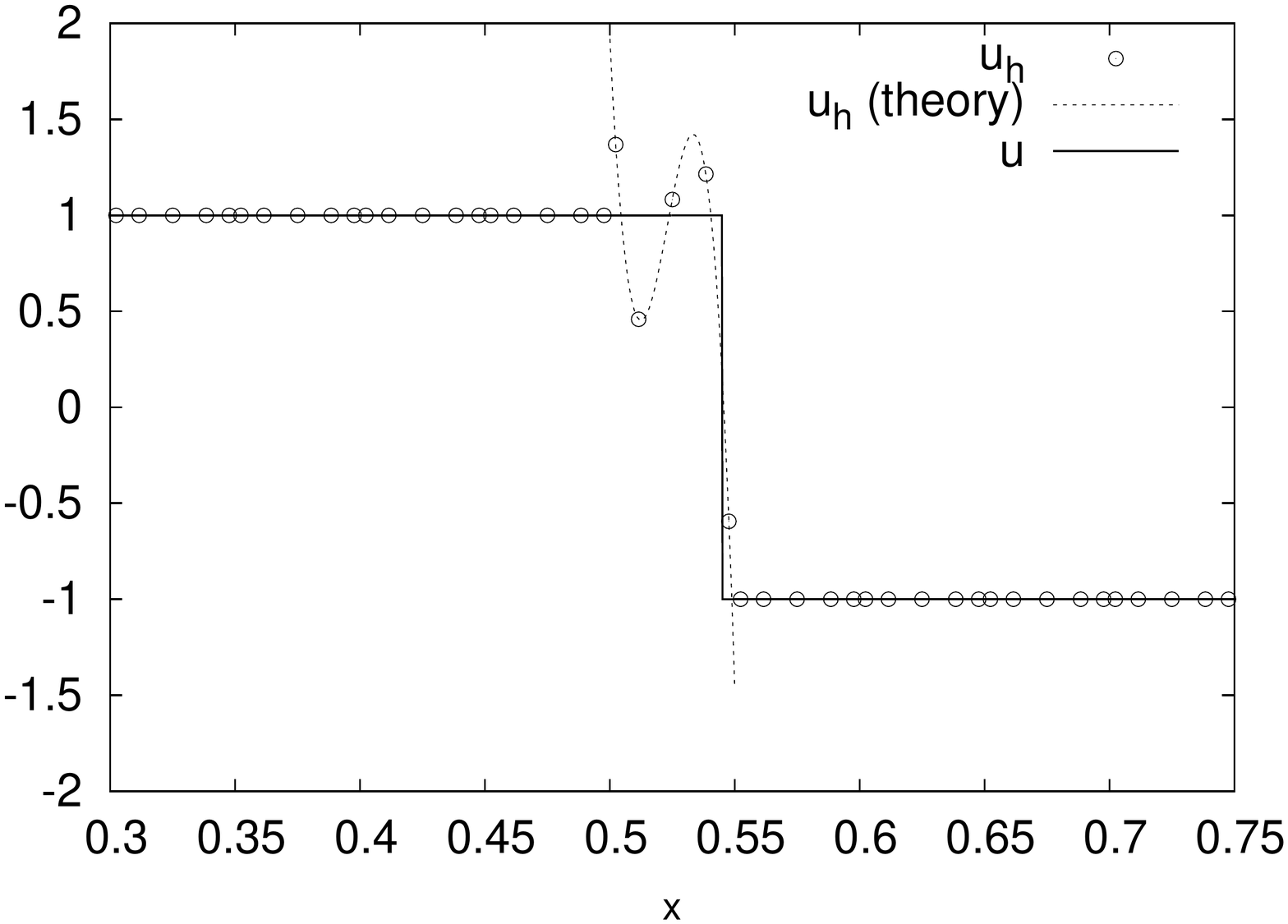 ,width=4.8cm}}\hspace{-0.8cm}
\subfigure[$s_c=1$            ]{\epsfig{figure=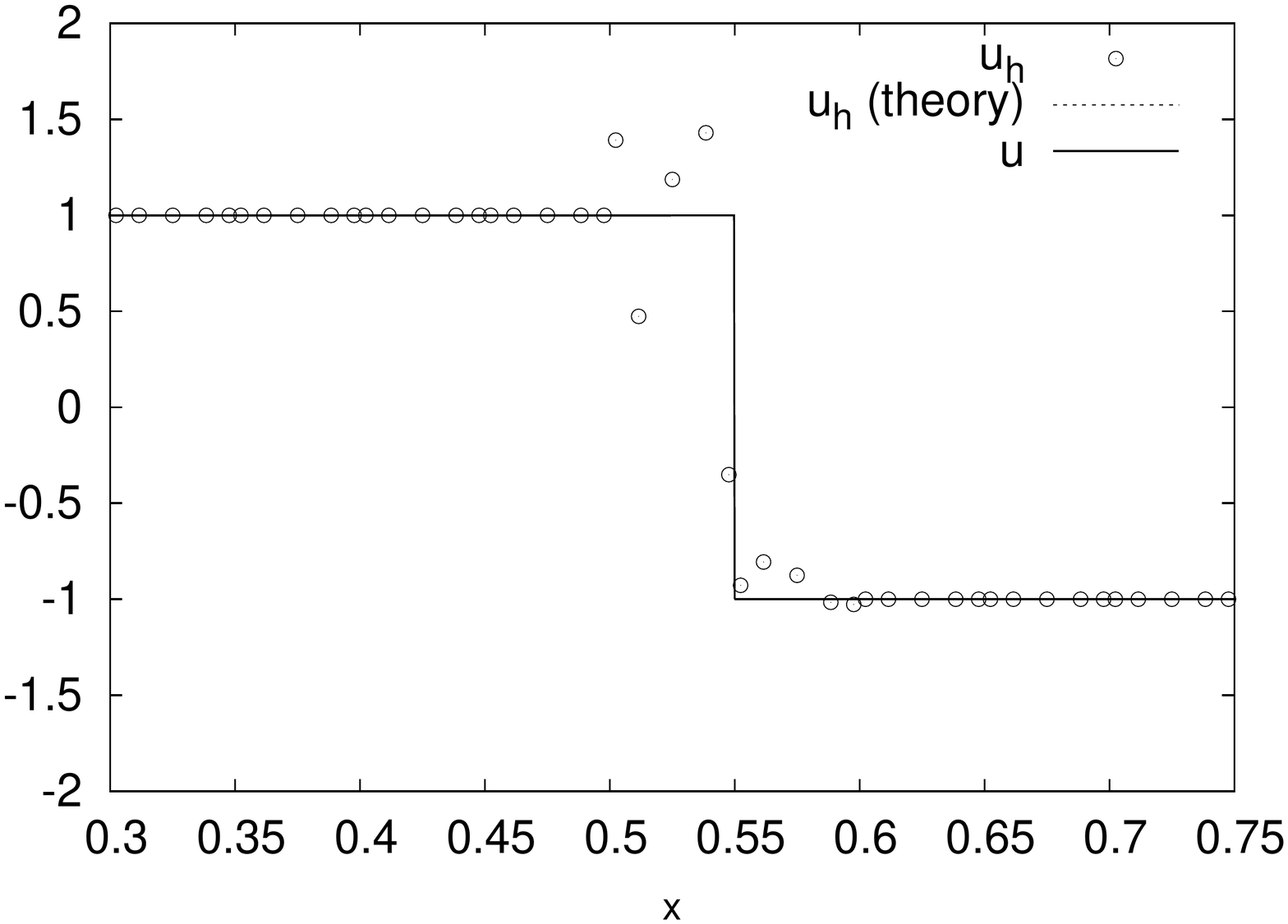 ,width=4.8cm}}
\caption{Steady-state solutions in the shock cell $\kappa_0$ for $p=3$ and the Godunov flux (\ref{eq:godunov_flux}).}
\label{fig:solutions_p3}
\end{center}
\end{figure}

In Figures~\ref{fig:solutions_llf} and \ref{fig:solutions_OSH}, we present numerical experiments with the LLF flux (\ref{eq:llf_flux}) and Engquist-Osher flux (\ref{eq:OSH_flux}). The results in the shock cell are remarkably similar to the theoretical solution for the Godunov flux (\ref{eq:godunov_flux}). This observation is in agreement with precedent numerical evidence of comparable resolution of the DG scheme with different numerical fluxes when the polynomial degree $p$ is increased \cite{cockburn-shu01,qui_et-al06}. Note that the Engquist-Osher flux reduces to the Godunov flux as soon as the left and right traces satisfy $u^-\geq\hat{u }$ and $u^+\leq\hat{u }$, where $\hat{u}$ is defined in section~\ref{sec:model_eqn}. For instance, it may be easily checked that this holds for $|s_c|\leq\tfrac{\sqrt{3}}{2}$ for $p=1$ and $|s_c|\leq\tfrac{\sqrt{3}}{3}$ for $p=2$ from Lemma~\ref{th:shock_sol_DOF0} and Theorem~\ref{th:shock_sol_p0123}. The main difference consists in the neighboring cells of the shock region where oscillations occur with the LLF flux. These oscillations are a consequence of Theorem~\ref{th:exp_decay} which details the mechanism of transmission of oscillations at interfaces with a monotone numerical flux. As an illustration, Table \ref{tab:exp_decay_LLF} gives the values of the left and right traces at interfaces $x_{j+1/2}$ of the oscillations of the numerical solution for $8\leq j\leq 10$ in the supersonic region. The trace $u_{j+1/2}^+-u_L$ for $j=10$ corresponds to the left trace in the shock cell $j_c=11$. As expected by Theorem~\ref{th:exp_decay} the sign of the jumps $u_{j+1/2}^\pm-u_L$ is conserved through the interfaces as a consequence of the monotonicity of the LLF flux. The last column provides the values of DOFs of $u_h-u_L$ in each cell scaled by the jump $u_{j+1/2}^+-u_L$. The values are always quite lower than $u_{j+1/2}^+-u_L$ and every DOF for $0\leq l < p$ is several order of magnitude lower than the last DOF $l=p$.

\begin{figure}
\begin{center}
\subfigure{\epsfig{figure=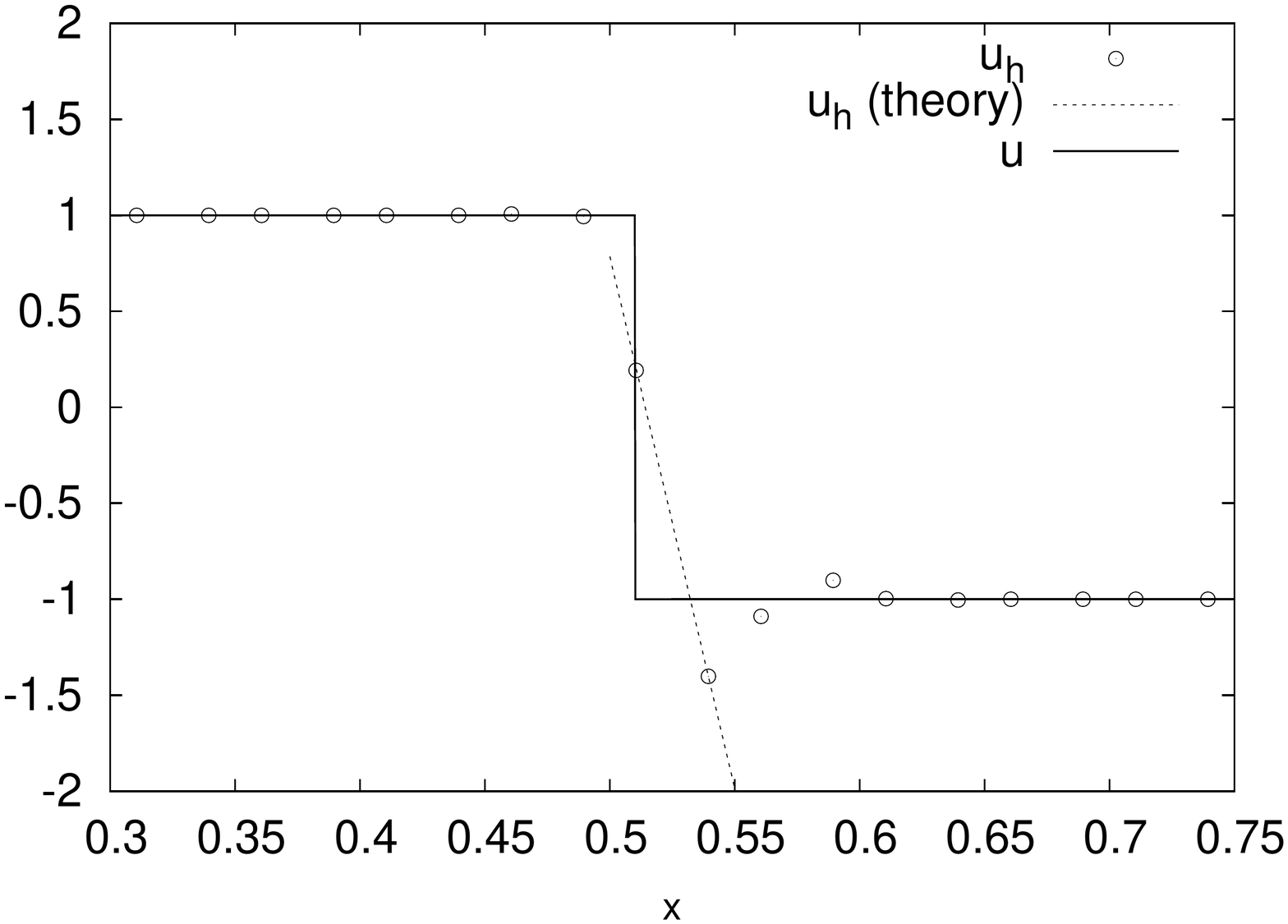 ,width=4.8cm}}\hspace{-0.8cm}
\subfigure{\epsfig{figure=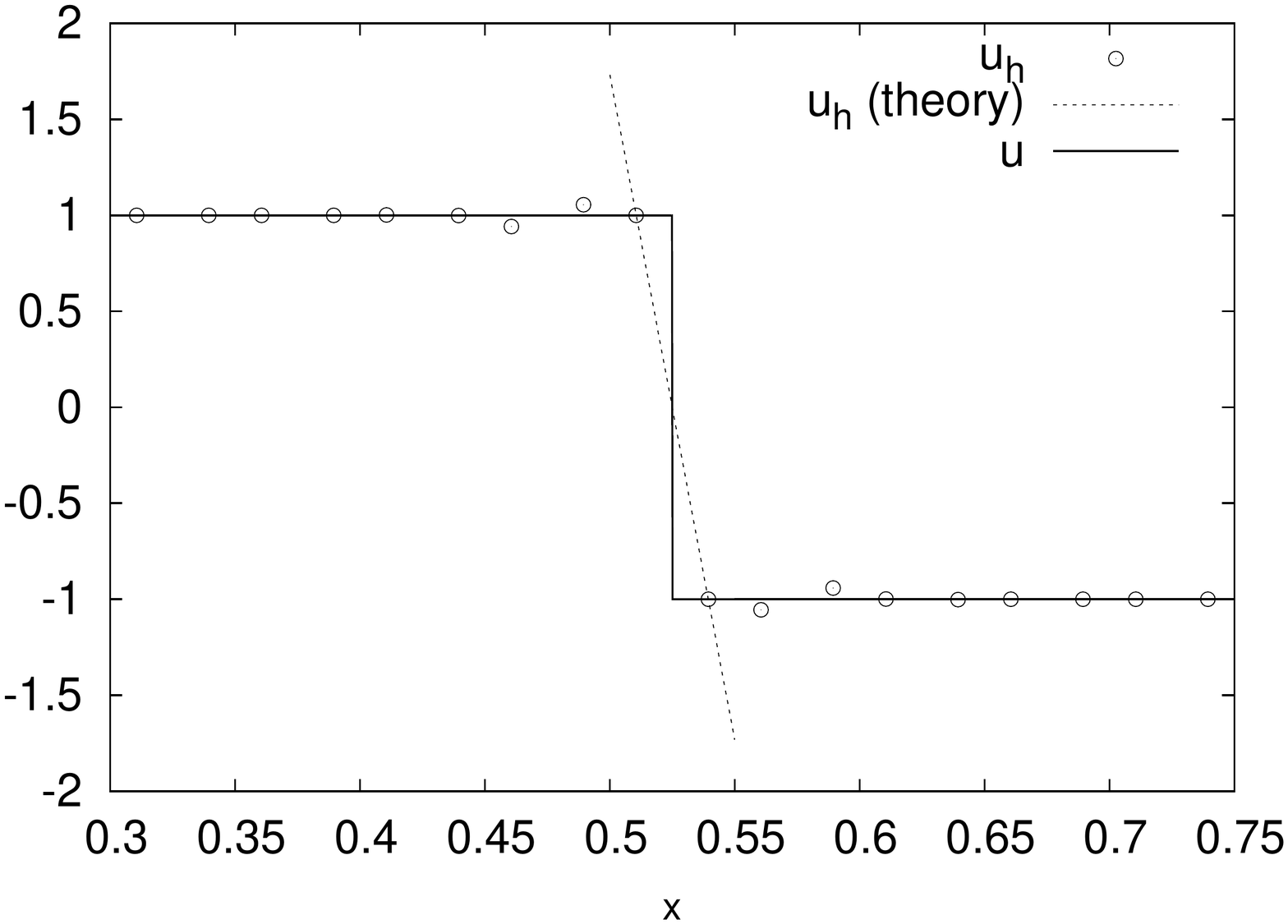 ,width=4.8cm}}\hspace{-0.8cm}
\subfigure{\epsfig{figure=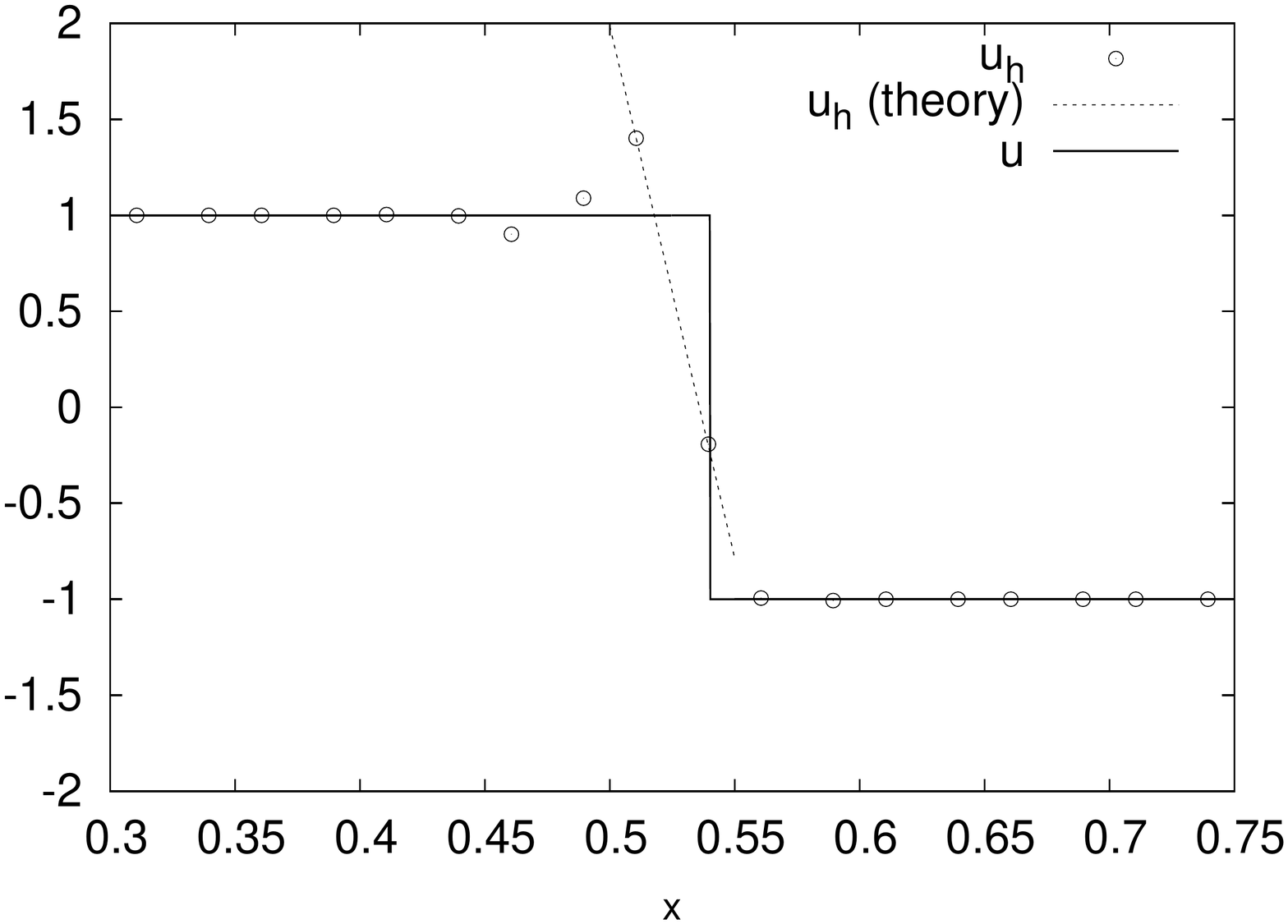 ,width=4.8cm}}\\
\subfigure{\epsfig{figure=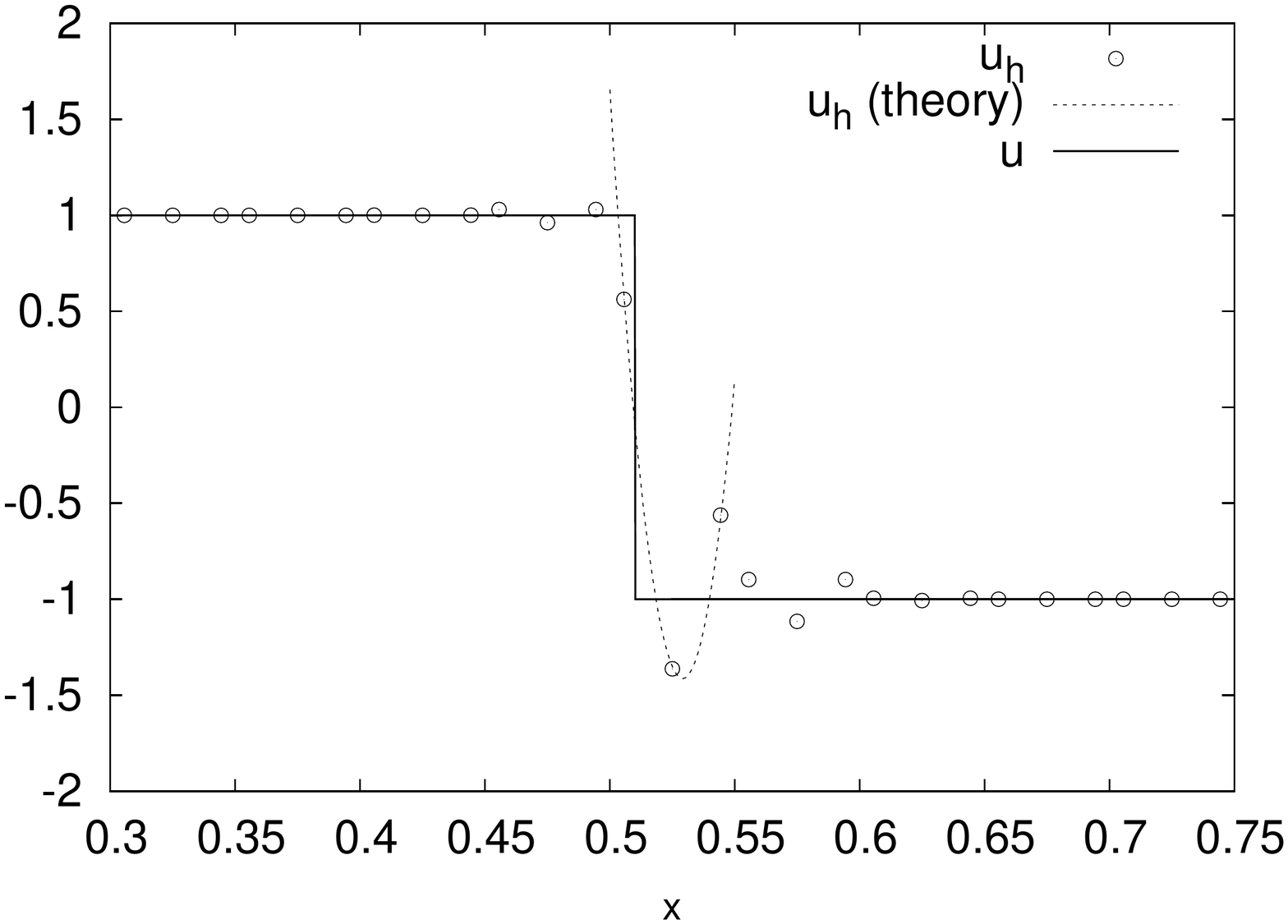 ,width=4.8cm}}\hspace{-0.8cm}
\subfigure{\epsfig{figure=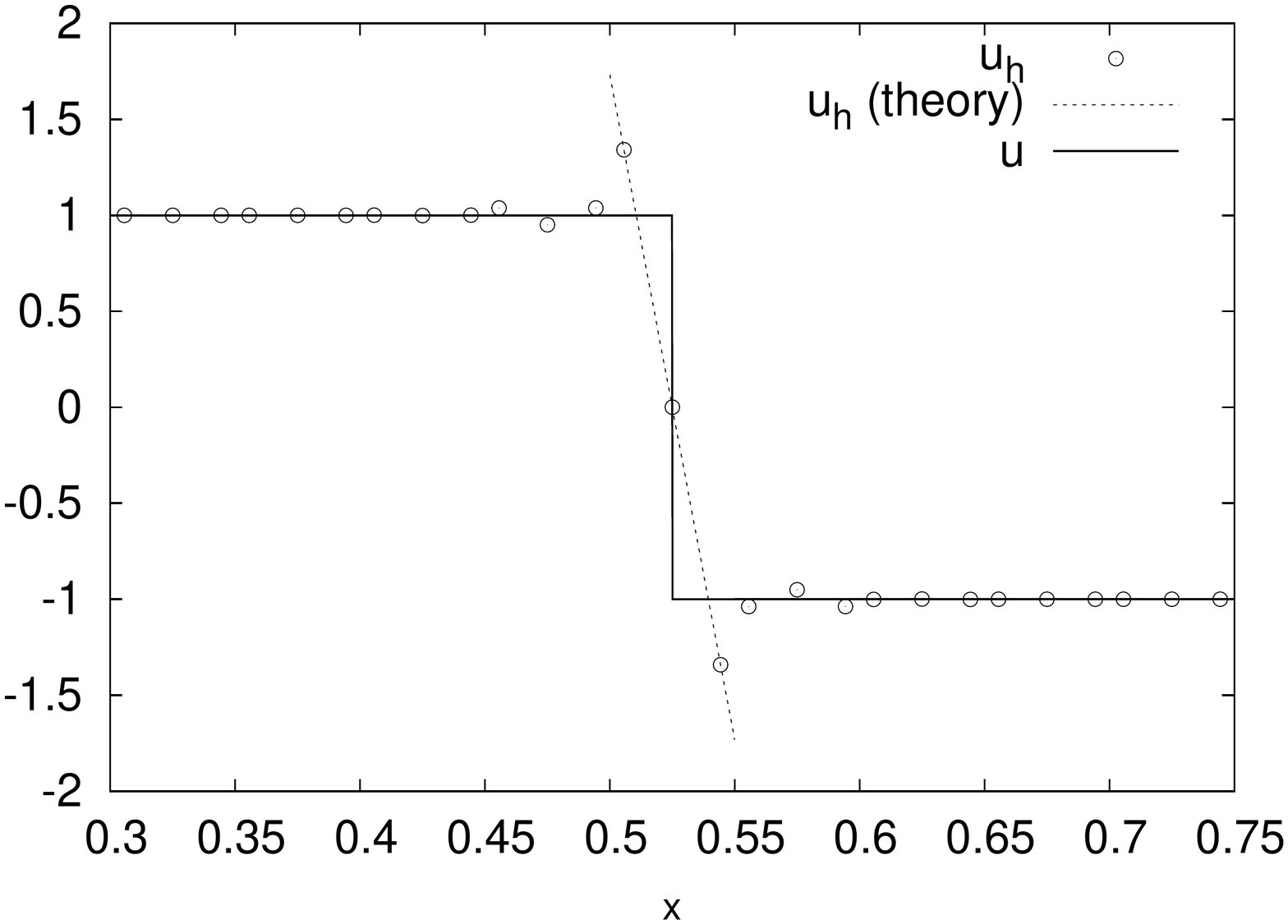 ,width=4.8cm}}\hspace{-0.8cm}
\subfigure{\epsfig{figure=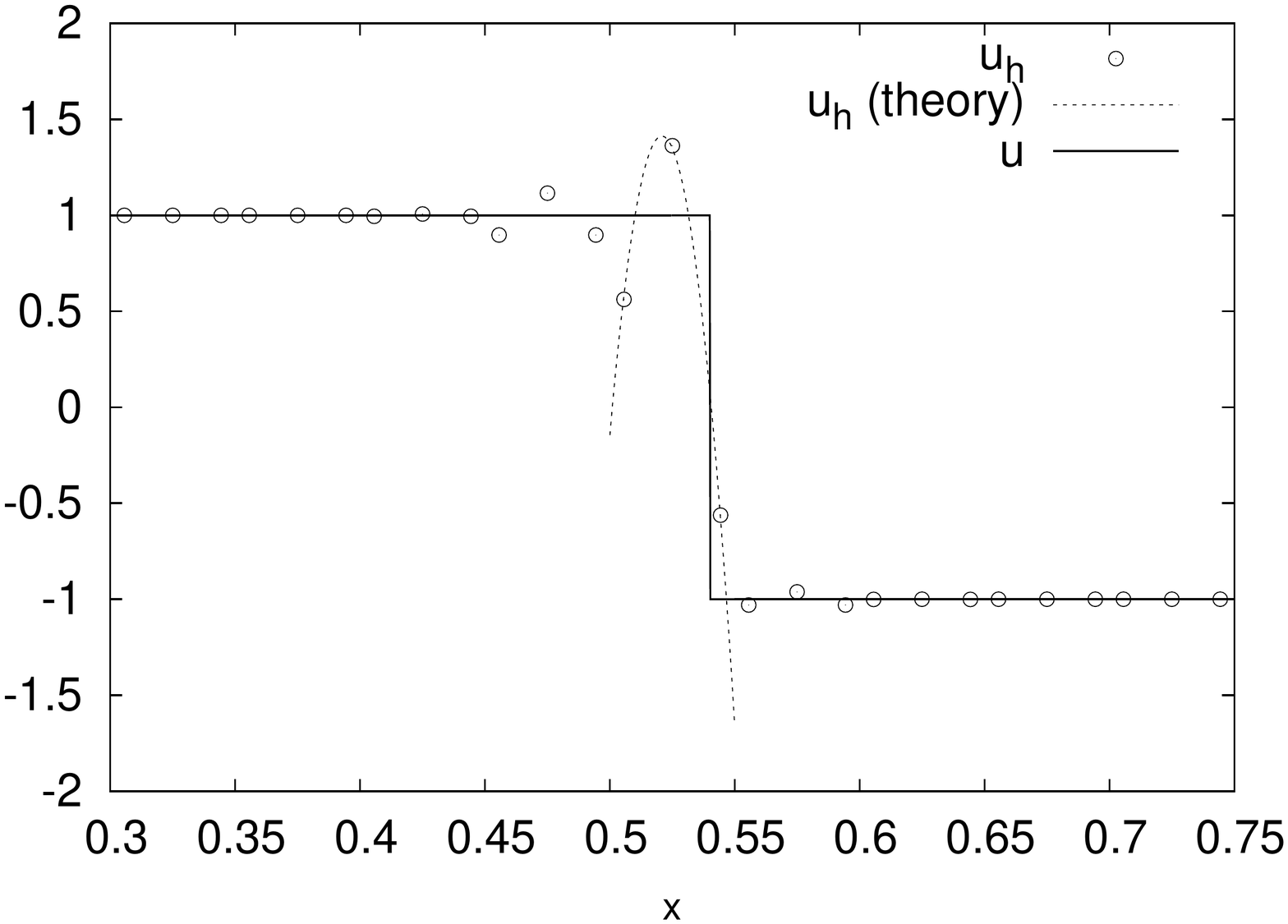 ,width=4.8cm}}\\
\setcounter{subfigure}{0}
\subfigure[$s_c=-\tfrac{3}{5}$]{\epsfig{figure=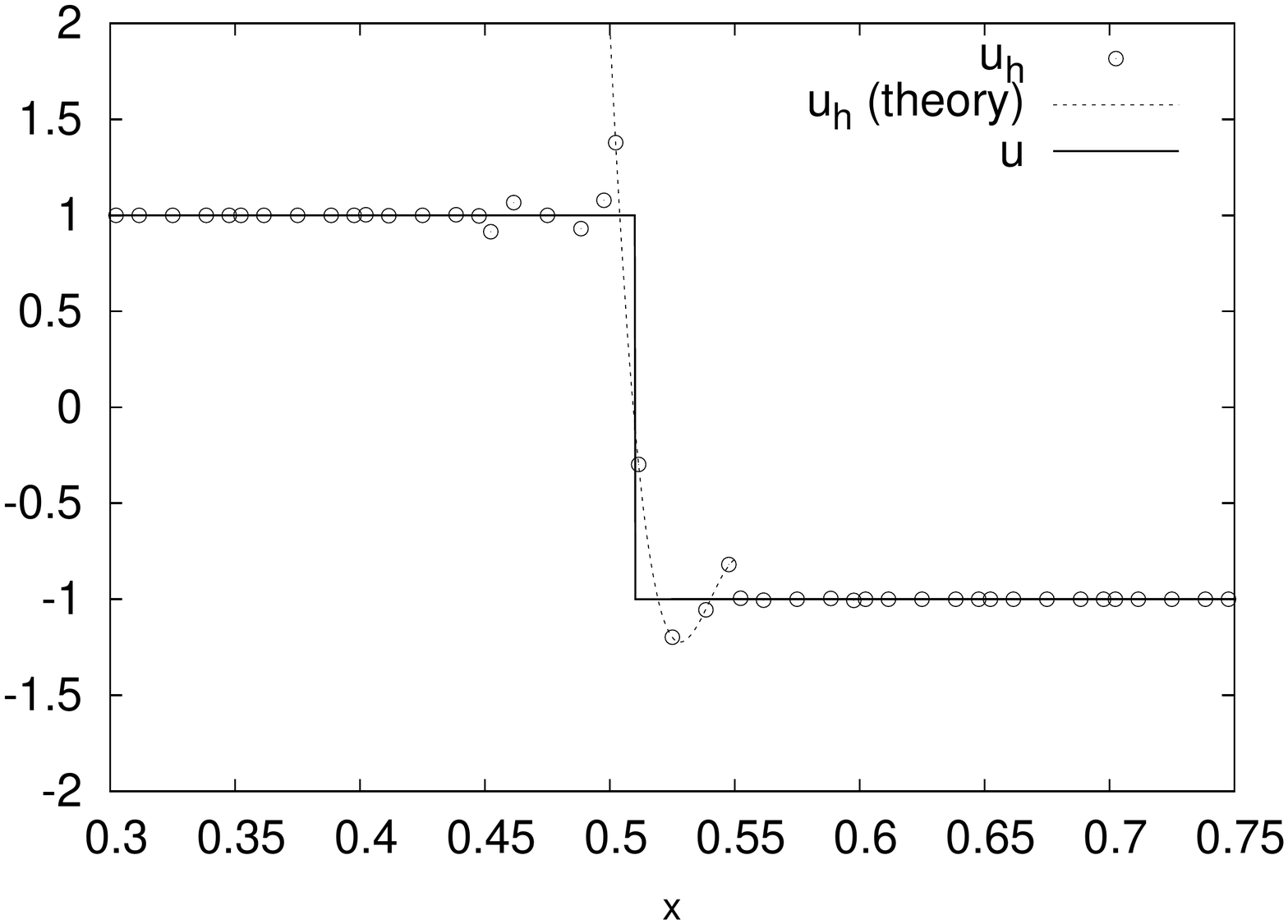 ,width=4.8cm}}\hspace{-0.8cm}
\subfigure[$s_c=0$            ]{\epsfig{figure=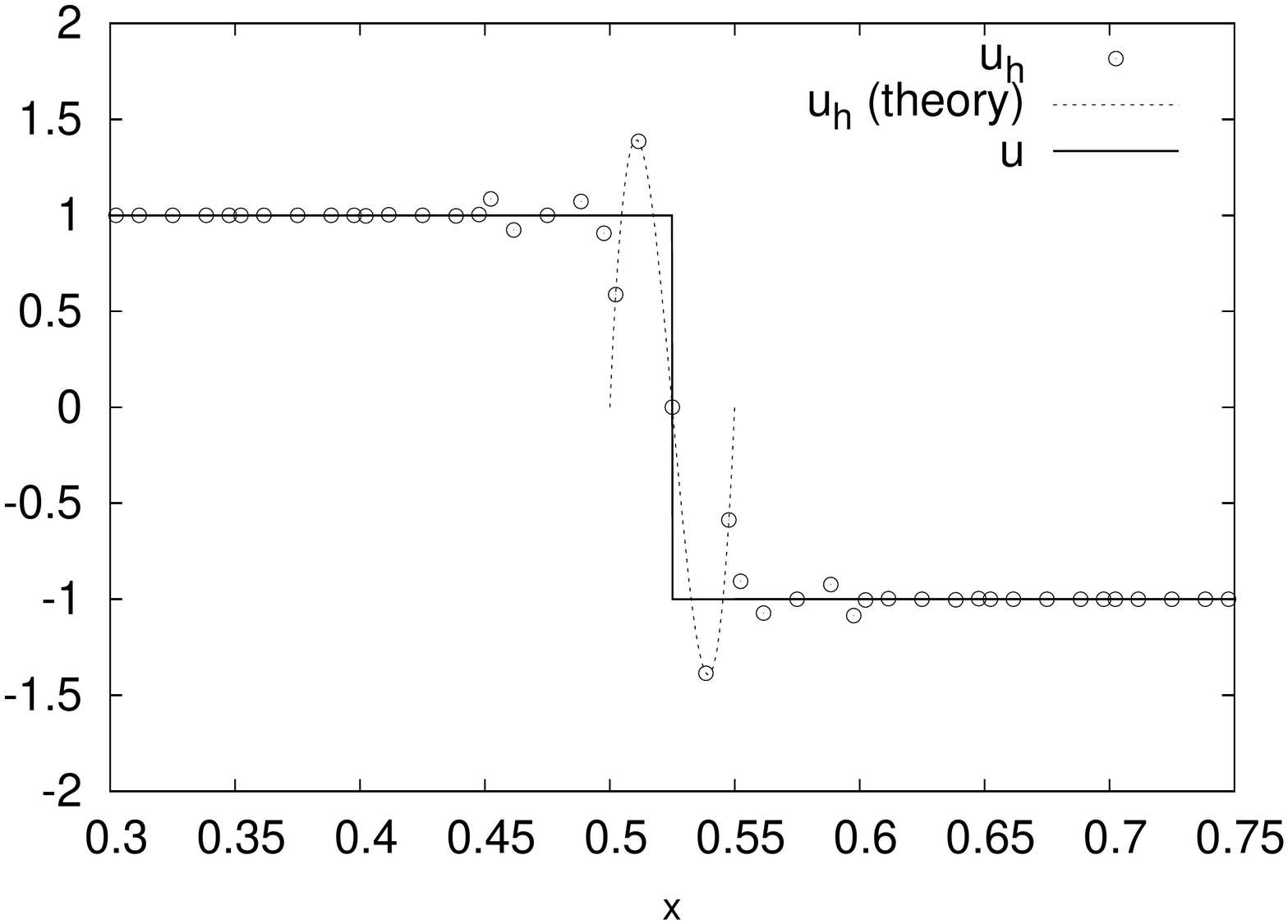 ,width=4.8cm}}\hspace{-0.8cm}
\subfigure[$s_c= \tfrac{3}{5}$]{\epsfig{figure=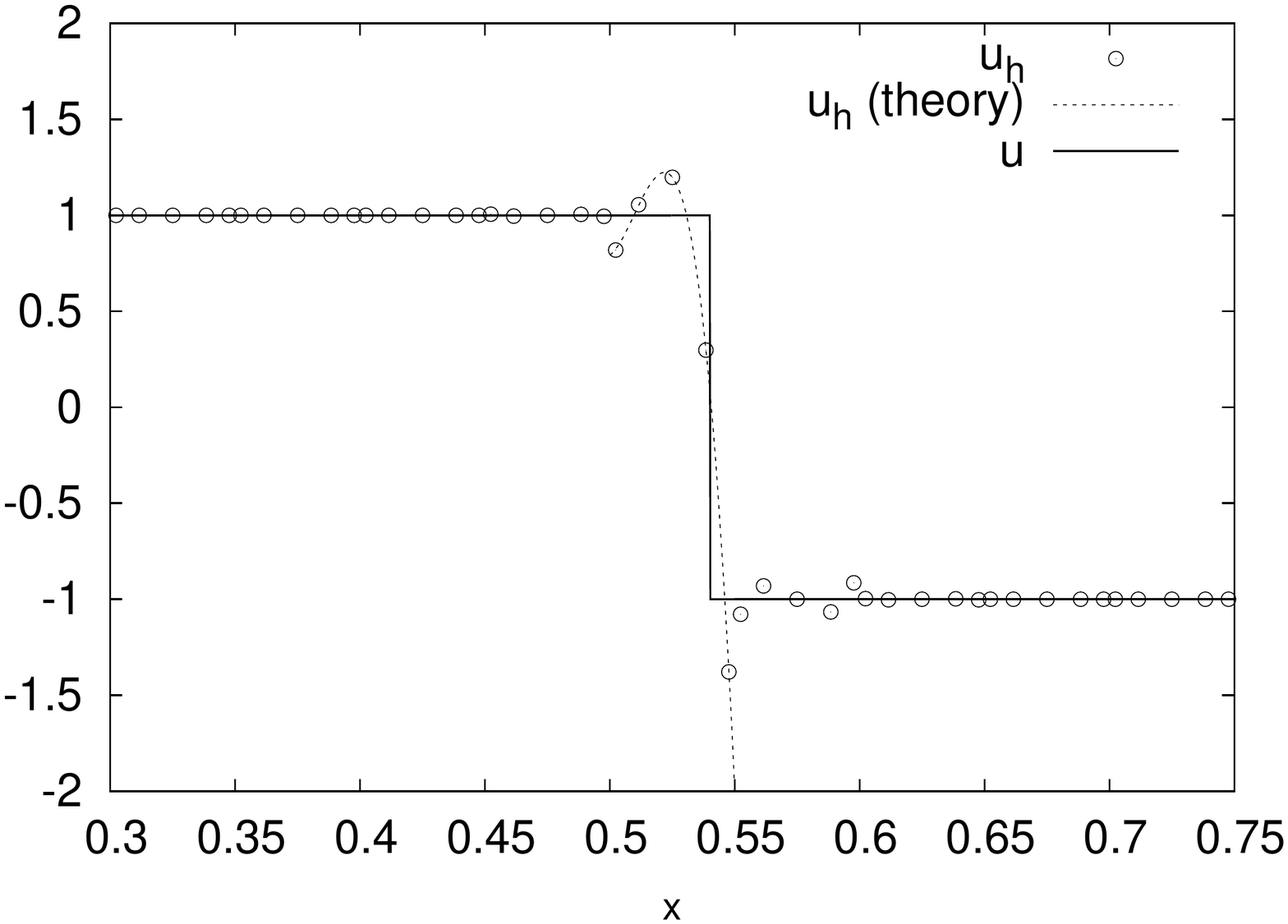 ,width=4.8cm}}
\caption{Steady-state solutions in the shock cell $\kappa_0$ obtained with the LLF flux (\ref{eq:llf_flux}) for $p=1$ (top row), $p=2$ (second row), and $p=3$ (bottom row).}
\label{fig:solutions_llf}
\end{center}
\end{figure}

\begin{figure}
\begin{center}
\subfigure{\epsfig{figure=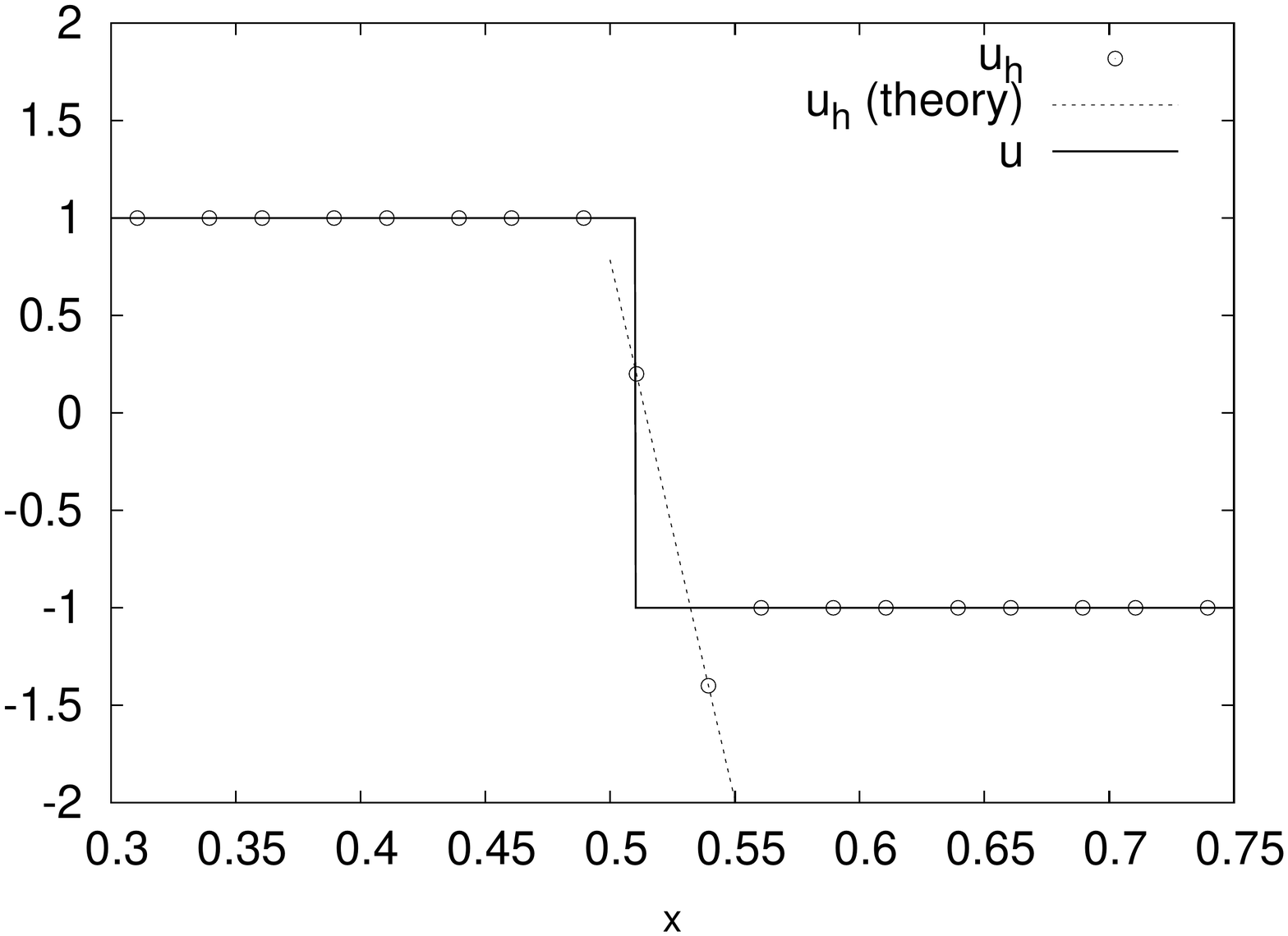 ,width=4.8cm}}\hspace{-0.8cm}
\subfigure{\epsfig{figure=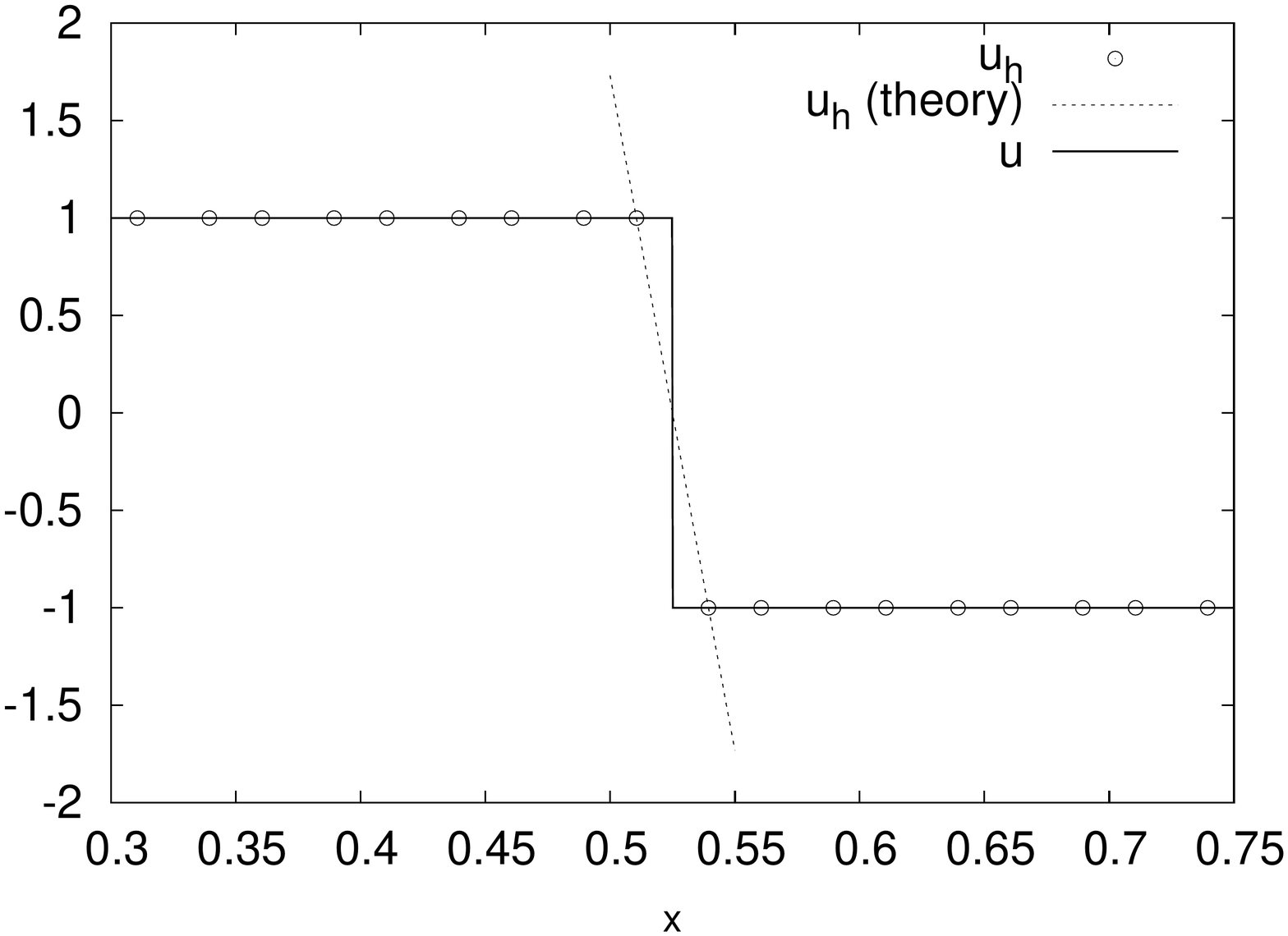 ,width=4.8cm}}\hspace{-0.8cm}
\subfigure{\epsfig{figure=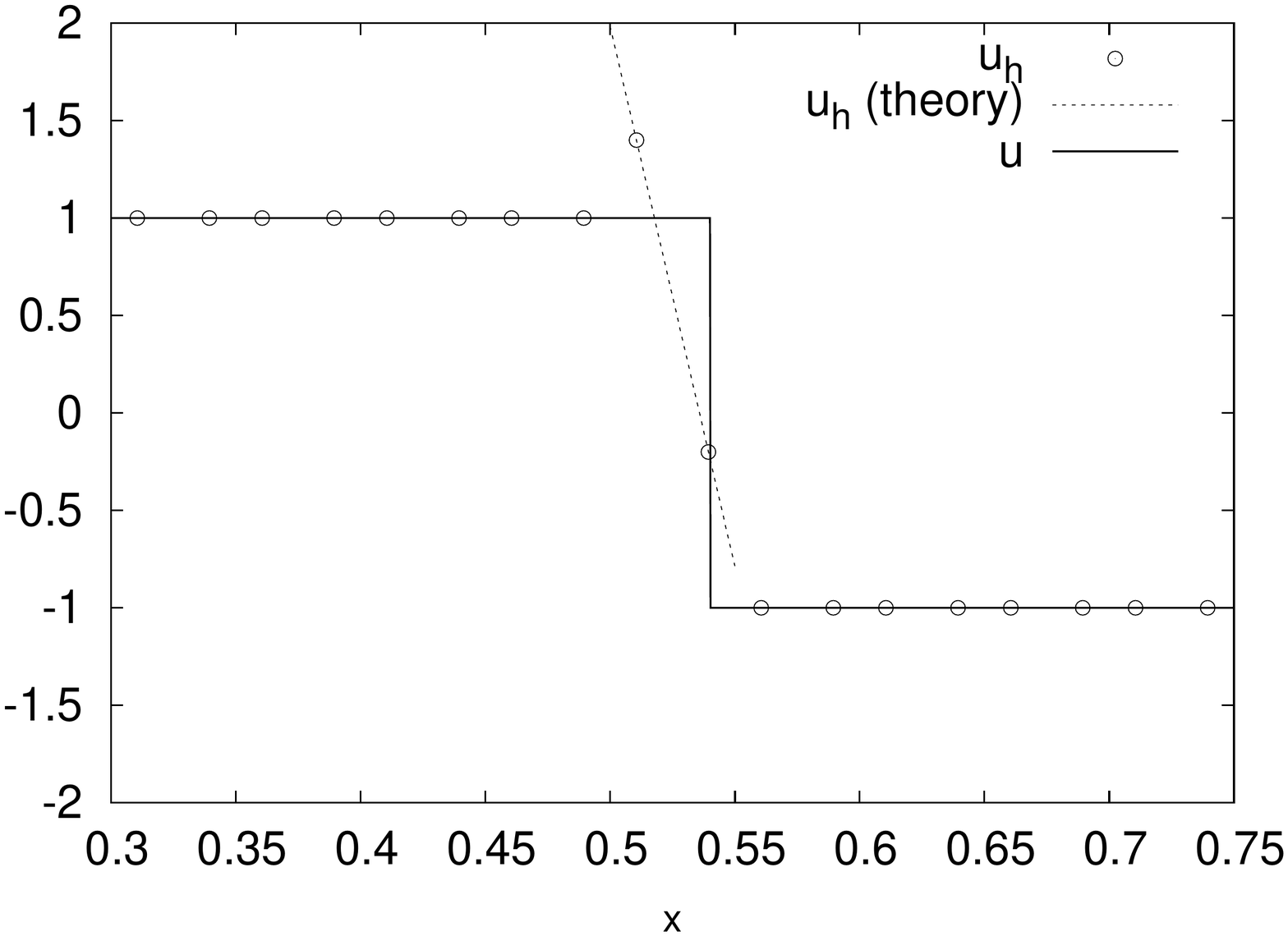 ,width=4.8cm}}\\
\subfigure{\epsfig{figure=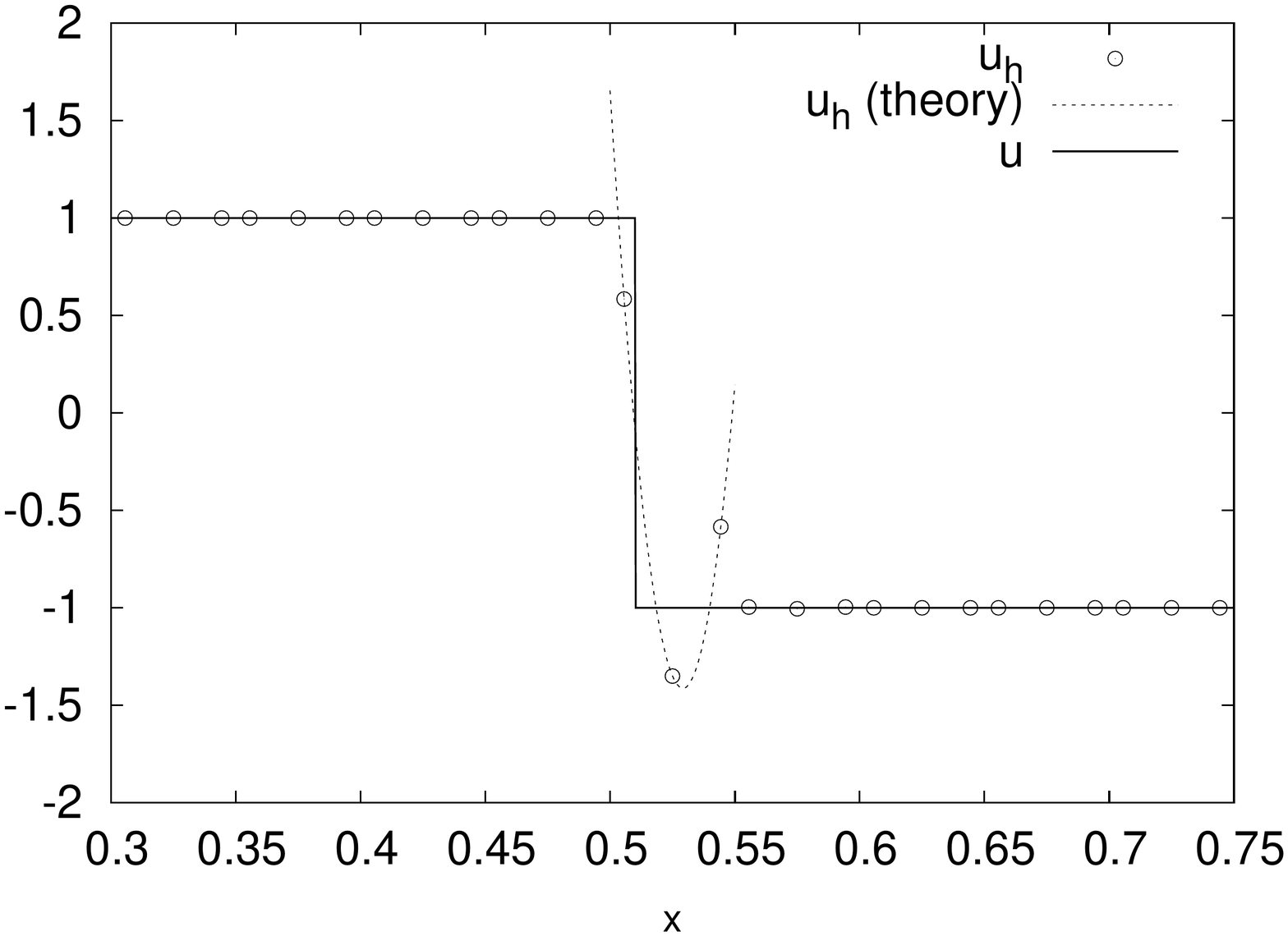 ,width=4.8cm}}\hspace{-0.8cm}
\subfigure{\epsfig{figure=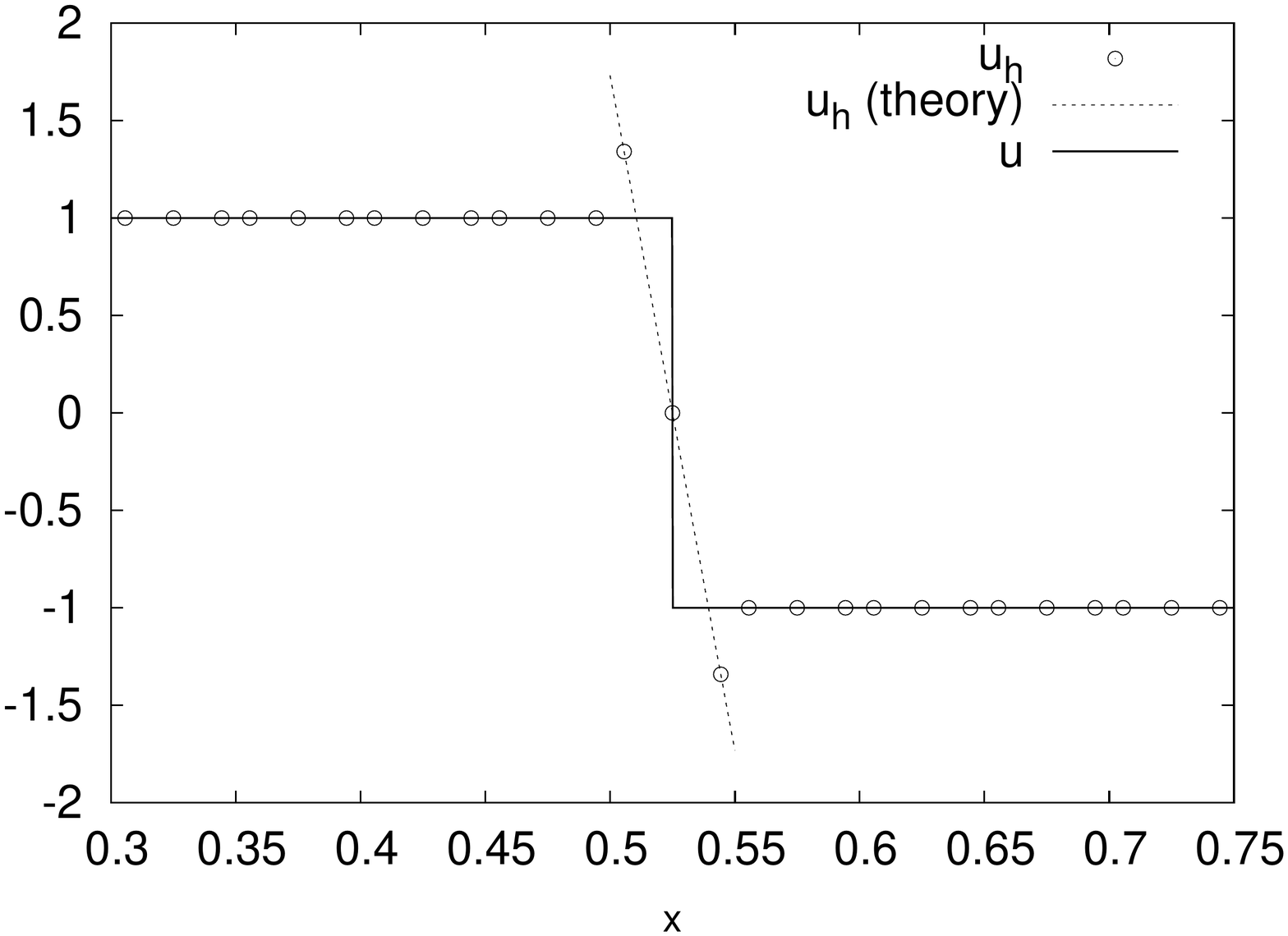 ,width=4.8cm}}\hspace{-0.8cm}
\subfigure{\epsfig{figure=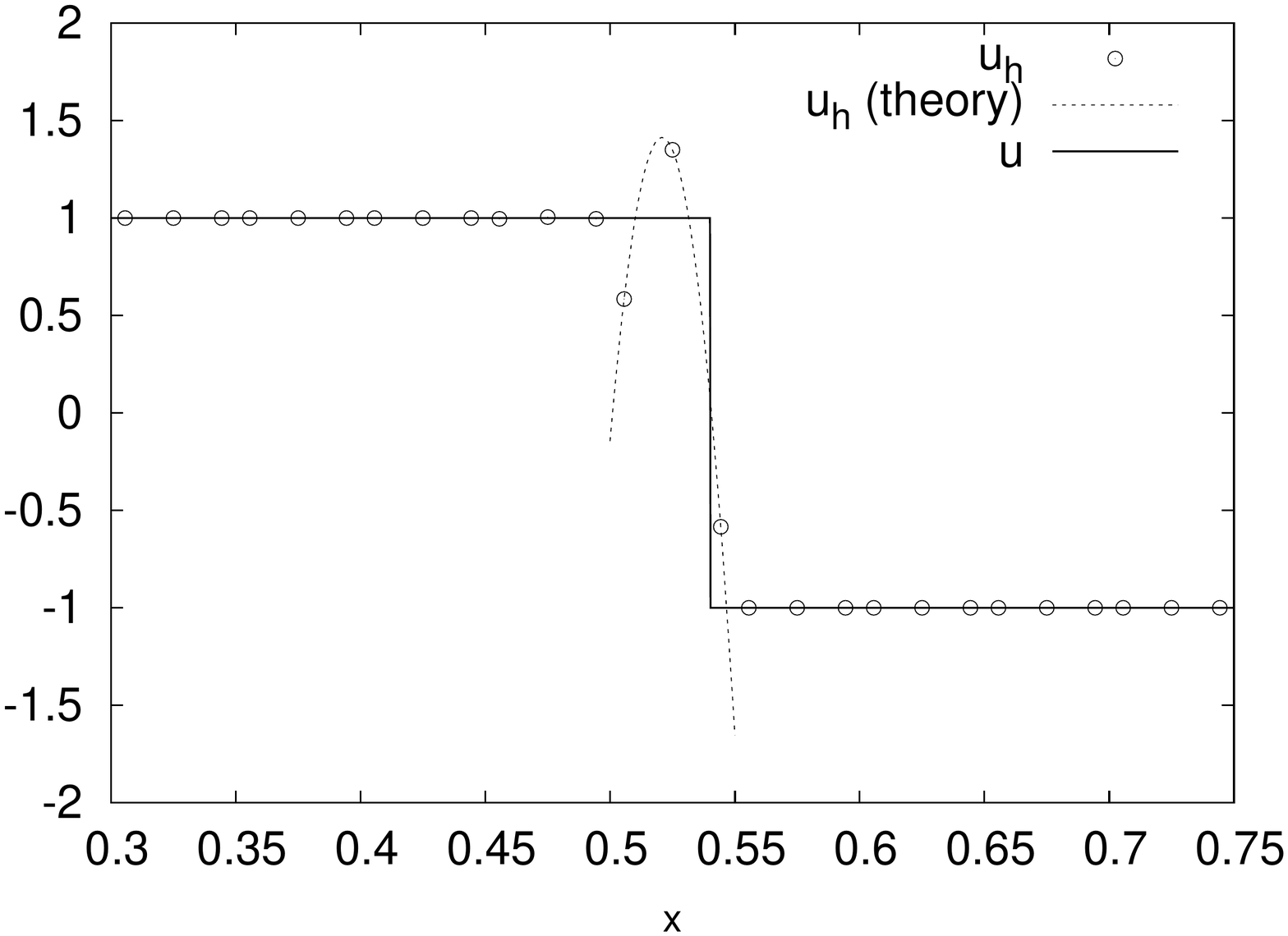 ,width=4.8cm}}\\
\setcounter{subfigure}{0}
\subfigure[$s_c=-\tfrac{3}{5}$]{\epsfig{figure=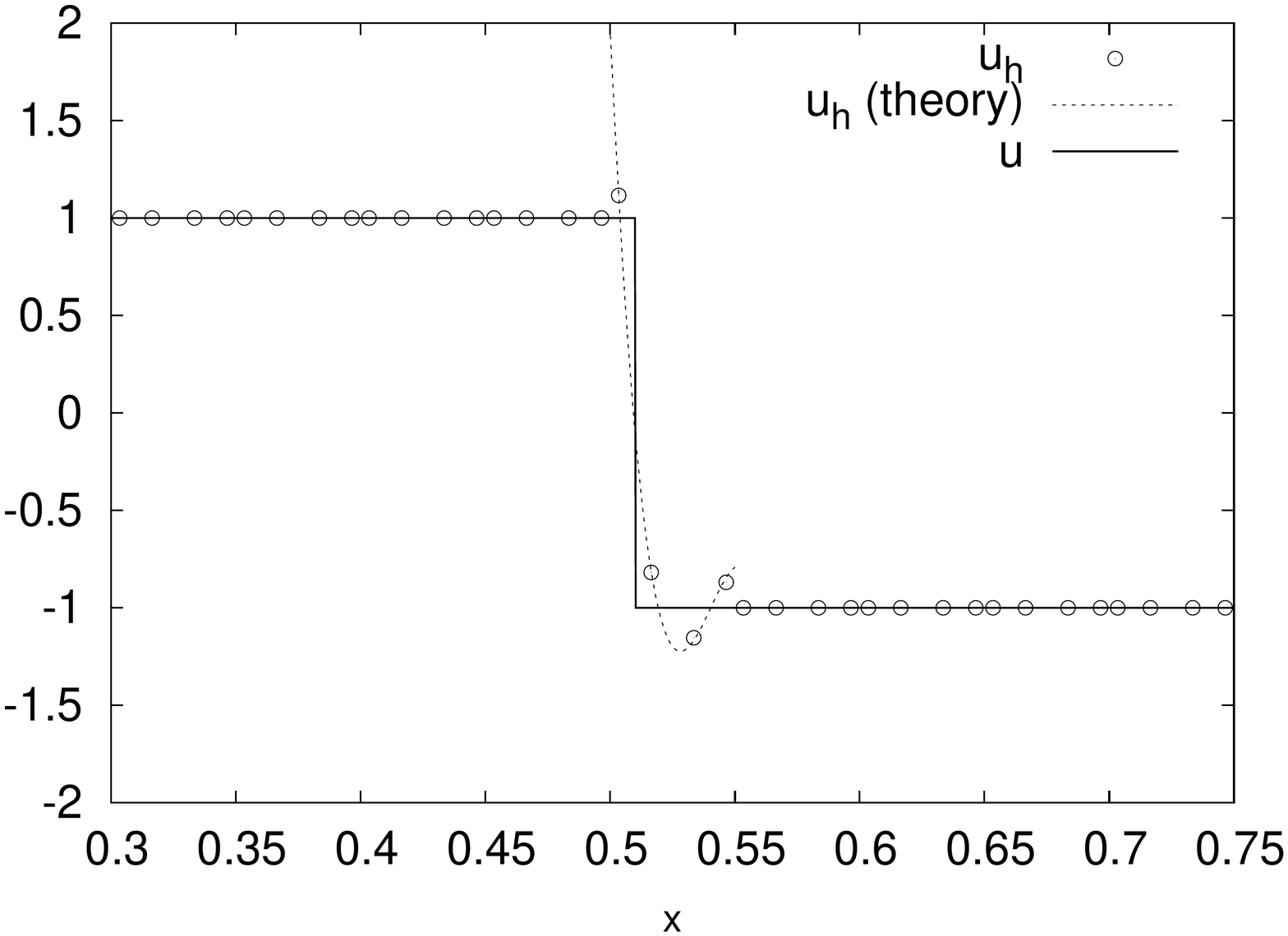 ,width=4.8cm}}\hspace{-0.8cm}
\subfigure[$s_c=0$            ]{\epsfig{figure=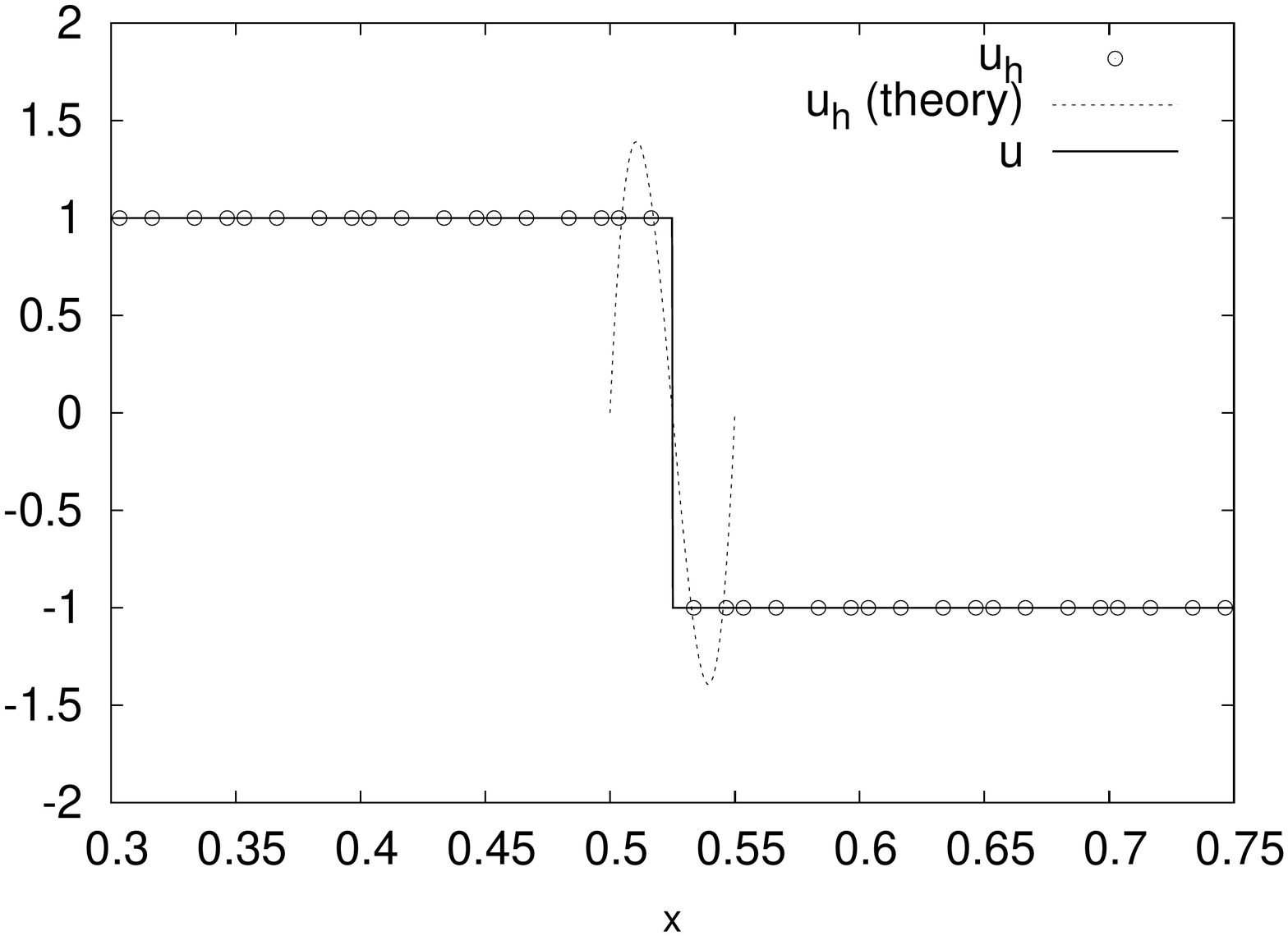 ,width=4.8cm}}\hspace{-0.8cm}
\subfigure[$s_c= \tfrac{3}{5}$]{\epsfig{figure=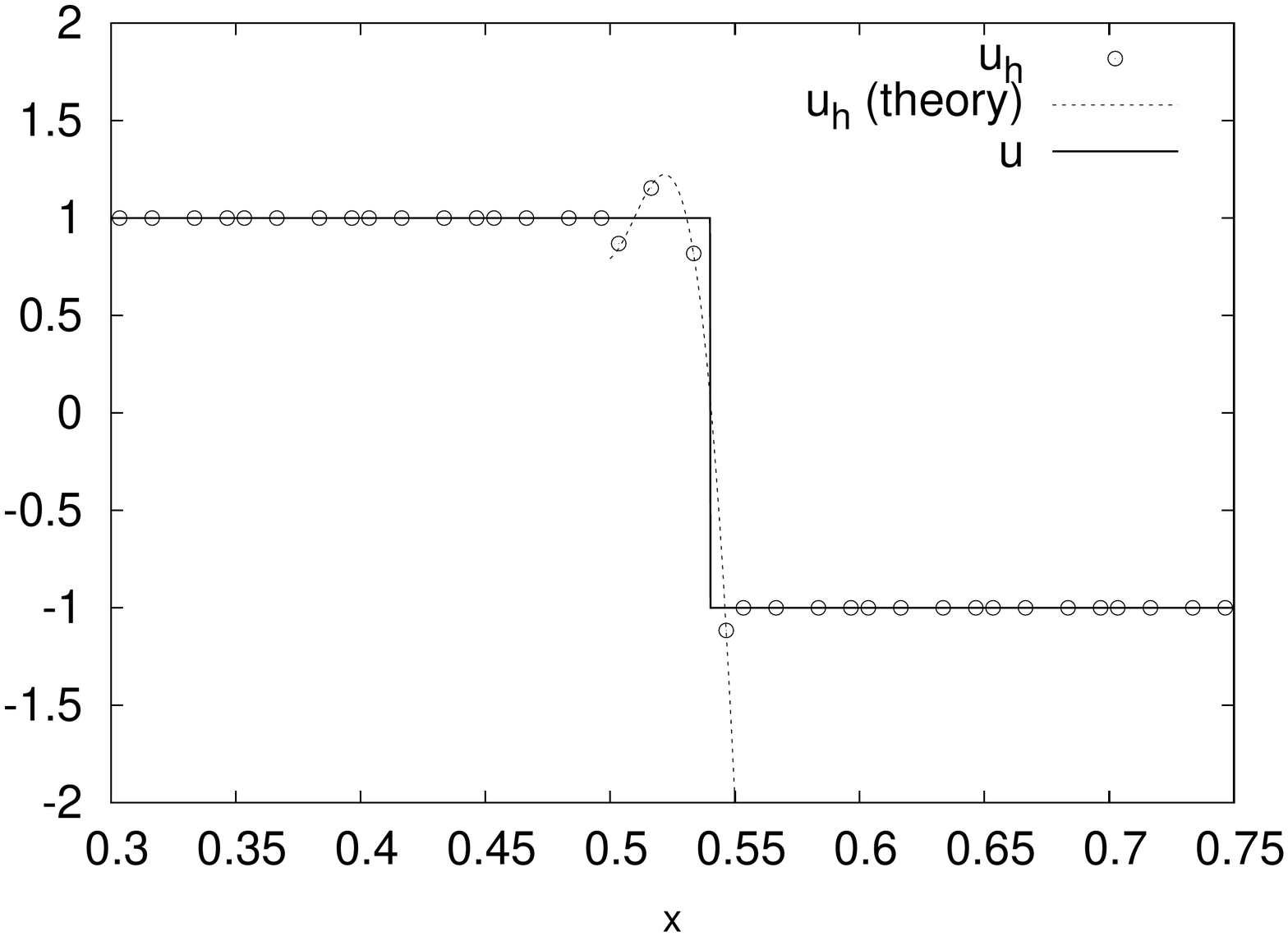 ,width=4.8cm}}
\caption{Steady-state solutions in the shock cell $\kappa_0$ obtained with the Engquist-Osher flux (\ref{eq:OSH_flux}) for $p=1$ (top row), $p=2$ (second row), and $p=3$ (bottom row).}
\label{fig:solutions_OSH}
\end{center}
\end{figure}

\begin{table}
     \begin{center}
     \caption{Illustration of Theorem~\ref{th:exp_decay} for the LLF flux. Results obtained in the supersonic region for $1\leq p \leq 3$ and $s_c=-\tfrac{3}{5}$ (top), $s_c=0$ (middle), and $s_c=\tfrac{3}{5}$ (bottom).}
     \begin{tabular}{|c|c|r|r|l|}
	\hline
	$p$ & $j$ & $u_{j+1/2}^--u_L$ & $u_{j+1/2}^+-u_L$ & $\Big|\tfrac{(U_j^l-u_L\delta_{l,0})_{0\leq l \leq p}}{u_{j+1/2}^+-u_L}\Big|$ \\
	\hline
	    & $10$ & $-1.141$e-$02$ & $-2.245$e-$01$ & $9.639$e-$05$ $5.075$e-$02$ \\
	$1$ & $9$  &  $3.215$e-$05$ & $1.137$e-$02$  & $1.515$e-$08$ $2.827$e-$03$ \\
	    & $8$  & $-2.585$e-$10$ & $-3.215$e-$05$ & $3.452$e-$12$ $8.038$e-$06$\\
	\hline
	    & $10$ & $7.513$e-$02$  & $6.335$e-$01$  & $9.049$e-$04$ $1.425$e-$13$ $1.194$e-$01$ \\
	$2$ & $9$  & $1.359$e-$03$  & $7.513$e-$02$  & $2.461$e-$06$ $1.198$e-$12$ $1.809$e-$02$ \\
	    & $8$  & $4.618$e-$07$  & $1.359$e-$03$  & $2.433$e-$11$ $7.067$e-$11$ $3.396$e-$04$ \\
	\hline
	    & $10$ & $1.632$e-$01$  & $1.003$e-$00$  & $1.918$e-$03$ $4.649$e-$05$ $1.100$e-$03$ $1.635$e-$01$ \\
	$3$ & $9$  & $-6.305$e-$03$ & $-1.648$e-$01$ & $1.721$e-$05$ $1.595$e-$08$ $9.839$e-$06$ $3.823$e-$02$ \\
	    & $8$  & $9.900$e-$06$  & $6.303$e-$03$  & $1.103$e-$09$ $9.581$e-$12$ $6.455$e-$10$ $1.570$e-$03$ \\
	\hline
	\hline
	    & $10$ & $9.637$e-$02$  & $7.320$e-$01$  & $2.187$e-$03$ $1.338$e-$01$ \\
	$1$ & $9$  & $-2.365$e-$03$ & $-9.957$e-$02$ & $9.357$e-$06$ $2.374$e-$02$ \\
	    & $8$  & $1.395$e-$06$  & $2.363$e-$03$  & $1.362$e-$10$ $5.901$e-$04$ \\
	\hline
	    & $10$ & $9.637$e-$02$  & $7.320$e-$01$  & $1.294$e-$03$ $1.293$e-$13$ $1.329$e-$01$ \\
	$2$ & $9$  & $2.214$e-$03$  & $9.637$e-$02$  & $5.089$e-$06$ $9.095$e-$13$ $2.297$e-$02$ \\
	    & $8$  & $1.224$e-$06$  & $2.214$e-$03$  & $4.613$e-$11$ $4.139$e-$11$ $5.529$e-$04$ \\
	\hline
	    & $10$ & $-6.651$e-$02$ & $-5.736$e-$01$ & $5.488$e-$04$ $5.358$e-$06$ $3.137$e-$04$ $1.157$e-$01$ \\
	$3$ & $9$  & $1.061$e-$03$  & $6.624$e-$02$  & $1.215$e-$06$ $1.888$e-$10$ $6.946$e-$07$ $1.602$e-$02$ \\
	    & $8$  & $-2.817$e-$07$ & $-1.061$e-$03$ & $3.889$e-$11$ $5.812$e-$11$ $6.699$e-$11$ $2.653$e-$04$ \\
	\hline
	\hline
	    & $10$ & $1.581$e-$01$  & $9.842$e-$01$  & $4.483$e-$03$ $1.651$e-$01$ \\
	$1$ & $9$  & $-6.460$e-$03$ & $-1.6695$e-$01$& $4.157$e-$05$ $3.865$e-$02$ \\
	    & $8$  & $1.035$e-$05$  & $6.446$e-$03$  & $2.773$e-$09$ $1.606$e-$03$ \\
	\hline
	    & $10$ & $-2.489$e-$01$  & $-1.182$e-$00$  & $5.007$e-$03$ $6.805$e-$14$ $2.054$e-$01$ \\
	$2$ & $9$  & $-1.390$e-$02$  & $-2.489$e-$01$  & $7.744$e-$05$ $3.599$e-$13$ $5.578$e-$02$ \\
	    & $8$  & $-4.799$e-$05$  & $-1.390$e-$02$  & $1.656$e-$08$ $6.728$e-$12$ $3.451$e-$03$ \\
	\hline
	    & $10$ & $-6.6171$e-$03$  & $-1.690$e-$01$  & $1.849$e-$05$ $1.797$e-$08$ $1.056$e-$05$ $3.913$e-$02$ \\
	$3$ & $9$  & $1.0901$e-$05$   & $6.614$e-$03$   & $1.276$e-$09$ $8.861$e-$12$ $7.436$e-$10$ $1.648$e-$03$ \\
	    & $8$  & $-2.9690$e-$11$  & $-1.090$e-$05$  & $4.032$e-$09$ $5.432$e-$09$ $6.223$e-$09$ $2.739$e-$06$ \\
	\hline
   \end{tabular}
    \label{tab:exp_decay_LLF}
    \end{center}
\end{table}

\subsection{Linear stability of steady shock profiles}\label{sec:num_xp_stab}

Figure~\ref{fig:ana_stab_vap_p1_p2} presents the spectra of the linearized operator (\ref{eq:linear_operator}) around steady-state solutions for different polynomial approximations and numerical fluxes when varying the relative shock position in the range $[0,1]$. Results are given for explicit first- and third-order time integration schemes. We are here interested in the slow convergence and even non convergence observed from the numerical experiments in the preceding section and illustrated in Figure~\ref{fig:residuals_p2}. Considering only $p=1$ and $p=2$ approximations thus appears sufficient for our purpose. The results for the Godunov flux have already been determined from the stability analysis in section~\ref{sec:ana_stab} and will be used for comparison. 

\begin{figure}
\begin{center}
\subfigure[]{\epsfig{figure=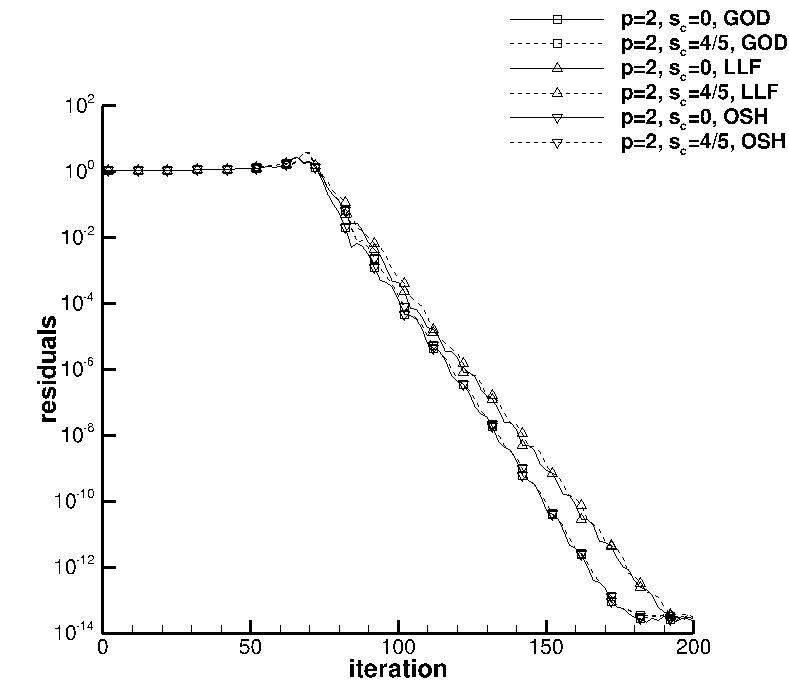 ,width=6cm}}
\subfigure[]{\epsfig{figure=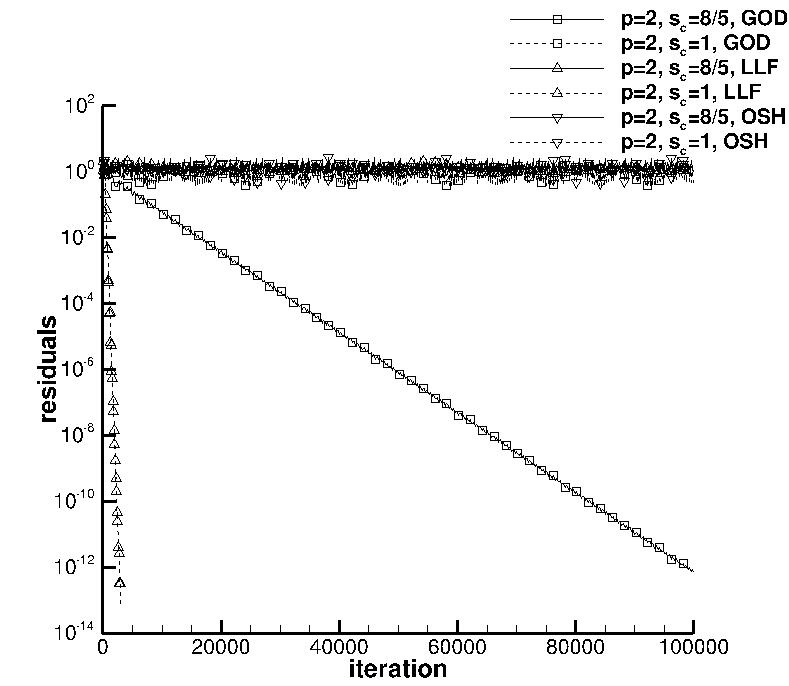 ,width=6cm}}
\caption{Convergence histories to steady-state solutions of the Burgers equation for different kind of numerical fluxes (GOD: Godunov, LLF, OSH: Engquist-Osher) and relative shock positions $0\leq s_c\leq1$ with $p=2$ and $N=20$.}
\label{fig:residuals_p2}
\end{center}
\end{figure}

For the second-order approximation, $p=1$, every eigenvalue is contained in the unit disc. In the case of the Godunov flux, the pairs of complex conjugate eigenvalues correspond to modes located in cells $j\neq 0$ and associated to uniform flow, while real eigenvalues correspond to modes in the shock cell as indicated in Table~\ref{tab:eigenvalues}. The spectra obtained for other numerical fluxes look similar even if a scattering of eigenvalues in cells $j\neq0$ is observed for the LLF flux. Using a third-order time integration scheme is seen to lower the modulus of eigenvalues as expected.

\begin{figure}
\begin{center}
\subfigure{\epsfig{figure=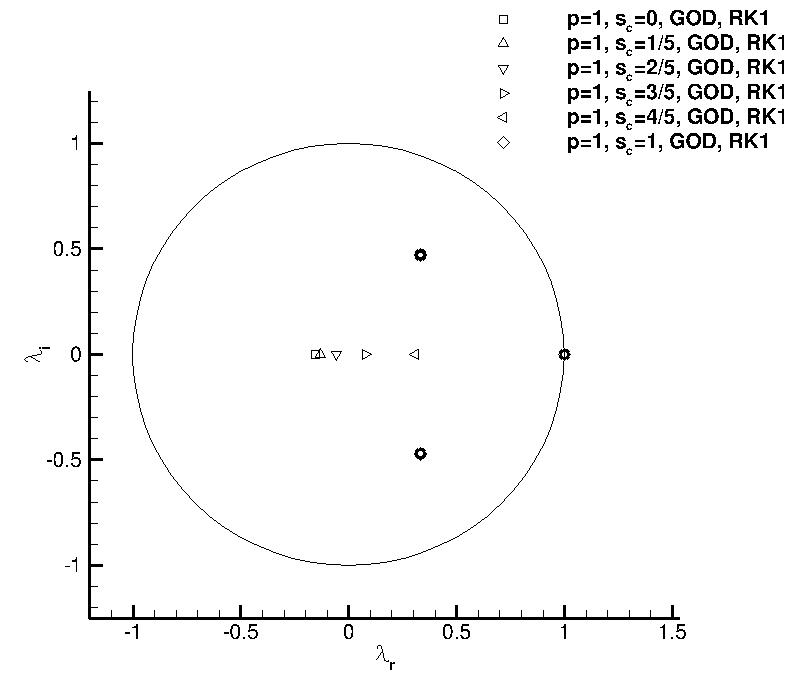 ,width=4.2cm}}\hspace{-0.0cm}
\subfigure{\epsfig{figure=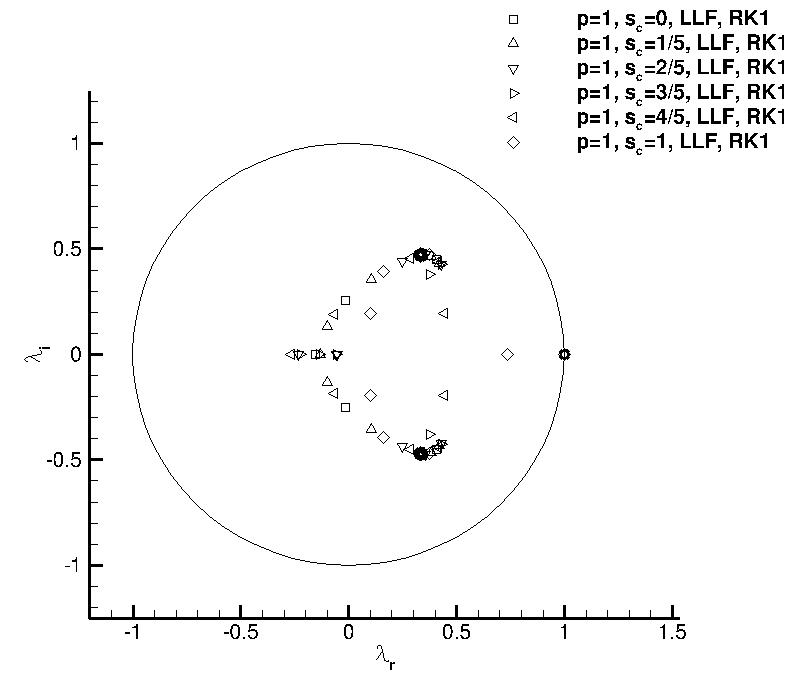 ,width=4.2cm}}\hspace{-0.0cm}
\subfigure{\epsfig{figure=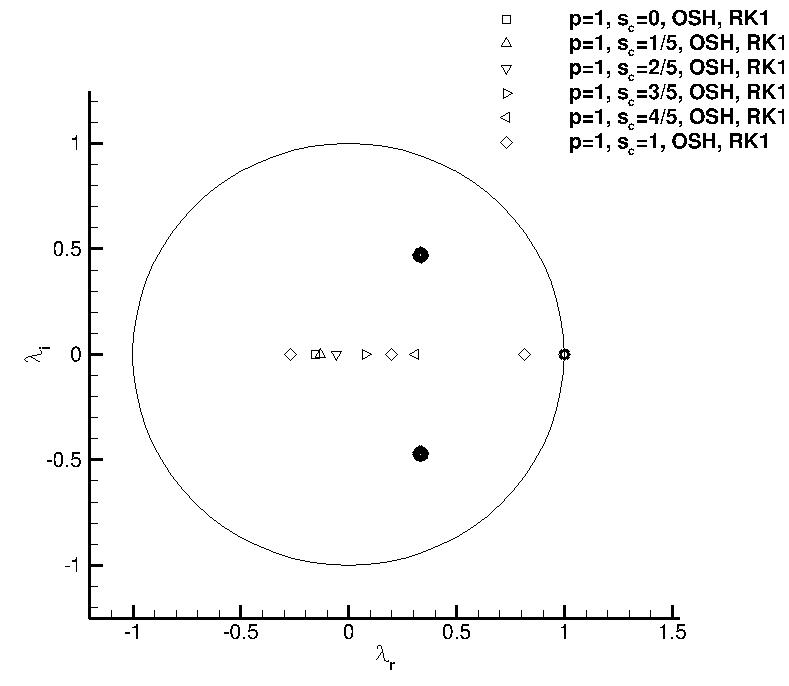 ,width=4.2cm}}\\
\subfigure{\epsfig{figure=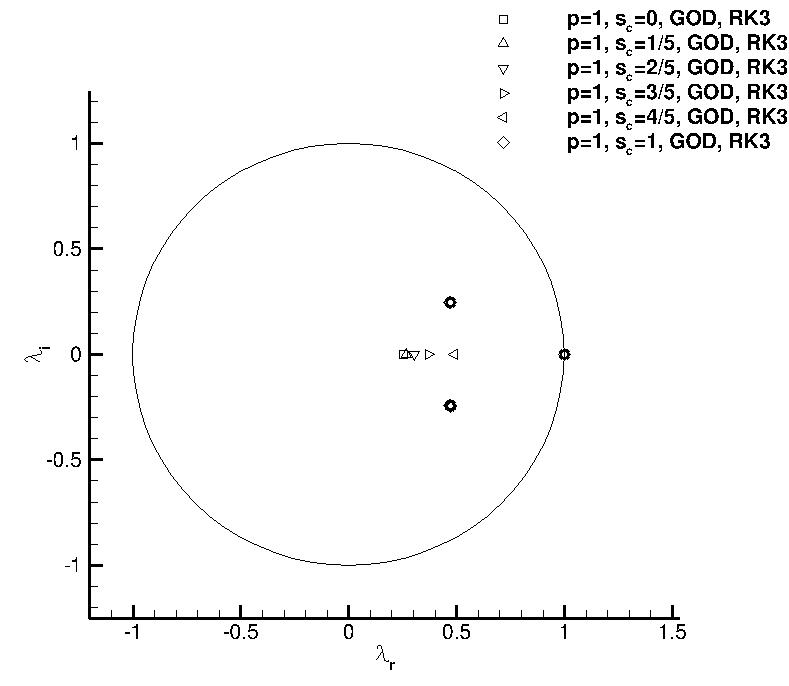 ,width=4.2cm}}\hspace{-0.0cm}
\subfigure{\epsfig{figure=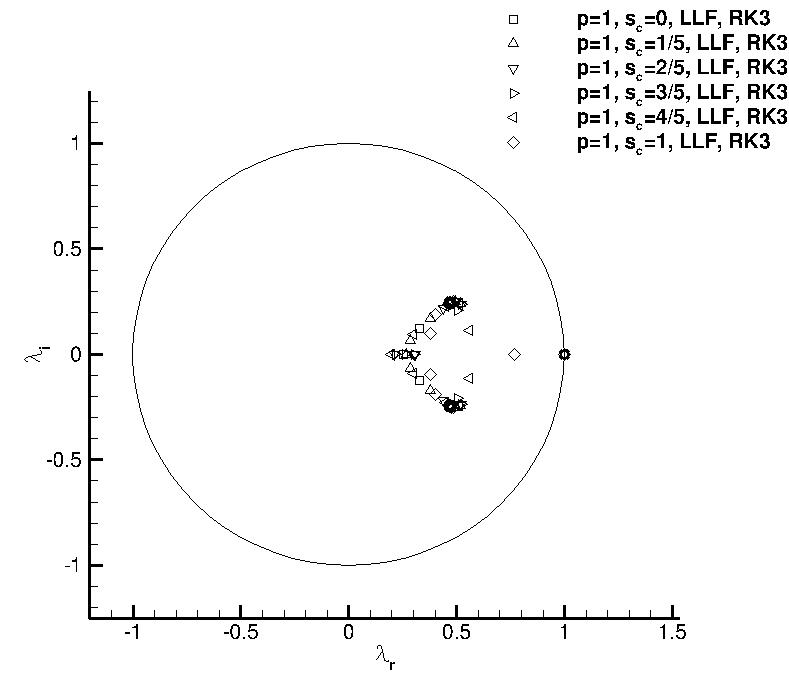 ,width=4.2cm}}\hspace{-0.0cm}
\subfigure{\epsfig{figure=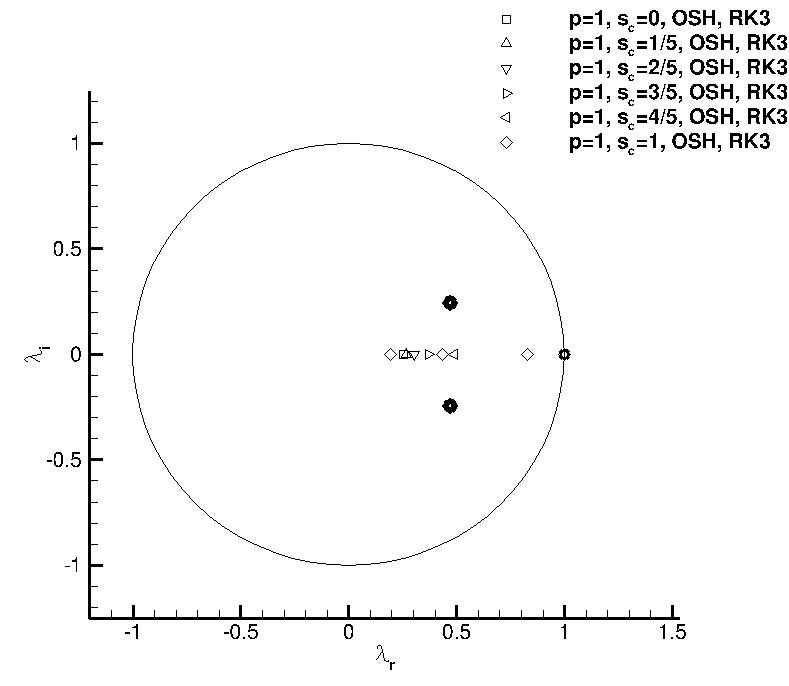 ,width=4.2cm}}\\
\subfigure{\epsfig{figure=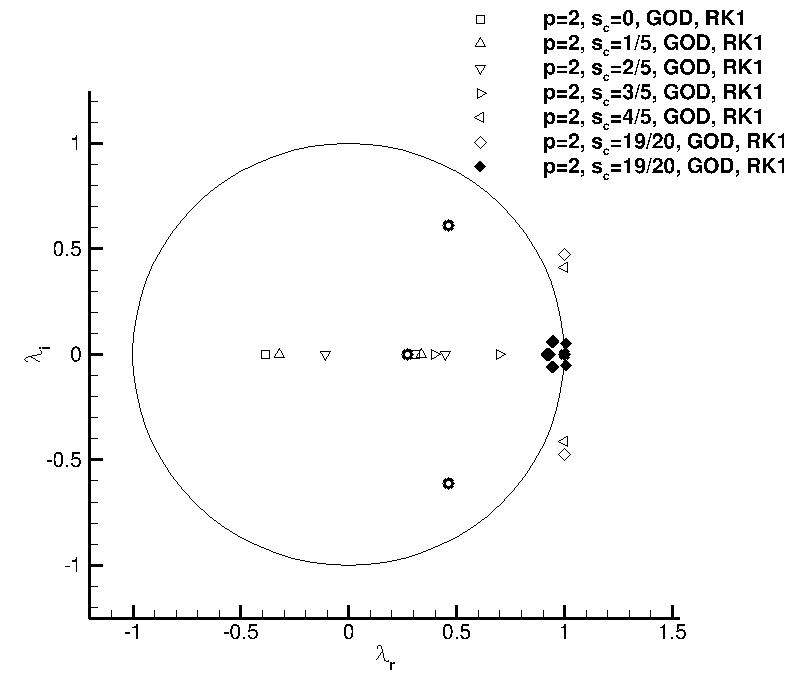 ,width=4.2cm}}\hspace{-0.0cm}
\subfigure{\epsfig{figure=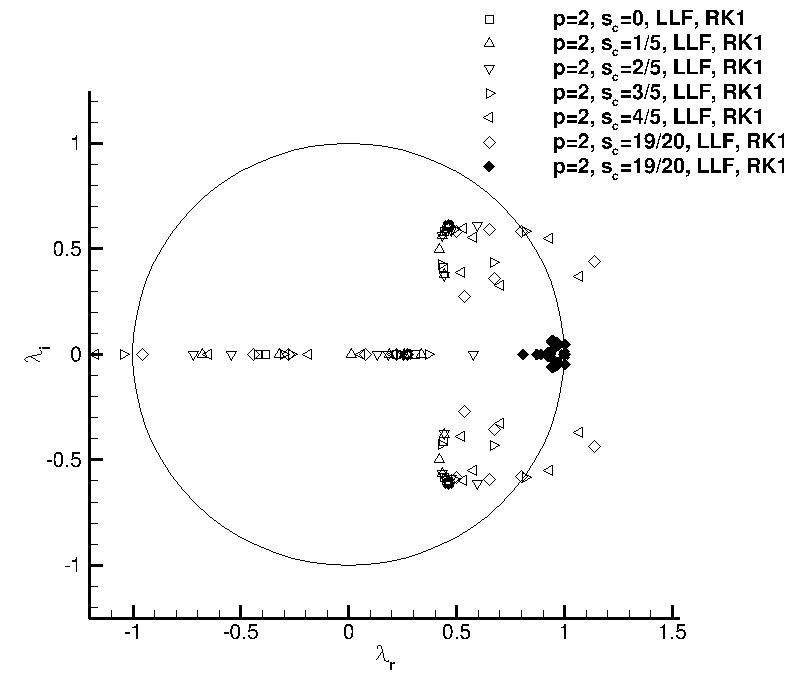 ,width=4.2cm}}\hspace{-0.0cm}
\subfigure{\epsfig{figure=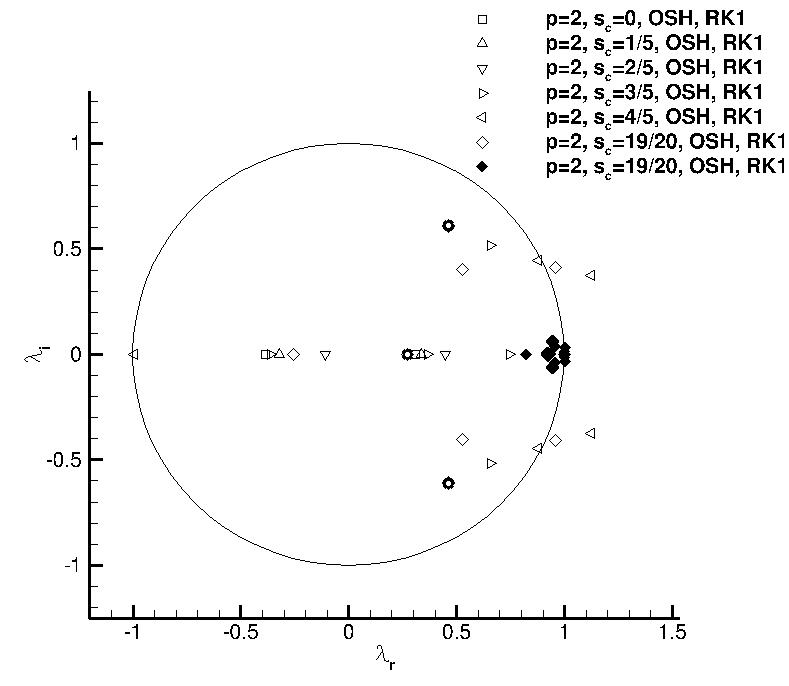 ,width=4.2cm}}\\
\setcounter{subfigure}{0}
\subfigure[Godunov]{\epsfig{figure=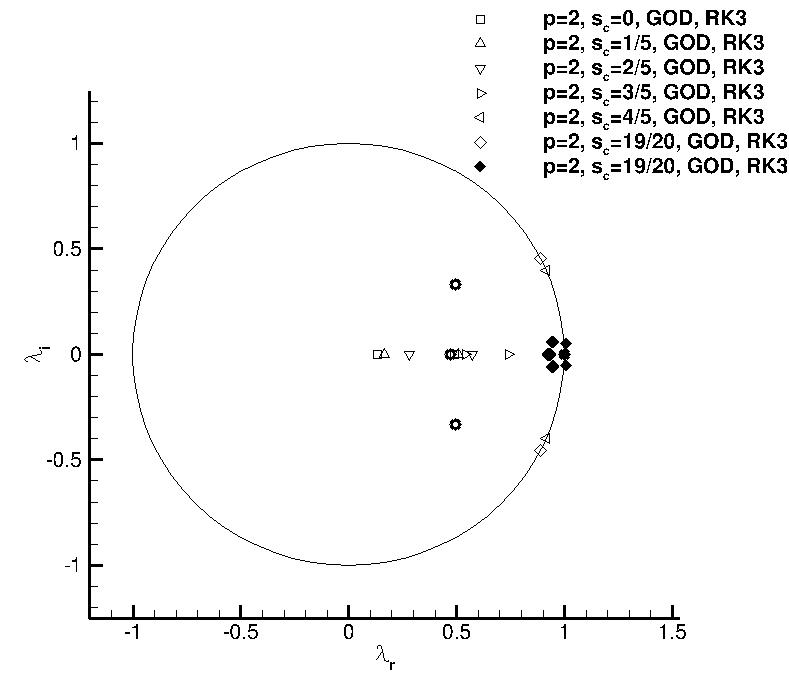 ,width=4.2cm}}\hspace{-0.cm}
\subfigure[LLF]    {\epsfig{figure=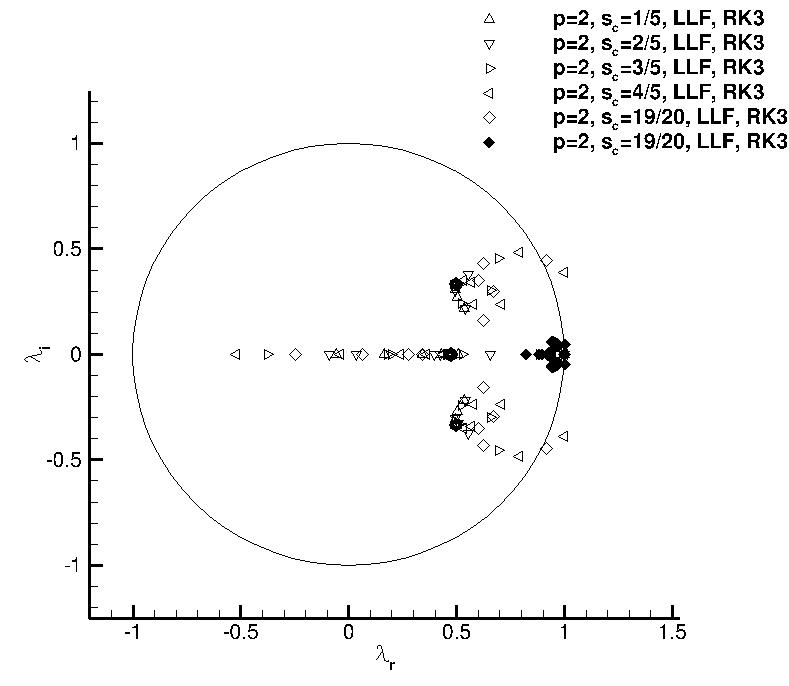 ,width=4.2cm}}\hspace{-0.cm}
\subfigure[Osher]  {\epsfig{figure=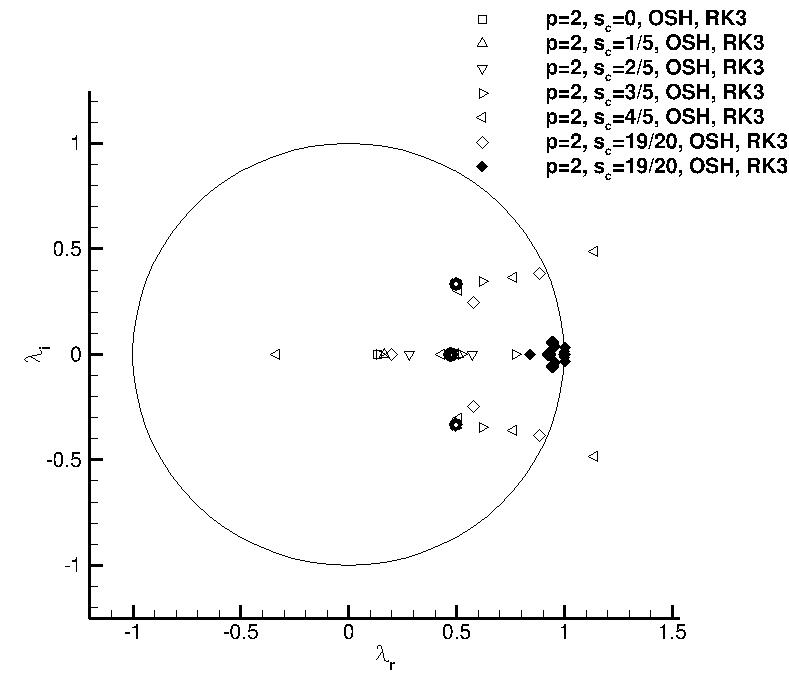 ,width=4.2cm}}
\caption{Eigenspectra of the linearized operator (\ref{eq:linear_operator}) around steady-state solutions of the Burgers equation for different kind of numerical fluxes (GOD: Godunov, LLF, and OSH: Engquist-Osher) and $0\leq s_c\leq1$, using the forward Euler (RK1) and explicit third-order Runge-Kutta (RK3) time integration schemes with $1\leq p\leq2$, $N=20$, and $\lambda=1/(2p+1)$ (open symbols) or $\lambda=0.1/(2p+1)$ (full symbols). The unit circle ${\cal C}(0,1)$ denotes the stability domain.}
\label{fig:ana_stab_vap_p1_p2}
\end{center}
\end{figure}

The third-order approximation, $p=2$, with the forward-Euler scheme exhibits unstable modes when the shock position tends to an interface. Here again, the spectra for every numerical scheme look similar. The complex eigenvalues are concentrated in the right half-plane, while some real eigenvalues are negative and may become unstable. From the stability analysis for the Godunov flux in Table~\ref{tab:eigenvalues}, one may distinguish among the former modes between stable eigenvalues associated to the uniform flow and unstable eigenvalues associated to the shock. For other numerical fluxes, one recovers the eigenvalues in cells $j\neq 0$ concentrated at the same locations $1-\lambda_L(3+\gamma_1)\simeq0.272$ and $1-\tfrac{\lambda_L}{2}(6-\gamma_1\pm\gamma_2i)\simeq0.464\pm0.610i$. Likewise, modes associated to the shock are also very similar. The eigenvalues are concentrated at the unit circle and become unstable as the shock position tends to the interface. We stress that lowering the CFL value (see eigenspectra for $p=2$ and $s_c=19/20$) shows that these modes remain unstable as predicted in Table~\ref{tab:eigenvalues} for $|s_c|>\tfrac{2}{3}$. Figure~\ref{fig:ana_stab_vep_p2} displays the structure of associated eigenvectors for two different shock positions corresponding to unstable modes. In section~\ref{sec:ana_stab_god}, it was shown that these modes only affect the highest DOF and result in a quadratic evolution with support in the shock cell. The structure of the modes for the Engquist-Osher and LLF fluxes is very similar and mainly concentrated in the shock cell. These modes are damped with the Runge-Kutta scheme, but are clustered around the unit circle and remain unstable for large $s_c$ values. These modes are expected to slow-down or prevent the convergence of some computations as observed in Figure~\ref{fig:residuals_p2}b and in the numerical experiments in the preceding section.

\begin{figure}
\begin{center}
\subfigure[$s_c=\tfrac{7}{10}$]{\epsfig{figure=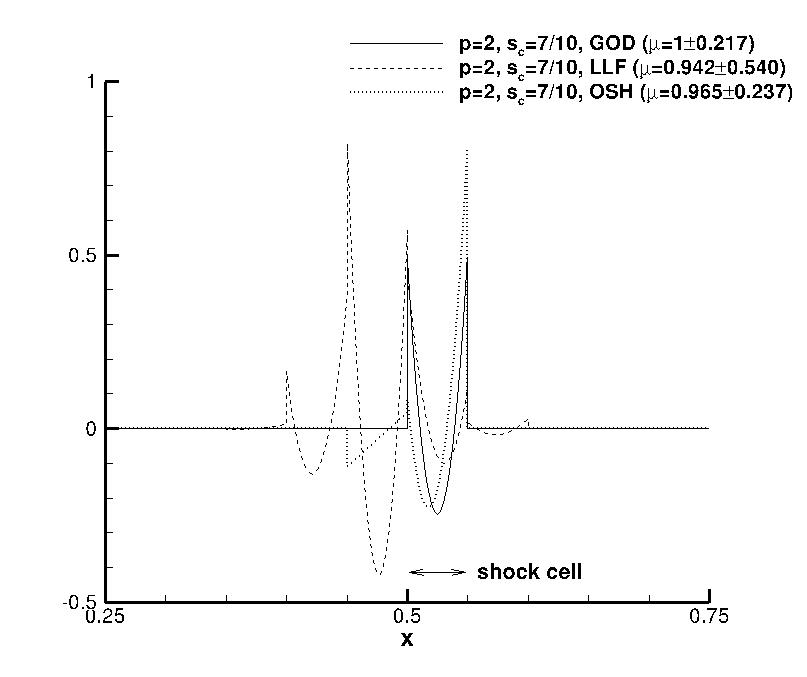 ,width=6cm}}
\subfigure[$s_c=\tfrac{4}{5}$] {\epsfig{figure=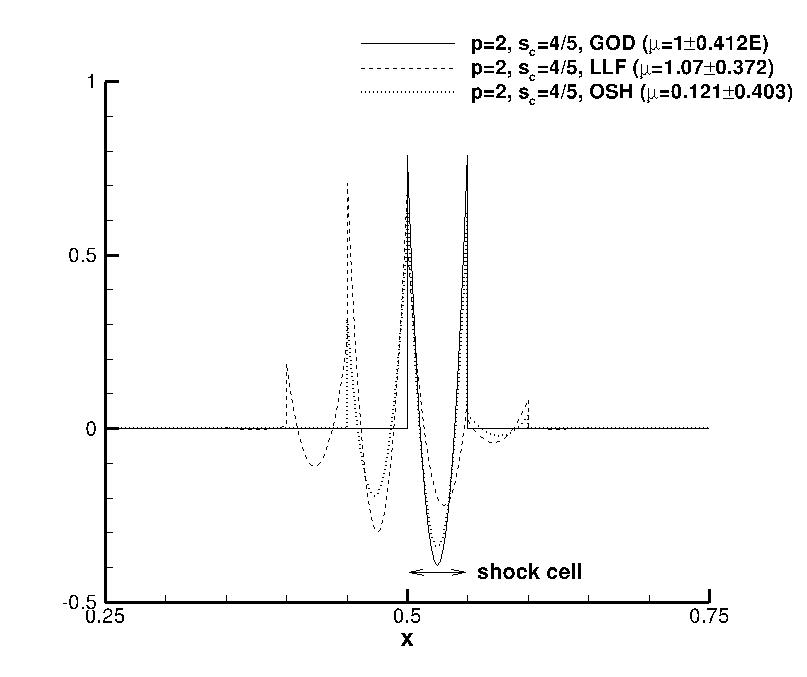 ,width=6cm}}
\caption{Real part of eigenvector associated to most amplified eigenvalues of the linearized operator (\ref{eq:linear_operator}) around steady-state solutions of the Burgers equation for different kind of numerical fluxes (GOD: Godunov, LLF, OSH: Engquist-Osher), using the forward Euler (RK1) and explicit third-order Runge-Kutta (RK3) time integration schemes with $p=2$, $N=20$, and $\lambda=\tfrac{1}{2p+1}$.}
\label{fig:ana_stab_vep_p2}
\end{center}
\end{figure}

\subsection{Stabilization with the spectral vanishing viscosity method}\label{sec:SVM}

The analytical results in Theorem~\ref{th:exp_decay} and section~\ref{sec:ana_stab_god_burgers} and numerical experiments in Table~\ref{tab:exp_decay_LLF} suggest the possibility of removing oscillations in the supersonic and subsonic regions by damping the higest DOF, say $U_j^p$ with $j\neq0$. Likewise, the structure of solutions in Theorem~\ref{th:shock_sol_p0123} and the stability analysis in section~\ref{sec:ana_stab} (see, {\it e.g.}, Table~\ref{tab:eigenvalues_sc_pm1}) support the necessity of acting on a larger range of DOFs, say $(U_j^l)_{m_j\leq l\leq p}$ where $m_j>0$, whenever the solution becomes discontinuous in cell $\kappa_j$. These observations motivate the application of the spectral viscosity method \cite{maday_tadmor89,maday_etal93} to the DG discretization. In this method, one supplements the explicit residuals (\ref{eq:local_residuals}) with an artificial viscosity of the form
\begin{equation*}
 -\varepsilon\int_{\kappa_j} Q\partial_xu_h\partial_xv_hdx,
\end{equation*}
\noindent where $\varepsilon>0$ plays the role of a viscosity coefficient and $Q$ denotes the spectral viscosity operator
\begin{equation*}
 Qv_h = \sum_{l=0}^p Q_j^lV_j^l\phi_j^l, \quad \forall v_h = \sum_{l=0}^p V_j^l\phi_j^l, \;\kappa_j\in\Omega_h,
\end{equation*}
\noindent with $Q_j^l=\exp\big(-(\tfrac{l-p}{m_j-p})^2\big)$ if $l\geq m_j$ and $Q_j^l=0$ otherwise. Here, we use a slightly different implementation from \cite{maday_etal93}. Indeed, using (\ref{eq:matN_kl}) we have
\begin{equation*}
 \int_{\kappa_j} Q\partial_xu_hd_x\phi_j^kdx = \sum_{l=m_j}^{p-1} Q_j^lD_j^l\int_{\kappa_j}\phi_j^ld_x\phi_j^kdx = \sum_{l=m_j}^{p-1} N_{k,l}Q_j^lD_j^l, \quad \forall \kappa_j\in\Omega_h,
\end{equation*}
\noindent where the coefficients $D_j^l$ are defined by $\partial_xu_h=\sum_{l=1}^{p}U_j^ld_x\phi_j^l=\sum_{l=0}^{p-1}D_j^l\phi_j^l$. Using the matrix ${\bf M}$ defined in (\ref{eq:matM_kl}), we get
\begin{equation*}
 D_j^l = \frac{2}{h}\sum_{n=l+1}^pM_{n,l}U_j^n, \quad \forall 0\leq l< p, \kappa_j\in\Omega_h.
\end{equation*}

We propose to apply this method in the context of our numerical experiments by selecting the range of modes where the spectral viscosity is applied according to the local smoothness of the solution: we set $m_j=1$ when the solution is irregular and $m_j=p-1$ otherwise. We apply the shock detection technique from \cite{krivodonova_etal04} to test the local smoothness of the solution.

Figure \ref{fig:SVM_effect} presents the convergence histories and solutions for the LLF numerical flux and third and fourth approximation orders. Results obtained with the above method and a viscosity coefficient $\varepsilon=h/(p+1)$ are compared to that obtained without artificial dissipation. We observe the stabilization of the computations by the numerical viscosity. The amplitude of the oscillations are also lowered and are effectively damped in the uniform region where viscosity is applied only to the highest DOFs. The technique allows the computations to converge to steady-state solutions and to lower oscillations while keeping accuracy in smooth regions of the flow. 

\begin{figure}
\begin{center}
\subfigure{\epsfig{figure=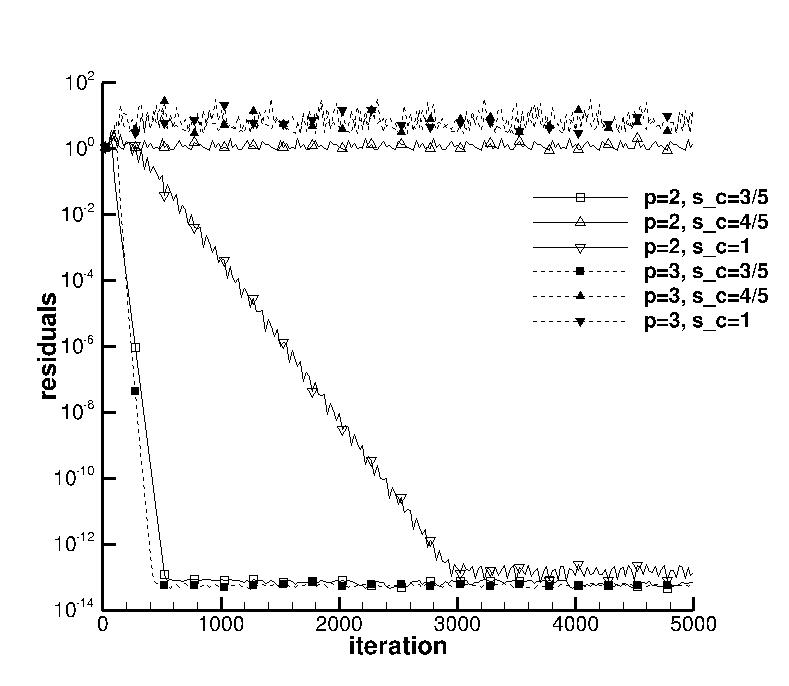 ,width=6cm}}
\subfigure{\epsfig{figure=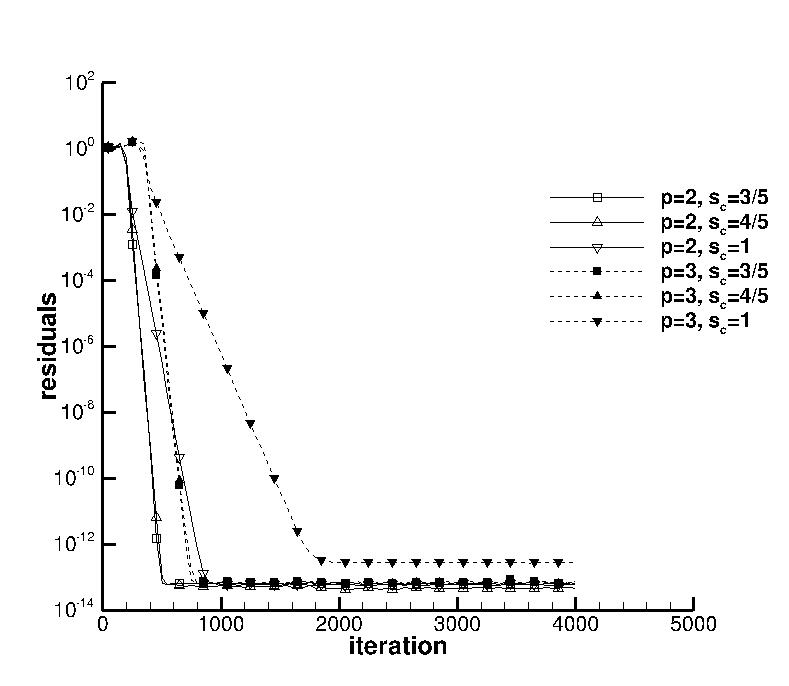 ,width=6cm}}\\
\setcounter{subfigure}{0}
\subfigure[]{\epsfig{figure=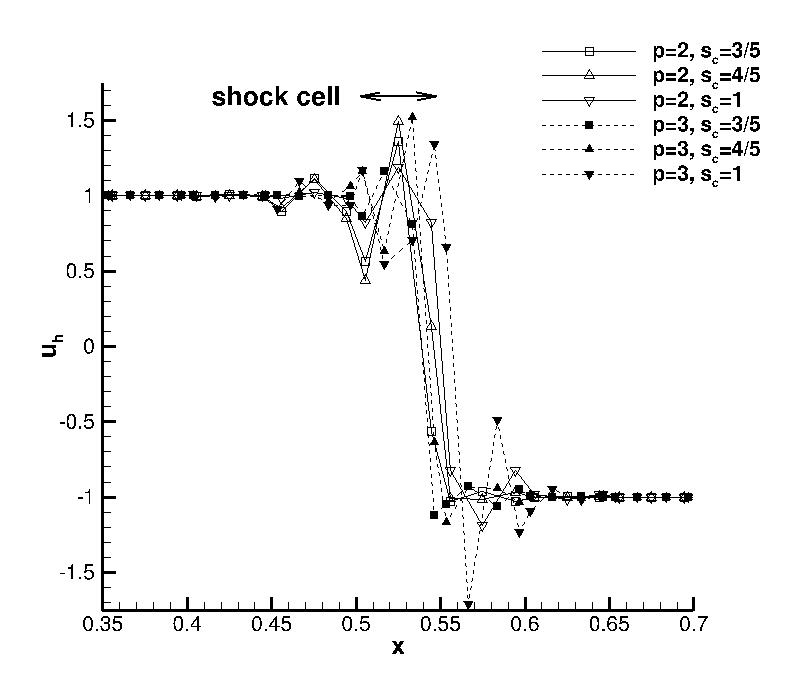 ,width=6cm}}
\subfigure[]{\epsfig{figure=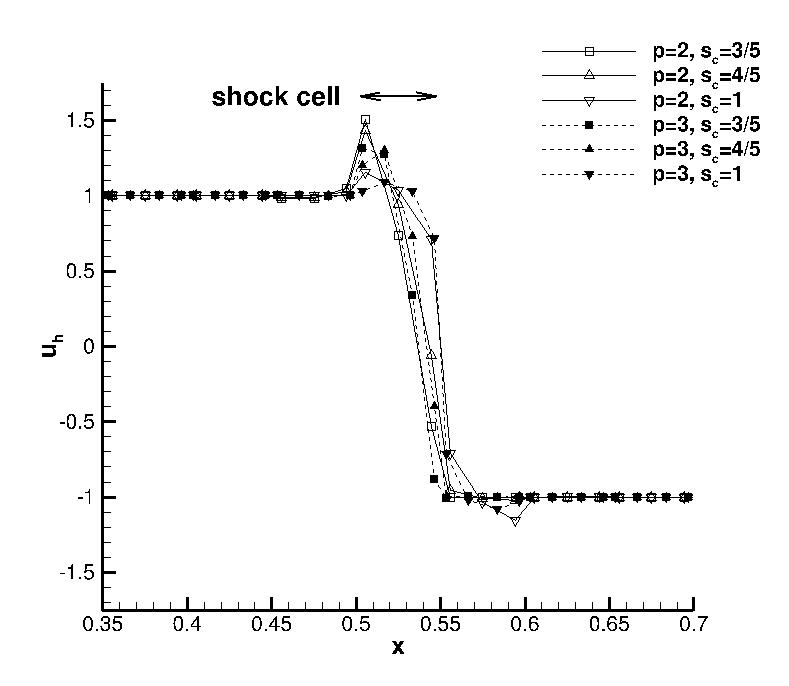 ,width=6cm}}
\caption{Convergence histories (top) and final solutions (bottom) without (left) and with the spectral viscosity method (right) of the Burgers equation using the LLF flux with $p=2$ and $p=3$.}
\label{fig:SVM_effect}
\end{center}
\end{figure}

%
%
\section{Concluding remarks}\label{sec:conclusions}

Discrete shock profiles for a scalar conservation law with a convex flux discretized with a DG method have been analyzed. Using the Godunov numerical flux, we show existence of stationary profiles that are oscillating for polynomial degree $p\geq1$ into one discretization cell only. The oscillations may vanish when the shock is located at an interface of the mesh. A linear stability analysis of the shock profiles show however that these latter solutions may be unstable. Considering the inviscid Burgers equation, these profiles are constructed analytically for $p\leq3$ and are shown to be parametrized by the shock strength and its relative position in the cell.

The extension of this analysis to other numerical fluxes is also investigated. A theoretical analysis shows that oscillations propagate in neighboring cells but decay exponentially fast from the shock position for some class of monotone numerical fluxes. Moreover, numerical experiments indicate that the shock profiles present strong similarities with the profiles obtained with the Godunov flux which may be considered as a relevant model for the analysis of the DG method.

Finally, these results show that, when using a hierarchical functional basis, only the highest DOFs are responsible, at first order, of linear instability and propagation of oscillations in neighboring cells. As an application of this property, the spectral vanishing viscosity method is successfully used to stabilize computations and damp oscillations through a selective action on  DOFs.

One main contribution of this work is the analysis of oscillations and unstable character of the DG method for scalar conservation laws. These results may help to design specific stabilization techniques and future investigations will focus on these methods.

\section*{Acknowledgments}
The author would like to thank Prof. Alain Lerat from Arts et M\'etiers ParisTech and Jean-Luc Akian from Onera for valuable discussions and their constructive comments. 

\appendix\section*{}\label{app:proof_sol_p123} 
In this appendix, we prove the results of Theorem \ref{th:shock_sol_p0123} which gives the solution of the numerical scheme in the cell where the shock is located as a function of the relative shock position, $s_c$, given by (\ref{eq:def_sc}). We note that the solution for the first DOF $U_{0}^0=u_Ls_c$ from (\ref{eq:shock_sol_DOF0}) in Lemma \ref{th:shock_sol_DOF0}. We also give the following result which holds for the Burgers equation and will be used to evaluate the local residuals associated to the equation for the second DOF: 
\begin{equation}\label{eq:intvol_p1}
 \int_{\kappa_j} f(u_h)d_x\phi_j^1 dx = \frac{2}{h}\int_{\kappa_j} \frac{1}{2}\big(\sum_{l=0}^p U_j^l\phi_j^l\big)^2 dx =  \sum_{l=0}^p\frac{(U_j^l)^2}{2l+1},
\end{equation}

\noindent where we have used the orthogonality of the function basis (\ref{eq:ortho_basis}). According to the assumption of Theorem \ref{th:shock_sol_p0123}, the exact shock position is assumed to satisfy $x_{-1/2}<x_c<x_{1/2}$. From Theorems~\ref{th:sol_super_zone_conv_flx} and \ref{th:sol_sub_zone_conv_flx}, it follows that the trace at the left interface satisfies $u_{-1/2}^-=u_L$ and a similar relation holds at the right interface $u_{1/2}^+=u_R$. Further assuming (\ref{eq:sign_condition_conv-flx}), the numerical fluxes at the left and right interfaces of cell $j_c=0$ read $\hat{h}(u_{-1/2}^-,u_{-1/2}^+)=f(u_{-1/2}^-)=f(u_L)$ and $\hat{h}(u_{1/2}^-,u_{1/2}^+)=f(u_{1/2}^+)=f(u_R)$. As a consequence, the equation for the first DOF always reduces to the trivial relation $f(u_L)=f(u_R)$. As suggested in Theorem~\ref{th:shock_sol_p0123}, we adopt the following notation for the DOFs in the cell: $U_{0}^l = u_Lu_l$, for all $0\leq l \leq p$, where $u_0=s_c$ according to (\ref{eq:shock_sol_DOF0}). Finally, we observe that according to (\ref{eq:LR_traces}), assumption (\ref{eq:sign_condition_conv-flx}) may be rewritten as follows
\begin{equation}\label{eq:CNS_sign_cond}
 \frac{u_{-1/2}^+}{u_L}=\sum_{l=0}^p(-1)^lu_l > -1,\quad  \frac{u_{1/2}^-}{u_L}=\sum_{l=0}^pu_l < 1.
\end{equation}

\subsection{solution for $p=1$}

Using (\ref{eq:intvol_p1}), the numerical scheme for the second DOF reads
\begin{equation}\label{eq:SN_shock_p1}
 -u_L^2(u_0^2 + \frac{u_1^2}{3}) + f(u_L)+f(u_R) = u_L^2\Big(-s_c^2-\frac{u_1^2}{3} + 1\Big) = 0,
\end{equation}
\noindent whose solution reads $u_1=\pm\sqrt{3(1-s_c^2)}$. The condition (\ref{eq:CNS_sign_cond}a) reads $s_c-u_1>-1$ and restricts the solution $u_1>0$ over $(\tfrac{1}{2},1)$, while the condition (\ref{eq:CNS_sign_cond}b) reads $s_c+u_1<1$ and restricts the solution $u_1>0$ over $(-1,-\tfrac{1}{2})$. The negative solution is the only one valid over $(-1,1)$ which gives the result (\ref{eq:shock_sol_p1}).

\subsection{solution for $p=2$}

Equations for the second and third DOFs read
\begin{subeqnarray*}
 -\int_{\kappa}f(u_h)d_x\phi_{0}^1dx + f(u_L)+f(u_R) &=& 0,\\
 -\int_{\kappa}f(u_h)d_x\phi_{0}^2dx + f(u_L)-f(u_R) &=& 0,
\end{subeqnarray*}
\noindent and give
\begin{subeqnarray}\label{eq:SN_shock_p2}
 -s_c^2 - \frac{u_1^2}{3} - \frac{u_2^2}{5} + 1 &=& 0,\\
 -2u_1\big(s_c+\frac{2}{5}u_2\big) &=& 0.
\end{subeqnarray}

The second equation has a trivial solution $u_1=0$ which gives $u_2=\pm\sqrt{5(1-s_c^2)}$. The validity of these solutions is imposed by (\ref{eq:CNS_sign_cond}) which reduces to $-1\leq s\pm\sqrt{5(1-s_c^2)}<1$ because the left and right traces are identical. Therefore, the solution $u_2=\sqrt{5(1-s_c^2)}$ is valid when $s<-\tfrac{2}{3}$ and the solution $u_2=-\sqrt{5(1-s_c^2)}$ is valid when $s> \tfrac{2}{3}$ which correspond to solutions (\ref{eq:shock_sol_p2}a,c).

The second solution of (\ref{eq:SN_shock_p2}b) reads $u_2=-\tfrac{5}{2}s_c$ and then (\ref{eq:SN_shock_p2}a) leads to $u_1=\pm\sqrt{3(1-9s_c^2/4)}$ which are valid only if $|s_c|\leq\tfrac{2}{3}$. Again, conditions (\ref{eq:CNS_sign_cond}) impose the negative solution $u_1=-\sqrt{3(1-9s_c^2/4)}$.

\subsection{solution for $p=3$}

The numerical scheme for the second, third and fourth DOFs give
%
\begin{subeqnarray}\label{eq:SN_shock_p3}
 -s_c^2 - \frac{u_1^2}{3} - \frac{u_2^2}{5} - \frac{u_3^2}{7} + 1 &=& 0,\\
 -2u_1\big(s_c+\frac{2}{5}u_2\big) - \frac{18}{35}u_2u_3 &=& 0,\\
 -s_c^2 - u_1^2 -\frac{17}{35}u_2^2 - \frac{u_3^2}{3} -2s_cu_2 - \frac{6}{7}u_1u_3 +1 &=& 0.
\end{subeqnarray}

We first observe that for a solution $(u_1,u_2,u_3)$ to (\ref{eq:SN_shock_p3}), then $(-u_1,u_2,-u_3)$ is also solution. Therefore, we only look for solutions with $u_1<0$ and deduce the other ones by symmetry. Then, we also note that the choice $u_2=0$ imposes $u_1=0$ through (\ref{eq:SN_shock_p3}b) and then (\ref{eq:SN_shock_p3}a,c) reduce to $1-s_c^2-\tfrac{u_3^2}{7}=0$ and $1-s_c^2-\tfrac{u_3^2}{3}=0$ which is possible only if $u_3=0$ and $s_c=\pm1$ but is excluded by the strict inequalities in (\ref{eq:CNS_sign_cond}). In the following, we thus consider $u_2\neq0$. Equation (\ref{eq:SN_shock_p3}b) induces $u_3=-\tfrac{7u_1}{9u_2}(5s_c+2u_2)$ and subtracting $7\times$(\ref{eq:SN_shock_p3}a) and $3\times$(\ref{eq:SN_shock_p3}c), multiplying by $u_2$, one obtains
\begin{equation*}
 -10\big(\frac{u_2}{3}+s_c\big)u_1^2 + \frac{2}{35}u_2^3+6s_cu_2^2+4(1-s_c^2)u_2 = 0,
\end{equation*}

\noindent whose solution for $u_1$ reads
\begin{equation}\label{eq:sol_shock_p3_u1}
 u_1^2 = \frac{3u_2}{5(u_2+3s_c)}\Big(\frac{u_2^2}{35}+3s_cu_2+2(1-s_c^2)\Big),
\end{equation}
\noindent and one deduces the solution for $u_3$ from 
\begin{equation}\label{eq:sol_shock_p3_u3}
 u_3=-\tfrac{7u_1}{9u_2}(5s_c+2u_2)
\end{equation}
\noindent which gives
\begin{equation*}
 u_3^2 = \frac{49(5s_c+2u_2)^2}{135u_2(u_2+3s_c)}\Big(\frac{u_2^2}{35}+3s_cu_2+2(1-s_c^2)\Big).
\end{equation*}

Substituting the last result and (\ref{eq:sol_shock_p3_u1}) into (\ref{eq:SN_shock_p3}a), one obtains the equation for $u_2$ only:
\begin{equation}\label{eq:sol_shock_p3_u2}
 2(7s_c+u_2)\Big(4u_2^3+7s_cu_2^2+14s_c^2u_2-\frac{7}{2}u_2+7s_c(1-s_c^2) \Big) = 0.
\end{equation}

Taking symmetries into account, the system (\ref{eq:SN_shock_p3}) has at most height real solutions. The first root of equation (\ref{eq:sol_shock_p3_u2}) $u_2=-7s_c$ leads to the first solution (\ref{eq:shock_sol_p3}a):
\begin{equation*}
 u_1 = -u_3 =\pm \frac{1}{5}\sqrt{\frac{21}{2}}\sqrt{5-54s_c^2},
\end{equation*}
%
\noindent if $|s_c|\leq\sqrt{\tfrac{5}{54}}$ from which we again retain the solution $u_1<0$. Assumption (\ref{eq:CNS_sign_cond}) requires $u_{j-1/2}^+=-7s_c>-1-s_c$ and $u_{j+1/2}^-=-7s_c<1-s_c$ which is satisfied if and only if $-\tfrac{1}{6}<s_c<\tfrac{1}{6}$. The three other roots of equation (\ref{eq:SN_shock_p3}) read 

\begin{subeqnarray}\label{eq:sol_shock_p3_u2b}
 u_2 &=& \frac{1}{12}\Big[-7s_{c}-\frac{7(17s_{c}^2-6)}{\Delta_3^{1/3}}+\Delta_3^{1/3}\Big],\\
 u_2 &=& \frac{1}{12}\Big[-7s_{c}-\frac{7(17s_{c}^2-6)}{\ol{\Delta}_3^{1/3}}-\ol{\Delta}_3^{1/3}\Big],\\
 u_2 &=& \frac{1}{12}\Big[-7s_{c}-\frac{7(1-i\sqrt{3})(17s_{c}^2-6)}{2(1+i\sqrt{3})\Delta_3^{1/3}}-\frac{(1+i\sqrt{3})}{2}\Delta_3^{1/3}\Big],
\end{subeqnarray}
\noindent where $\Delta_3=7\Delta_2$, $\ol{\Delta}_3=-7\Delta_2$ and $\Delta_2$ has been defined in equation (\ref{eq:prop_def_Deltai}a). The last root (\ref{eq:sol_shock_p3_u2b}c) is complex. From (\ref{eq:prop_def_Deltai}), it follows that the first root (\ref{eq:sol_shock_p3_u2b}a) exists only when $\Delta_1^2=7776s_{c}^6-10008s_{c}^4+3359s_{c}^2-56\geq0$ and $\Delta_2>0$. The former condition has only two real roots $s_c=\pm s_1$ and is satisfied providing that $|s_c|\geq s_1$ defined by
\begin{equation*}
  s_1 = \frac{1}{18\sqrt{2}}\Big[278-\frac{2^{1/3}8411+(1373963-6687\sqrt{15603})^{2/3}}{2^{-1/3}(1373963-6687\sqrt{15603})^{1/3}}\Big]^{\tfrac{1}{2}},
\end{equation*}
\noindent while the latter condition has one root $s_c=-\sqrt{6/17}$ and requires the additional conditions $s_c<-\sqrt{6/17}$ or $s_c>s_1$. After substituting the former solution (\ref{eq:sol_shock_p3_u2b}a) into (\ref{eq:sol_shock_p3_u1}) and (\ref{eq:sol_shock_p3_u3}) one obtains the solution (\ref{eq:shock_sol_p3}c) for $u_1$ and $u_3$ when $\Delta_2>0$. It may be checked that (\ref{eq:CNS_sign_cond}b) is always satisfied over the range $[-1,-\sqrt{6/17}]\cup[s_1,1)$, while the condition (\ref{eq:CNS_sign_cond}a) is satisfied only over the range $(-1,-\sqrt{6/17}]\cup(1/6,1]$. The solutions $u_1$ and $u_3$ with opposite signs cannot satisfy those conditions for any $s$ value.

Likewise, the second solution (\ref{eq:sol_shock_p3_u2b}b) holds when $\ol{\Delta}_2=-\Delta_2>0$, that is when $-\sqrt{6/17}<s_c<s_1$ and leads to solution (\ref{eq:shock_sol_p3}b). This solution satisfies (\ref{eq:CNS_sign_cond}) if $-\sqrt{6/17}<s_c<-1/6$.

%
%
%

\end{document}